\documentclass[10pt,reqno]{amsart}

\usepackage[a4paper,left=35mm,right=35mm,top=30mm,bottom=30mm,marginpar=25mm]{geometry}
\usepackage{comment}

\usepackage{amsmath}
\usepackage{amssymb}
\usepackage{amsthm}
\usepackage[latin1]{inputenc}
\usepackage{eurosym}
\usepackage[dvips]{graphics}
\usepackage{graphicx}

\usepackage{algorithm}
\usepackage[noend]{algpseudocode}

\usepackage{float}
\usepackage[caption = false]{subfig}

\usepackage{epsfig}
\usepackage{hyperref}
\usepackage{dsfont}
\usepackage[nocompress, space]{cite}

\RequirePackage[normalem]{ulem} 
\RequirePackage{color}\definecolor{RED}{rgb}{1,0,0}\definecolor{BLUE}{rgb}{0,0,1} 

\usepackage[displaymath,mathlines]{lineno}

\allowdisplaybreaks

\usepackage{ifthen}
\usepackage[invisible]{MLatex}

\makeindex

\newcommand{\sgn}{\operatorname{sgn}}

\newcommand{\Rr}{{\mathbb{R}}}

\newcommand{\Nn}{{\mathbb{N}}}

\def\leq{\leqslant}
\def\geq{\geqslant}

\numberwithin{equation}{section}

\newtheoremstyle{thmlemcorr}{10pt}{10pt}{\itshape}{}{\bfseries}{.}{10pt}{{\thmname{#1}\thmnumber{
#2}\thmnote{ (#3)}}}
\newtheoremstyle{thmlemcorr*}{10pt}{10pt}{\itshape}{}{\bfseries}{.}\newline{{\thmname{#1}\thmnumber{
#2}\thmnote{ (#3)}}}
\newtheoremstyle{defi}{10pt}{10pt}{\itshape}{}{\bfseries}{.}{10pt}{{\thmname{#1}\thmnumber{
#2}\thmnote{ (#3)}}}
\newtheoremstyle{remexample}{10pt}{10pt}{}{}{\bfseries}{.}{10pt}{{\thmname{#1}\thmnumber{
#2}\thmnote{ (#3)}}}
\newtheoremstyle{ass}{10pt}{10pt}{}{}{\bfseries}{.}{10pt}{{\thmname{#1}\thmnumber{
A#2}\thmnote{ (#3)}}}

\theoremstyle{thmlemcorr}
\newtheorem{theorem}{Theorem}
\numberwithin{theorem}{section}

\theoremstyle{thmlemcorr*}
\newtheorem{theorem*}{Theorem}
\newtheorem{lemma*}[theorem]{Lemma}
\newtheorem{corollary*}[theorem]{Corollary}
\newtheorem{proposition*}[theorem]{Proposition}
\newtheorem{problem*}[theorem]{Problem}
\newtheorem{conjecture*}[theorem]{Conjecture}

\theoremstyle{defi}

\newtheorem{hyp}{Assumption}

\theoremstyle{remexample}
\newtheorem{remark}[theorem]{Remark}

\newtheorem{teo}[theorem]{Theorem}
\newtheorem{lem}[theorem]{Lemma}
\newtheorem{pro}[theorem]{Proposition}
\newtheorem{cor}[theorem]{Corollary}

\theoremstyle{ass}

\begin{document}
	
	\begin{PRE}
		Needs["MonteCarloOptimizer`", "/Users/gomesd/Dropbox (Personal)/Work/trab/MonteCarloSwarm/Guide/MonteCarloOptimizer.wl"]	
	\end{PRE}

\title[Derivative-Free Global Minimization]{Derivative-Free Global Minimization in One Dimension: Relaxation, Monte Carlo, and  Sampling}

\author{Alexandra A. Gomes}
\address[A. A. Gomes]{
	CEMSE Division, King Abdullah University of Science and Technology (KAUST), CEMSE Division , Thuwal 23955-6900. Saudi Arabia.}
\email{alexandra.gomes@kaust.edu.sa}

\author{Diogo A. Gomes}
\address[D. A. Gomes]{
        CEMSE Division, King Abdullah University of Science and Technology (KAUST), CEMSE Division , Thuwal 23955-6900. Saudi Arabia.}
\email{diogo.gomes@kaust.edu.sa}

%
\thanks{
The research reported in this publication was supported by funding from King Abdullah University of Science and Technology (KAUST).
D. Gomes was supported by King Abdullah University of Science and Technology (KAUST) baseline funds and KAUST OSR-CRG2021-4674.
}

\begin{abstract}
We introduce a derivative-free global optimization algorithm that efficiently computes minima for
 various classes of one-dimensional functions, including non-convex, and non-smooth functions. 
This algorithm numerically approximates the gradient flow of a relaxed functional,
 integrating strategies  such as  Monte Carlos methods, rejection sampling, and adaptive techniques. 
 These strategies enhance performance in solving a diverse range of optimization problems while significantly reducing the number of required function evaluations compared to established methods. We present a proof of the convergence of the algorithm and illustrate its performance by comprehensive benchmarking.
 The proposed algorithm offers a substantial potential for real-world models. It is particularly advantageous in situations requiring computationally intensive objective function evaluations.
\end{abstract}

\maketitle


\section{Introduction}

Often, real-world models lead to complex optimization problems with challenging objective functions. These functions may have unknown or difficult-to-compute formulas, hard-to-determine derivatives, or could be non-differentiable or discontinuous. As a result, there is a demand for algorithms for approximating a global minimizer
using a limited number of objective function evaluations. Such algorithms are crucial when evaluating the objective function is computationally expensive and time-consuming. For recent accounts on derivative-free optimization algorithms, see \cite{conn2009introduction}, \cite{MR3070154}, or \cite{larson2019derivative}, and Section \ref{prior} below.

The main contribution of this paper is a new  derivative-free global minimization algorithm capable of achieving high success rates for both convex and non-convex functions in $\mathbb{R}$ with few objective function evaluations. 
Our algorithm integrates three main ideas: the relaxation of the optimization problem, the use of Monte Carlo methods and rejection sampling, and careful error control to devise a time-stepping strategy. 
We rigorously prove the algorithm's convergence and demonstrate its performance through benchmarking against multiple algorithms.

\subsection{Relaxation and gradient flows}

Our algorithm utilizes the gradient flow of a relaxed functional. Here, 
we introduce and motivate this relaxed functional approach.
Let $f: \mathbb{R} \rightarrow \mathbb{R}$ be an objective function with a global minimum $\bar{x}$, which may not be unique. Assume that $f$ is continuous and satisfies the polynomial growth conditions stated in \eqref{pgc} and Assumption \ref{a2} in Section \ref{mathprosec}.
Let $\Gamma_{\mu, \sigma}(x)$ be the Gaussian probability distribution function with mean $\mu \in \mathbb{R}$ and standard deviation $\sigma \in \mathbb{R}^+$,
\[
\Gamma_{\mu, \sigma}(x)=\frac{e^{-(x-\mu )^2/(2 \sigma^2)}}{\sqrt{2\pi}  \sigma}.
\]
We consider the relaxed functional
\begin{equation}
	\label{Fdef}
	F(\mu, \sigma) = \int_{\mathbb{R}} f(x) \Gamma_{\mu, \sigma}(x) dx.
\end{equation}
As shown in Proposition \ref{relaxpro}, we have
\begin{equation}
	\label{relaxid}
	\inf_{\mu, \sigma} F = \min_x f.
\end{equation}
Thus,  the problem of minimizing $f$ 
can be transformed into an equivalent problem of minimizing 
 $F$, albeit at the cost of doubling the number of variables.  
 As discussed in Section \ref{mathprosec}, one advantage of this method is that calculating the gradient of $F$ does not require evaluating the derivative of $f$. Hence, minimizing $F$ using gradient flow can be achieved without computing derivatives of $f$.

Integration with respect to the Gaussian smooths out local minima, preserving the global features of f while reducing high-frequency oscillations. This dampening can also be attributed to 
$F$  satisfying the modified heat equation

\begin{equation}
\label{heat}
\frac {\partial F}{\partial \sigma}=\sigma \frac {\partial^2 F}{\partial \mu^2},
\end{equation}
as shown in Proposition \ref{heatpro}.
Figure \ref{relfig} illustrates a smoothing behavior for the objective function $f(x)=x^2-\cos 10 x$. There,  we see that as $\sigma$ increases,
$F$ becomes convex in $\mu$. The global minimum of $f$, $\bar x=0$, corresponds to the infimum of $F$ at $(\mu,\sigma)=(0,0)$, as expected.
\begin{figure}[ht]
	\centering
	\includegraphics[width=0.35\textwidth]{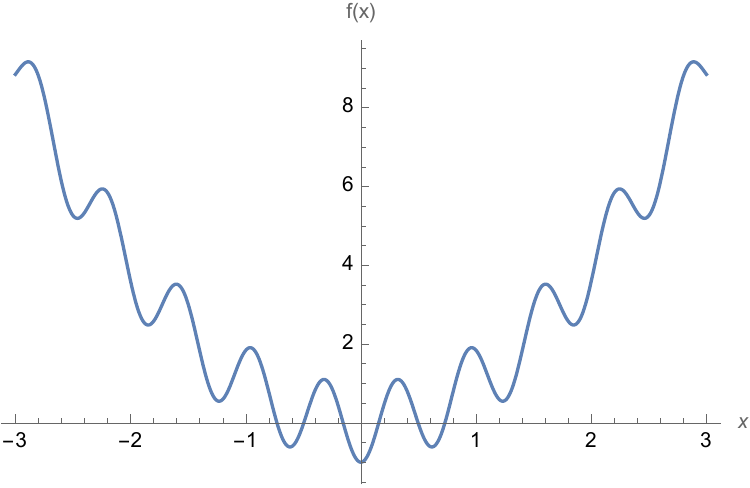}
	\includegraphics[width=0.35\textwidth]{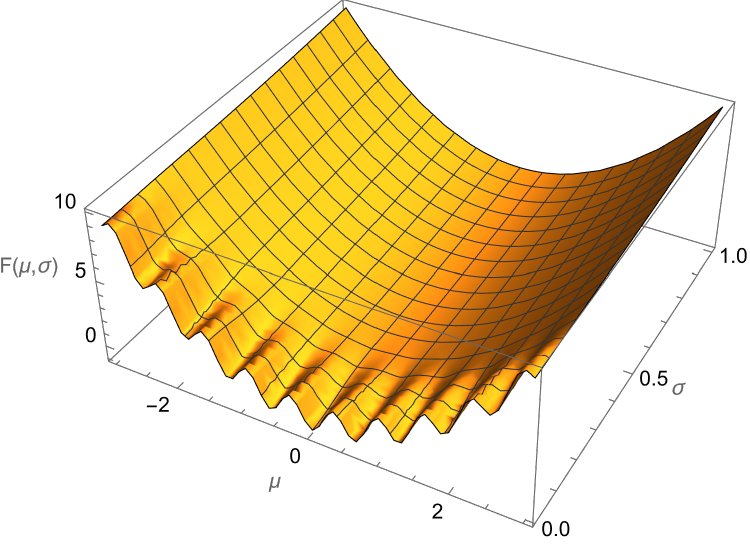}
	\caption{Original objective function $f(x)=x^2-\cos 10 x$ (left) and the corresponding
	relaxed objective function $F$ (right).}
	\label{relfig}
\end{figure}

To minimize $F$, we consider the gradient flow of $F$ in $(\mu, \sigma)$; that is,
\begin{equation}
	\label{gf}
	\begin{bmatrix}
		\dot \mu\\
		\dot \sigma
	\end{bmatrix}=-\nabla F(\mu, \sigma).
\end{equation}
This gradient flow decreases $F$. More precisely, 
\begin{equation}
	\label{Fdec1}
	\frac{d}{dt} F(\mu, \sigma)=-\left(\frac {\partial F}{\partial \mu}\right)^2-\left(\frac {\partial F}{\partial \sigma}\right)^2\leq 0. 
\end{equation}
In addition, \eqref{gf}
has several desirable properties
discussed in detail in Section \ref{mathprosec}: if $f$ is strictly convex, 
$\sigma\to 0$ and $\mu$ converges to the global minimizer of $f$ as $t\to \infty$. If $f$ is
non-convex, any local maxima $\hat x$ of $f$, corresponds to a point $(\mu, \sigma)=(\hat x,0)$
which is repellent for the gradient flow. 
Figure \ref{stream} illustrates this behavior. It displays the flow lines corresponding to the function depicted in Figure \ref{relfig}. We observe that initial conditions with sufficiently large $\sigma$ are drawn towards the global minimizer, while the local maxima of $f$ repel the flow for small $\sigma$, as anticipated.

\begin{figure}[ht]
	\centering
	\includegraphics[width=0.35\textwidth]{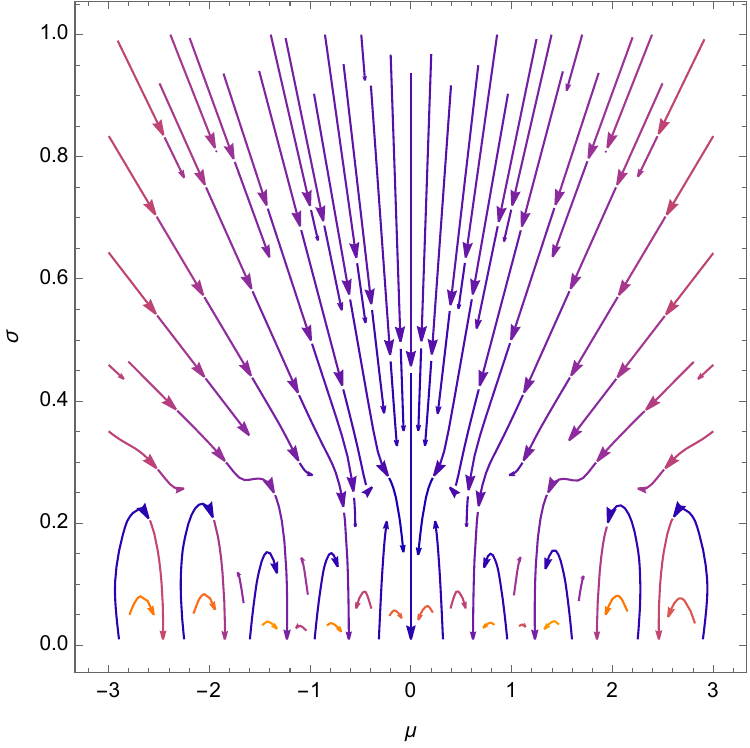}
	\caption{Gradient flow stream plot for the function $F$ depicted in Figure \ref{relfig}}
	\label{stream}
\end{figure}

\subsection{Gradient flow approximation}

Calculating the closed-form expression for \( \nabla F \) can be complex or even infeasible. 
Thus, we need to calculate \( \nabla F \) using numerical methods, such as Monte Carlo integration. 
We can then approximate the gradient flow, \eqref{gf}, using Euler's method. 
However, this approach poses two challenges. First, the Monte Carlo integration error is \( O(\frac{1}{\sqrt{n}}) \), where \( n \) is the sample size. This makes accurate estimation of \( \nabla F \) difficult with limited data points.
Second, the time step in Euler's method is limited by the Lipschitz constant $L$ of $\nabla F$; for stability, we must use a time step not exceeding $2/L$. $L$ may be difficult to estimate if the numerical computation of $\nabla F$ 
has significant errors. 
Although convergence improvement strategies like variance reduction techniques exist and alternative integration methods offer improved stability and convergence properties, we opt for a different approach. Our approach uses error estimates
to avoid stability issues while keeping the number of function evaluations small. 

We start by noting that if $f$ is a quadratic function, i.e., $f(x) = q(x) = a + bx + cx^2$, where $a, b, c \in \Rr$, the gradient $\nabla F$ can be calculated exactly (Proposition \ref{gradflowcomp}). Moreover, the gradient flow,
\begin{equation}
\label{qgf}
\begin{bmatrix}
	\dot \mu\\
	\dot \sigma
\end{bmatrix}
=
\begin{bmatrix}
	-b-2 c\mu\\
	-2 c\sigma
\end{bmatrix},  
\end{equation}
is an uncoupled linear equation with an explicit solution (Proposition \ref{linsol}). By substituting $f$ with a quadratic function, we can compute the exact solution of the previous gradient flow for arbitrary time, thus avoiding stability limitations. However, we must address the error introduced by replacing $f$ with a quadratic function (Section \ref{gfsec}).
At iteration $j$, a natural choice is to replace $f$ with a quadratic function $q_j = a_j + b_jx + c_jx^2$ that minimizes the error
\begin{equation}
\label{msexact}
\int_{\Rr} |q_j-f|^2 \Gamma_{\mu_j, \sigma_j}(x) dx.
\end{equation}
The first-order optimality condition for the above variational problem states that for any second-order polynomial $p(x)$
\begin{equation}
\label{nocexact}
\int_{\Rr} p (q_j-f)  \Gamma_{\mu_j, \sigma_j}(x) dx=0.
\end{equation}
Let $F^{q_j}$ be given by \eqref{Fdef} with $f$ replaced by $q_j$. The computation of $\nabla F$ and $\nabla F^{q_j}$ involves only integrals of the form $\int_\Rr p(x)f(x)  \Gamma_{\mu_j, \sigma_j}(x)dx$ or $\int_\Rr p(x) q_j(x)  \Gamma_{\mu_j, \sigma_j}(x)dx$ for certain second-order polynomials $p$ (see Proposition \ref{gradflowcomp}). 
Because \ref{nocexact} implies \[\int_\Rr p f  \Gamma_{\mu_j, \sigma_j}dx=\int_\Rr p q_j  \Gamma_{\mu_j, \sigma_j}dx,\] we have
\begin{equation}
\label{qid}
\nabla F(\mu_j, \sigma_j)=\nabla F^{q_j}(\mu_j, \sigma_j).
\end{equation}
Thus, at each iteration, we propose to solve
\[
\begin{bmatrix}
	\dot \mu\\
	\dot \sigma
\end{bmatrix}
=
\begin{bmatrix}
	-b_j-2 c_j\mu\\
	-2 c_j\sigma
\end{bmatrix},  
\]
with initial condition $(\mu(0), \sigma(0))=(\mu_j, \sigma_j)$.
As $\mu$ and $\sigma$ change due to the gradient flow, 
$\nabla F(\mu, \sigma)$ is no longer identical to $\nabla F^{q_j}(\mu, \sigma)$ but, by continuity, remains close for some time.
Therefore, we derive estimates for the maximal time step in which the quadratic approximation remains valid. It is worth noting that the identity \eqref{qid} is exact even if $f$ is not well-approximated by a quadratic, in which case the error estimates yield a small time step.

Our algorithm uses Monte Carlo integration to approximate \eqref{msexact} by a finite sum. Then, we apply least squares to determine $q_j$. Finally, the error estimates involve integral expressions computed via Monte Carlo integration. See Sections \ref{sampsec} and \ref{gfsec}.

\subsection{Prior work}
\label{prior}

Our primary focus is on derivative-free global optimizers, particularly evolutionary methods. Derivative-free algorithms are especially valuable for optimizing functions given by a black-box function, where no exact derivatives are available. For a comprehensive bibliographic overview of global optimization methods, including historical perspectives and recent advancements, we recommend referring to \cite{LOCATELLI2021100012}.

Our algorithm uses random sampling to explore the feasible space.  However, 
	the sampling distribution evolves according to a gradient flow. 
	Nelder-Mead \cite{MR3363409}, multilevel coordinate search 
\cite{MR1707795}, or pattern search \cite{HookeJeeves}, \cite{MR321541}, rely on a direct search of the feasible space combining ideas from optimization with heuristic procedures. The convergence analysis of many of these algorithms is reasonably well understood, as well as some of their limitations; see, for example, \cite{lagrias1996convergence}, \cite{MR1430554}, \cite{MR2048154}, and  
\cite{mckinnon1998convergence}.  To better explore the state space, 
a random search approach was introduced in \cite{rastrigin1963convergence}, and multiple improvements were proposed in the literature, for example, the Luus-Jaakola algorithm \cite{LuusJaakola}. 
Another important random search contribution is 
		\cite{matyas1965random}
	and recent improvements  \cite{ghadimi2013stochastic}, \cite{duchi2015optimal},
and \cite{nesterov2017random}, sometimes called
	zeroth order optimization.
Driven by machine learning advances, random search is popular in black-box hyperparameter optimization; see  
		\cite{JMLR:v13:bergstra12a} and \cite{YANG2020295}.

Evolutionary algorithms allow for a broad exploration of the feasible space for highly non-convex or expensive black-box objective functions. Two well-known population-based algorithms in the evolutionary family are the particle swarm algorithm, introduced in \cite{488968} and \cite{699146}, and 
the differential evolution from \cite{storn1997differential}; see also 
		\cite{price2006differential}, \cite{locatelli2013global}, \cite{DAS20161} and \cite{DELSER2019220}.
These algorithms maintain a population of candidate solutions that are combined to create new candidates using suitable heuristics, such as combination and mutation. For both particle swarm and differential evolution, there are several convergence  results, see \cite{LNV20151}, \cite{LNV20152}, \cite{Ghosh2012107} and \cite{LocVasile2015}. 
Another development in differential evolution strategies is the combination with local optimization algorithms. These memetic algorithms, a synergistic relation between local and global optimization strategies, are found in Schoen's work, \cite{schoen2021efficient} and \cite{MANSUETO2021107849}, for example.
Finally, population-based methods are also combined with decomposition-based strategies in global optimization, see \cite{MA2019} for a recent review. 
Our algorithm maintains a population of points where we have evaluated the function.
We obtain a new distribution from which additional points are sampled using rejection sampling. Our approach differs from swarm methods in that a gradient flow governs the evolution of these distributions.

	In \cite{lu2022swarmbased}, the authors introduce Swarm-Based Gradient Descent.
	The swarm includes agents, each defined by their position and mass. Their relative mass determines the agents' step size: heavier agents move in the local gradient direction with smaller time steps, while lighter agents use a backtracking protocol with larger steps. The authors explore the choice of time-step, creating a dynamic split between heavier "leaders" expected to approach local minima and lighter "explorers" who, with their large steps, are likely to find improved positions. Further, at each step, mass is transferred from agents with higher objective function values to those with lower ones. Unlike our approach, this method requires the evaluation of derivatives.

Evolutionary algorithms date back to 
\cite{rechenberg1973evolutionsstrategie}, \cite{schwefel1977evolutionsstrategien} (see 
		\cite{beyer2002evolution}). Other than particle swarms and differential evolution algorithms, 
evolutionary strategies also include
genetic algorithms \cite{barricelli1957symbiogenetic}, \cite{jh1975adaptation}, \cite{goldberg1989genetic}, 
cross-entropy methods \cite{rubinstein2004cross}, estimation of distribution algorithms \cite{pelikan2002survey}, \cite{larranaga2001estimation}, \cite{bosman2000expanding}
natural evolution strategies \cite{sun2009efficient}, \cite{glasmachers2010exponential}, 
\cite{wierstra2014natural},  and covariance matrix adaptation evolutionary strategies  (CMA-ES) \cite{hansen1996adapting}, \cite{hansen2001completely}, \cite{hansen2006cma}. 
Our algorithm incorporates several concepts and characteristics from previous approaches. Firstly, we view our gradient flow for $F$ as a method to learn the probability $\delta_{\bar \mu}(x)$, where $\bar \mu$ represents a global minimum of $f$, by sampling through a series of Gaussian distributions. This concept is employed in estimation distribution algorithms,
where this distribution is computed using  maximum likelihood estimation.

The concept of a gradient flow in the space of probability measures 
was proposed in \cite{berny2000selection}, and it was 
later explored using natural gradients \cite{amari1998natural} to 
develop natural evolution strategies in \cite{wierstra2014natural}. 
The authors suggest that sampling is not needed at every step of the algorithm. They introduce an importance weighting to avoid  sampling and employ statistical tests to control the quality of this approximation.
Similarly, our algorithm does not sample at every iteration, as we utilize rejection sampling based on previous evaluations.

In two recent papers, \cite{chaudhari2018deep} and \cite{osher2022hamilton}, the 
authors explore using sampling and Hamilton-Jacobi equations to build optimization algorithms. 
The first paper, \cite{chaudhari2018deep}, introduces a new zero-order algorithm called Hamilton-Jacobi-based Moreau Adaptive Descent (HJ-MAD). This algorithm guarantees convergence to global minima, given a continuous objective function. 
The authors demonstrate HJ-MAD's efficiency, showing that it outperforms other algorithms in several nonconvex examples and consistently converges to the global minimizer.
The second paper,   \cite{osher2022hamilton}, focuses on accurately approximating the Moreau envelope and proximals. This approach allows for solving high-dimensional optimization problems using a relatively low number of samples. 
The Hopf-Cole transformation, which converts the Hamilton-Jacobi equation
used in those papers into the heat equation used here, provides a connection between this approach
and our algorithm. A main difference, however, is that in contrast with our work, their 
regularization parameter, that somewhat corresponds to $\sigma$ in our approach, does not have to converge to $0$. 

Our algorithm 
also shares some features with the Covariance Matrix Adaptation Evolution Strategy (CMA-ES) algorithm. Both algorithms seek to adjust the mean and
 standard deviation (or, with CMA-ES, the covariance matrix) of a distribution to find the minimum of a function. 
 Moreover, CMA-ES can be seen as a gradient flow, as explored in the paper \cite{akimoto2010bidirectional}. 
 In general, CMA-ES attempts to track the principal components of the covariance matrix and adapt the geometry of the objective function (see \cite{LNV20151} and \cite{LNV20152} for improvements on Hansen and co-authors' algorithm, as well as convergence results). 
 Because we are working in one dimension, there is no need to consider a covariance matrix. 
 In future research, we will explore higher-dimensional cases where analyzing principal components of the covariance may prove valuable. 
 Our proposed algorithm differs from previous CMA-ES strategies in three main ways: (1) it employs rejection sampling to minimize function evaluations,  (2) it uses a quadratic model to approximate the gradient of the quadratic functional, integrates the corresponding gradient flow with an exponential-like integrator, and (3) uses error estimates for faster convergence within prescribed error tolerances.

Simulated annealing is another meta-heuristic method for minimizing a function $f$ \cite{pincus}, \cite{khachaturyan1979statistical}, \cite{kirkpatrick1983optimization}, related to the Metropolis-Hastings algorithm. It generates a sequence of candidate solutions through the following process: given a point $x_j$, it samples a new point $y$ within its neighborhood; the acceptance probability for $y$ depends on $f(x_j)$, $f(y)$, and a temperature parameter $T_j$; if $y$ is accepted, $x_{j+1}$ is set to $y$; otherwise, another point is sampled. Initially, $T_j$ 
has a high value, allowing non-improving points to be accepted. As $j$ increases, $T_j$ decreases, favoring acceptance of only improving points.
In our algorithm, the closest analog to a temperature parameter is $\sigma_j$, which does not strictly decrease, as demonstrated by the flow lines in Figure \ref{stream}. A possible variation involves prescribing a fixed cooling schedule for $\sigma$, such as $\dot \sigma=-\sigma$, and adjusting $\mu$ according to the gradient flow $\dot \mu=-\frac{\partial F}{\partial \mu}$. However, the mathematical properties of this approach remain uncertain as it does not guarantee a monotone decrease in $F(\mu(t), \sigma(t))$.

In \cite{vardhan2022tackling},  the authors examine the perturbed stochastic gradient descent (SGD) method for non-convex optimization problems and 
identify a class of non-convex functions for which convergence to a  global minimum occurs.  
The perturbed SGD method can be interpreted as sampling from a given Gaussian, which is similar to the approach here. However, the algorithm in that paper requires the computation of derivatives.

Model-based algorithms \cite{more1983computing}, \cite{shultz1985family}, \cite{byrd1987trust} approximate the objective function using a quadratic and attempt to minimize it within a trust region where the approximation is valid. Bayesian optimization, \cite {Mockus1975166} and 
\cite{mockus1978application}, combines trust region concepts with stochastic analysis \cite{garnett_bayesoptbook_2022}: the objective function is treated as a realization of a random process.
Our algorithm, which also approximates the objective function using a quadratic function, presents a fundamentally different approach from previous methods.
Specifically, the quadratic function we use may not necessarily provide an accurate approximation.
However, by combining the least squares optimality conditions with the algebraic structure of Gaussians, we demonstrate that the gradient flow associated with the quadratic approximation closely aligns with the original, regardless of the approximation's accuracy.

%


\subsection{Algorithm outline}

We outline the proposed minimization algorithm, which generates a sequence $(\mu_j, \sigma_j)$ approximating a minimizer of $F$. This is achieved by approximating the gradient flow of $F$ starting at $(\mu_0, \sigma_0)$. A naive approach would involve generating $(\mu_j, \sigma_j)$ according to Algorithm \ref{naive}, and terminating the algorithm 
when a stopping criterion is met.
%
\begin{algorithm}
	\caption{Naive algorithm}
	\label{naive}
	\begin{algorithmic}[1]
		\State Initialize $(\mu_0, \sigma_0)$   
		\While{Stopping criteria not met} 
		\State Sample $n_0$ points $x_i$
		according to the probability distribution $\Gamma_{\mu_j, \sigma_j}$.
		\State Using the pairs  $(x_i, f(x_i))$ and 
		least squares, find a quadratic approximation $q_j$ of  $f$. 
		\State Solve the exact gradient flow \eqref{qgf} for $q_j$ up some time step $\bar T$ with initial conditions $(\mu_j, \sigma_j)$ to produce $(\mu_{j+1}, \sigma_{j+1})$.
		\EndWhile
		\State Output best point found. 
	\end{algorithmic}
\end{algorithm}

Algorithm \ref{naive} presents several challenges that our algorithm addresses. First, computing $n_0$ points at every iteration, results in many function evaluations, as demonstrated in Section \ref{numresults}. To mitigate this, we employ a rejection sampling technique that reuses previous evaluations whenever possible, with mathematical details provided in Section \ref{rejesec}.
Another issue concerns the time step. As discussed in Section \ref{gfsec}, the gradient flow of
the least squares approximation 
$q_j$ of $f$ accurately approximates the gradient flow of $f$ for short times, even if there is a significant approximation error between $f$ and $q_j$. We select the time step based on error estimates rather than an arbitrary fixed value, as explained in Section \ref{timestepsec}. These three concepts -- quadratic approximation, rejection sampling, and adaptive time step -- are major contributions of our algorithm and make it efficient and competitive to minimize $F$ instead of $f$. 
We prove the convergence of the algorithm in Section \ref{tsa}.

Additionally, we implemented several improvements to enhance performance further. For instance, the number of points sampled is fixed in Algorithm \ref{naive}. However, it is natural to 
employ more points for poor quadratic approximations of $f$ and fewer for better ones, as detailed in Section \ref{adaptivitysec}.
Thus, we developed an adaptive sample size strategy that chooses a variable number $ n_j$ of samples at iteration $j$.
 Moreover, as Section \ref{sparsesec} outlines, we may not need to  sample and find a new quadratic approximation, $q_j$,  at every iteration when $q_j$ fits $f$ well.
Thus, we implement a sparse  sampling strategy that only samples $f$ and estimates $q_j$ when needed.  
 Finally, our stopping criterion accounts for error estimates and handles boundary and interior points differently, as discussed in Section \ref{stop}. Following the final iteration, we use a postprocessing step to improve accuracy by utilizing the quadratic approximation computed in the last iteration (Section \ref{postprocesssec}).

Given the relatively few points used in rejection sampling, restart strategies can enhance  accuracy without significantly increasing the number of function evaluations. We employ a combination of two methods. 
Occasionally, due to random sampling, 
the algorithm may evaluate $f$ at points better than the final value, indicating potential convergence to a local minimum. To address this, we implemented a restart strategy, ensuring our algorithm's result is always close to the best point where $f$ was evaluated (Section \ref{restartsec}) by restarting at the best point found so far. 
 In addition to this restarting strategy, we can also repeat the algorithm with a random initial condition but using all prior function evaluations. This further improves our code accuracy with a minimal increase in function evaluations.  We discuss this boosting strategy in Section \ref{boosting}.

We summarize the complete algorithm in Algorithm \ref{algo}  and refer the reader to later sections
for the technical details.
\begin{algorithm}
	\caption{Proposed algorithm}
	\label{algo}
	\begin{algorithmic}[1]
		\State Initialize $\Lambda_0=\{\}$   \Comment{Sec. \ref{rejesec}}
		\While{Boosting cycles limit is not achieved}		\Comment{Sec. \ref{boosting}}
		\State Initialize $(\mu_0, \sigma_0)$   \Comment{Either user provided or defaults in Sec. \ref{parsec}}
		\While{Restart criterion met} \Comment{Sec. \ref{restartsec}}
		\While{Stopping criteria not met} \Comment{Sec. \ref{stop}}
		\State Compute new sample of $ n_j$ points and update $\Lambda_j$
		\Comment{if using sparse  sampling (SR) only when needed, Sec. \ref{rejesec} and \ref{sparsesec}}
		\State If using adaptive sample size compute $n_{j+1}$ \Comment{Sec. \ref{adaptivitysec}}
		\State Compute quadratic approximation $q_j$ (if SR is used, only when needed) \Comment{Sec. \ref{lsas}}
		\State Compute $T_j$  \Comment{Sec. \ref{timestepsec}}
		\State Compute $(\mu_{j+1}, \sigma_{j+1})$   (if using SR compute
		$\tilde \gamma_i$)  \Comment{Sec. \ref{lfi} and \ref{sparsesec}}.
		\EndWhile
		\State If restart criterion is met set $(\mu_{j+1}, \sigma_{j+1})$ to restart value \Comment{Sec. \ref{restartsec}}
		\EndWhile	
		\State Postprocessing  \Comment{Sec. \ref{postprocesssec}}
		\EndWhile			
		\State Output the best point found. 
	\end{algorithmic}
\end{algorithm}

\subsection{Numerical experiments}

To evaluate the overall performance of our algorithm, we benchmarked it against four global optimization algorithms: Nelder-Mead, Random Search, Differential Evolution, and Simulated Annealing, illustrating the strengths and weaknesses of each algorithm under different conditions.

In summary, our algorithm excels in  the number of function evaluations and optimization efficiency. Its relative complexity leads to higher computational overhead, making it slower in terms of run time for less computationally intensive functions.  The significant reduction in the number of function evaluations 
makes our algorithm suitable for applications where such evaluations are expensive.
 Section \ref{numresults} and Appendix \ref{breakdown} provide a comprehensive breakdown of our findings.

Special thanks to Luis Espath for his thoughtful input and suggestions.

\section{Mathematical properties of the relaxed functional and associated gradient flow}
\label{mathprosec}

Our algorithm minimizes an objective function, $f$, through the approximation of the gradient flow \eqref{gf} of a relaxed functional, $F$. This section highlights key mathematical properties of $F$ and its gradient flow that motivate and justify our approach.
 We begin with the assumptions on $f$ and the basic properties of  $F$. Then, in Section \ref{cca}, we examine the properties of \eqref{gf} when $f$ is convex
and demonstrate convergence to a minimizer. Next, in Section \ref{nca}, we consider the non-convex case. Lastly, Section \ref{gfa} examines
the gradient flow for quadratic functionals, from which we develop an iterative scheme, the basis of our algorithm.

\subsection{Assumptions and elementary properties}

We want to ensure that the integral in \eqref{Fdef} exists
and that we can exchange derivatives with the integral sign. 
For this, we consider the class of functions, $C^k_p(\Rr)$,
that have $k$ continuous derivatives
and such that for all  $j$ with $0\leq j\leq k$, 
their $j$-th order derivatives, $f^{(j)}$, 
satisfy the following polynomial growth condition
\begin{equation}
\label{pgc}
\lim_{|x|\to +\infty}\frac{f^{(j)}(x)}{|x|^m}=0
\end{equation}
for some  $m>0$.
The space $C^0_p(\Rr)$ is denoted by $C_p(\Rr)$.

Suppose $f\in C_p(\Rr)$. 
$F$ is a convolution of 
a function with  polynomial growth, $f$, 
with a Gaussian. Consequently, $F$ is smooth and  we can exchange derivatives with the integral sign to compute the derivatives of $F$. Moreover, $\Gamma_{\mu, \sigma}$
converges (in weak sense) to the Dirac-$\delta$ at $\mu$ when $\sigma\to 0$. 
Accordingly, 
we have the following two elementary properties:
\begin{enumerate}
	\item $F$ is smooth for $(\mu, \sigma)\in \Rr\times \Rr^+$;
	\item For any $\mu$, $\lim_{\sigma\to 0}F(\mu, \sigma)=f(\mu)$.
\end{enumerate}

To guarantee a global minimizer for $f$ (though not necessarily unique), we assume the following:
\begin{hyp}
\label{a2}
$f$ satisfies the growth condition
\[
\lim_{|x|\to \infty} f(x)=+\infty.
\]
\end{hyp}

If $f\in C_p(\Rr)$ satisfies the above assumption, it has a global minimizer, $\bar x$, that may not be unique. In addition to ensuring a global minimizer's existence, this assumption implies the following growth property for $F$:
\begin{pro}
Suppose that  $f\in C_p(\Rr)$ satisfies Assumption \ref{a2}. Then
\begin{equation}
\label{lim}
\lim_{|\mu|+\sigma\to +\infty} F(\mu, \sigma)=+\infty.
\end{equation}
\end{pro}
\begin{remark}
When $|\mu|+\sigma\to +\infty$, then either $\mu$ or $\sigma$ diverge. Accordingly, this proposition means that 
$F$ becomes large once one of its arguments is large. Thus, for minimization purposes, 
both $\mu$ and $\sigma$ can be constrained to a bounded set. This observation is
explored in different results, e.g., Corollary \ref{boundtraj} and Proposition \ref{convmu}.
\end{remark}
\begin{proof}
First, we observe the following elementary fact. 
When either $\mu$ or $\sigma$ diverge, $\Gamma_{\mu, \sigma}(x)$ converges pointwise to $0$. 
Further, this convergence is uniform for $x$ on any compact set  $K\subset \Rr$. More precisely, elementary computations show that there exists a constant $C_K>0$ such that when $|\mu|+\sigma>A$, for large $A$, we have
\[
\Gamma_{\mu, \sigma}(x)\leq \frac{C_K}{\sqrt{A}}
\]
for all $x\in K$. 
Thus,
\[
\lim_{|\mu|+\sigma\to +\infty} \int_K \Gamma_{\mu, \sigma}(x) dx \to 0. 
\]
Consequently, $\lim_{|\mu|+\sigma\to +\infty} \int_{K^c} \Gamma_{\mu, \sigma}(x) dx \to 1$, where $K^c$ is the complement of $K$.

Furthermore, for every $M>0$, the set
$K_M=\{x\in\Rr: f-\min_x f\leq M\}$ is compact. Moreover, 
\[
F(\mu, \sigma)-\min_x f=\int_{\Rr} (f-\min_x f)\Gamma_{\mu, \sigma}(x) dx ][]\geq M \int_{K_M^c}\Gamma_{\mu, \sigma}(x) dx.
\]
Therefore, 
\[
\lim_{|\mu|+\sigma\to +\infty}F(\mu, \sigma)\geq \min_x f+ M, 
\]
for any real $M$. Hence, \eqref{lim} holds.  
\end{proof}

\begin{cor}
\label{boundtraj}
Suppose that  $f\in C_p(\Rr)$ satisfies Assumption \ref{a2}. 
Let $(\mu(t), \sigma(t))$ solve \eqref{gf}. Then, there exists $C>0$ such that 
$|\mu(t)|+ \sigma(t)\leq C$ for all $t>0$. 
\end{cor}
\begin{proof}
Because the gradient flow \eqref{gf} decreases $F$, the result follows from the preceding proposition. 
\end{proof}

\begin{pro}
	\label{relaxpro}
Suppose that  $f\in C_p(\Rr)$ satisfies Assumption \ref{a2}. Then 
	$\inf_{\mu, \sigma} F=\min_x f$; that is, \eqref{relaxid} holds.
\end{pro}
\begin{proof}
Because $\Gamma_{\mu, \sigma}$ is a probability density function and $f\geq \min_x f$, we have
	\[
	F(\mu, \sigma)\geq  \min_x f. 
	\]
	Therefore, $\inf_{\mu, \sigma} F\geq \min_x f$. 
	Let $\bar x$ be a global minimizer of $f$. In the limit $\sigma \to 0$, $\Gamma_{\mu, \sigma}$ converges to a Dirac-$\delta$ distribution and thus
	$$
	\lim_{\sigma \to 0} F(\bar x,\sigma )=f(\bar x).
	$$
	Then,
	$$
	\inf_{\mu, \sigma} F \leq \lim_{\sigma\to 0} F(\bar x,\sigma )=f(\bar x)=\min_x f. 
	$$
	Finally, if $\inf_{\mu, \sigma} F\geq \min_x f$ and $\inf_{\mu, \sigma} F\leq\min_x f $, then  \eqref{relaxid} holds.
\end{proof}

\begin{remark}
Assumption \ref{a2} was only used in the preceding proof to ensure the existence of a minimum $\bar x$. 
The result holds even if $f$ lacks a global minimum, in which case, we have
$\inf_{\mu, \sigma} F=\inf_x f$.
\end{remark}

\begin{pro}
	\label{heatpro}
	Suppose that  $f\in C_p(\Rr)$. Then $F$	satisfies \eqref{heat}.	
\end{pro}
\begin{proof}
	Differentiate \eqref{Fdef} with respect to $\sigma$ to show that it is equal to the second partial derivative of \eqref{Fdef} with respect to $\mu$ multiplied by $\sigma$.
\end{proof}

If $f\in C^2_p$, we have that $\lim_{\sigma\to 0} \frac{\partial F} {\partial \mu }(\mu, \sigma)$ and $\lim_{\sigma\to 0} \frac{\partial^2 F} {\partial \mu^2 }(\mu, \sigma)$ exist. Thus, because of the preceding proposition, $\nabla F$, extends as a continuous function up to $\sigma=0$. 
Thus, we can consider the gradient flow for $\sigma \geq 0$ instead of $\sigma>0$.

A further consequence of the preceding proposition is the following.
\begin{cor}
	Suppose that  $f\in C_p(\Rr)$ and satisfies  Assumption \ref{a2}. Then $F$ has no minima for $\sigma>0$. 
\end{cor}
\begin{proof}
By Proposition \ref{heatpro}, $F$ satisfies a modified heat equation. The minimum principle gives the result. 
\end{proof}

\subsection{Convex case}
\label{cca}

We now explore the properties of the relaxed functional $F$ and its associated gradient flow when $f$ is convex. To simplify the discussion, we assume that $f\in C^2_p(\Rr)$ and, thus, 
$f$ is convex if and only if $f''\geq 0$.
We recall that a $C^2$  function is strictly convex if $f''>0$ and is uniformly convex if there exists $\tau>0$ such that $f''(x)\geq \tau$ for all $x\in \Rr$. 
Uniformly convex functions satisfy Assumption \ref{a2}, whereas strictly convex functions may not; for instance, $f(x)=e^x$ is strictly convex but does not meet Assumption \ref{a2}. Consequently, Assumption \ref{a2} must be verified independently for convex functions that are not uniformly convex. Furthermore, if $f$ is strictly convex and satisfies Assumption \ref{a2}, its minimizer is unique.

\begin{pro}[Preservation of convexity]
	\label{preconv}
Suppose that  $f\in C_p^2(\Rr)$.
 If $f$ is convex, then $\mu \mapsto F(\mu, \sigma)$ is also convex; that is, if $f''\ge 0$, then $F_{\mu\mu}\ge 0$. Moreover, if $f$ is strictly convex, then $F_{\mu\mu}> 0$. Furthermore, 
 if $f$ is uniformly convex with $f''\geq \tau>0$, then $F_{\mu\mu}\geq \tau$. 
\end{pro}
\begin{proof}
	Consider the change of variable $u=x-\mu$ and rewrite the definition \eqref{Fdef}  of $F$ as
	$$
	\frac{\partial^2 F}{\partial \mu\partial \sigma}=\int_{\Rr} f(u+\mu) \Gamma_{0, \sigma}(u) du.
	$$
	Because $f\in C_p^2(\Rr)$, Differentiating twice with respect to $\mu$ yields
	$$
	\frac{\partial^2 F}{\partial \mu^2}= \int_{\Rr} f'' \Gamma_{0, \sigma}(u) du\ge 0
	$$
	from the hypothesis $f''\ge 0$. If $f$ is strictly convex, 
the preceding inequality is strict, and if $f$ is uniformly convex with $f''\geq \tau>0$, then $\frac{\partial^2 F} {\partial \mu^2 }\geq \tau$, taking into account that $\Gamma_{0, \sigma}$ is a probability density. 
\end{proof}

\begin{remark}
The preceding proposition does not assert convexity of $F$ in $(\mu, \sigma)$, only in $\mu$ for fixed $\sigma$.  Together with \eqref{heat}, this convexity gives $\frac{\partial F} {\partial \sigma}\geq 0$. 
\end{remark}

\begin{pro}
	\label{convsigma} 
Suppose that  $f\in C_p^2(\Rr)$ satisfies Assumption \ref{a2}. Suppose further that $f$ is strictly convex. 
Let $(\mu, \sigma)$ solve  \eqref{gf}. 
Then, as $t\to \infty$, we have $\sigma(t)\to 0$.
Moreover, if $f$ is uniformly convex, then $\sigma\to 0$ exponentially. 
\end{pro}
\begin{proof}
	From  \eqref{heat} and \eqref{gf}, we have
	$$
	\dot \sigma = - \frac {\partial F}{\partial \sigma}= -\sigma \frac {\partial^2 F}{\partial \mu^2}<0,
	$$
	since, from Proposition \ref{preconv}, $\frac {\partial^2 F}{\partial \mu^2} >0$ if $f''>0$. That is, $\sigma(t)$ is a strictly decreasing function and therefore approaches its infimum, which is 
	$\lim_{t\to +\infty}\sigma(t)\geq 0$.  Now, we prove that this limit vanishes.  We have
	\[
	\int_0^{+\infty}\sigma(t) \frac {\partial^2 F}{\partial \mu^2}(\mu(t), \sigma(t)) dt =-	\int_0^{+\infty}\dot \sigma dt=\sigma(0)-\lim_{t\to +\infty } \sigma(t)<+\infty. 
	\]
	Taking into account that $\sigma(t) \frac{\partial^2 F}{\partial \mu^2}(\mu(t), \sigma(t)) \geq 0$, we conclude that $\sigma(t) \frac{\partial^2 F}{\partial \mu^2}(\mu(t), \sigma(t))$ is integrable on $[0, +\infty)$. Thus, 
	 there exists a sequence $T_j\to \infty$ such that 
	$\sigma(T_j) \frac{\partial^2 F}{\partial \mu^2}(\mu(T_j), \sigma(T_j))\to 0$. 
	Because of Corollary \ref{boundtraj}, $(\mu(T_j), \sigma(T_j))$ is bounded. Hence, we can extract a subsequence such that $(\mu(T_j), \sigma(T_j))\to (\bar \mu, \bar \sigma)$. 
	If $\bar \sigma= 0$, because $\sigma$ is decreasing and hence convergent we have
	$\lim_{t\to +\infty } \sigma(t)=0$. 
	We claim this must be the case,
	as the alternative leads to a contradiction. In fact, if $\sigma(T_j)$ does not converge to $0$, 
    we have  $\frac{\partial^2 F}{\partial \mu^2}(\mu(T_j), \sigma(T_j))\to 0$. By continuity, $\frac{\partial^2 F}{\partial \mu^2}(\bar \mu, \bar \sigma)=0$, 
	which contradicts the inequality $\frac{\partial^2 F}{\partial \mu^2}>0$ from Proposition \ref{preconv}. 
	
	In the uniformly convex case, that is, if $f''\geq \tau>0$, we have by Proposition \ref{preconv},
	$F_{\mu\mu}\geq \tau$. Hence, 
	\[
	\dot \sigma\leq -\tau \sigma, 
	\]
	which gives $\sigma(t)\leq \sigma_0e^{-\tau t}$.
\end{proof}

\begin{pro}
	\label{convmu}
Suppose that  $f\in C_p^2(\Rr)$ satisfies Assumption \ref{a2}. Suppose further that $f$ is strictly convex. Then $\mu(t)\to \bar x$ as $t\to\infty$,
	where $\bar x$ is the unique minimizer of $f$.
\end{pro}
\begin{proof}
By the preceding proposition, we have $\sigma(t)\to 0$. Furthermore, $\mu(t)$ is bounded by Corollary
\ref{boundtraj}. By  integrating the identity in \eqref{Fdec1} and discarding the term $\left(\frac {\partial F}{\partial \sigma}\right)^2$, we have
\begin{equation}
\label{est}
\int_0^{+\infty} \left(\frac {\partial F}{\partial \mu}\right)^2dt \leq F(\mu(0), \sigma(0))-\lim_{t\to \infty}F(\mu(t), \sigma(t))<+\infty.
\end{equation}
We claim that $\frac {\partial F}{\partial \mu}(\mu(t), \sigma(t))\to 0$ as $t\to \infty$. If this is not the case, there is 
$\epsilon>0$ and 
a
sequence $T_j\to \infty$ such that $\left(\frac {\partial F}{\partial \mu}\right)^2(\mu(T_j), \sigma(T_j))>\epsilon$. Note that
$\mu$ and  $\sigma$ solve the ODE \eqref{gf} and take values on a compact set $K$. Accordingly, their time derivative, the right-hand side of \eqref{gf}, is bounded uniformly in time. Therefore,
$\mu$ and  $\sigma$ are Lipschitz 
continuous on $[0,+\infty)$. Further, on $K$, $F_\mu$ is uniformly continuous. Accordingly, there exists
$\delta>0$ such that 
$\left(\frac {\partial F}{\partial \mu}\right)^2(\mu(t), \sigma(t))>\frac \epsilon 2$, for all $T_j-\delta<t<T_j+\delta$. This contradicts the 
bound in \eqref{est}. Because
 $\sigma\to 0$,  $\frac{\partial F}{\partial \mu}(\mu(t), \sigma(t))-f'(\mu(t))\to 0$ as $t\to \infty$. Hence
 $f'(\mu(t))\to 0$. By continuity of $f'$, any accumulation point of $\mu(t)$ is a critical point and, by 
 convexity, a minimizer of $f$. 
 Since $f$ is strictly convex and satisfies Assumption\ref{a2} there is a unique minimizer $\bar x$. Thus, $\mu(t)$  converges to $\bar x$.
\end{proof}

\subsection{Non-convex case}
\label{nca}

The analysis of the non-convex case is more complex. Nevertheless, there are two aspects that 
we can address rigorously: the behavior as $\sigma\to 0$ and $+\infty$. 

First, we consider the behavior of \eqref{gf} near critical points of $f$. 
Due to \eqref{heat}, we can rewrite \eqref{gf} as 
\begin{equation}
\label{flow}	
	\begin{bmatrix}
	\dot \mu\\
	\dot \sigma
\end{bmatrix}
=-
\begin{bmatrix}
	\frac{\partial F}{\partial \mu}\\
	\sigma \frac{\partial^2 F}{\partial \mu^2}
\end{bmatrix}.
\end{equation}
This shows that $\sigma$ is increasing in regions where $F$ is concave in $\mu$.  Furthermore, we have the following result.
\begin{pro}
Let $f\in C^4_{p}(\Rr)$. 
Moreover, let $\tilde x$
be a critical point of $f$. Then, at $(\mu, \sigma)=(\tilde x, 0)$, the linearization of 
\eqref{gf} is
\[
	\begin{bmatrix}
	\dot \mu\\
	\dot \sigma
\end{bmatrix}
=-
\begin{bmatrix}
f''(\tilde x)&0\\
0&f''(\tilde x)
\end{bmatrix}
\begin{bmatrix}
\mu-\tilde x\\
\sigma
\end{bmatrix}.
\]
Because non-degenerate local maxima satisfy $f''<0$,  these points are repelling, whereas
non-degenerate local minima are attractive. 
\end{pro}
\begin{proof}
If $f\in C^4_p(\Rr)$, $F$ extends as a $C^2$ function up to $\sigma=0$. 
We compute the linearization of \eqref{gf} using its alternative form \eqref{flow}. 
A direct computation gives
\[
\lim_{\sigma\to 0} \frac{\partial^2 F}{\partial \mu^2}=f''(\tilde x).
\]
Using the identity \eqref{heat}, we have
\[
\frac{\partial^2 F}{\partial \mu\partial \sigma}=\sigma \frac{\partial^3 F}{\partial \mu^3}.
\]
Thus, 
\[
\lim_{\sigma\to 0} \frac{\partial^2 F}{\partial \mu\partial \sigma}=\lim_{\sigma\to 0} \sigma \frac{\partial^3 F}{\partial \mu^3}=0.
\]
Similarly, using \eqref{heat},  we have 
\[
\frac{\partial^2 F}{\partial \sigma^2}=\sigma \frac{\partial^3 F}{\partial \mu^2\partial \sigma}+\frac{\partial^2 F}{\partial \mu^2}, 
\]
hence 
\[
\lim_{\sigma\to 0} \frac{\partial^2 F}{\partial \sigma^2}=f''(\tilde  x).
\]
Gathering these computations, we get the linearization of 
\eqref{gf} in the statement. 
\end{proof}

The second aspect concerns the asymptotic behavior of $F$ as $\sigma\to \infty$.  For this, 
we  consider the case
where
$f=f_c+f_p$, where $f_c$ is a convex function and $f_p$ is a non-convex perturbation. 

\begin{pro}
Suppose that $f=f_c+f_p$ where $f_c$ is a convex function and $f_p$ is a non-convex perturbation. 
Let $F_c$ and $F_p$ be the corresponding relaxed functionals. 
Assume that $f_c$ is comparable to a quadratic function in the following sense:
for some $C>0$,
\[
\frac 1 C \leq f_c''(x)\leq C. 
\]
Then, we have
\begin{equation}
\label{ulb}
\frac 1 C\leq \frac{\partial^2 F_c}{\partial \mu^2}\leq C.
\end{equation}
Assume further that $f_p\in C^1_p(\Rr)$ is  an integrable function with $f'_p$ 
also integrable. Then, 
as $\sigma\to \infty$, 
\[
F(\sigma, \mu)=F_c(\mu, \sigma)+O\left(\frac{1}{\sigma}\right),
\]
and
\[
\left| \frac{\partial F}{\partial \mu}(\sigma, \mu)\right|=\left|\frac{\partial F_c} {\partial \mu} (\mu, \sigma)\right|+O\left(\frac{1}{\sigma}\right).
\]
\end{pro}
\begin{proof}
The proof of \eqref{ulb} is similar to the one in 	
Proposition \eqref{preconv}.
	
Because $f_p$ and its derivatives are integrable, we have
\[
\|F_p\|_{L^\infty}\leq C\frac{\|f_p\|_{L^1}}{\sigma}.
\]
In addition, we have
\[
\left\|\frac{\partial F_p} {\partial \mu }\right\|_{L^\infty}\leq C\frac{\|f_p'\|_{L^1}}{\sigma}.\qedhere
\]
\end{proof}

From the previous proposition, we see that when $\sigma$ is large, the convex part of $F$ drives the gradient flow component corresponding to $\mu$.

\subsection{Gradient flow}
\label{gfa}

Now, we give a formula for $\nabla F$ and study the gradient flow for a quadratic function.
Here, we provide an explicit solution for \eqref{gf} that is
crucial for constructing the iterative step in our algorithm.

\begin{pro}
	\label{gradflowcomp}
	The gradient of $F$ is 
	\begin{equation}
	\label{exactgf}
	\nabla F(\mu, \sigma)=
	\begin{bmatrix}
		\frac{\partial F} {\partial \mu }\\
		\frac{\partial F} {\partial \sigma}
	\end{bmatrix}
	=
	\begin{bmatrix}
		\int_{\Rr} \frac{x-\mu}{\sigma^2} f(x) \Gamma_{\mu, \sigma}(x) dx \\
		\int_{\Rr}  \frac{(x-\mu -\sigma ) (x-\mu + \sigma)}{\sigma^3} f(x) \Gamma_{\mu, \sigma}(x) dx
	\end{bmatrix}.
	\end{equation}
	In particular, if $f=a + b x + c x^2$, with $a,b,c\in\mathbb{R}$, then
	\begin{equation}
	\label{gfq}
	\nabla F
	=
	\begin{bmatrix}
		b+2 c\mu\\
		2 c\sigma
	\end{bmatrix}.
	\end{equation}
\end{pro}
\begin{proof}
	Differentiating
	\eqref{Fdef} with respect to $\mu$ and $\sigma$ results in the two integrals above. 
	Direct integration yields $F_{\mu}
	= b+2c\mu$ and $F_{\sigma} = 2c\sigma$ when $f$ is quadratic.
\end{proof}

\begin{cor}
\label{lsiden}
Let $q_j$ be the minimizer of \eqref{msexact}. Then
$\nabla F(\mu_j, \sigma_j)=\nabla F^{q_j}(\mu_j, \sigma_j)$.
\end{cor}
\begin{proof}
As shown in the preceding computation, the gradient of $F$ is computed by integrating a second-order polynomial $p(x)$ multiplied by $f \Gamma_{\mu, \sigma}$. However, by the necessary optimality condition in \eqref{nocexact}, we have
\[
\int p q_j \Gamma_{\mu_j, \sigma_j}dx=\int p f \Gamma_{\mu_j, \sigma_j}dx, 
\]
which gives the desired identity. 
\end{proof}

%
%

\begin{pro}
	\label{linsol}
Consider the gradient flow for a quadratic function $q=a+b x +c x^2$; that is, 
\begin{equation}
\label{quadgf}
\begin{bmatrix}
	\dot \mu\\
	\dot \sigma
\end{bmatrix}
=
\begin{bmatrix}
	-b-2 c\mu\\
	-2 c\sigma
\end{bmatrix}.
\end{equation}
Then, 
\begin{equation}
\label{quadgfsol}
\mu(t)=
\frac{ b(e^{-2 c t}- 1)}{2 c}+ \mu(0)e^{-2 c t}, 
\quad 
\sigma(t)=e^{-2 c t}\sigma(0). 
\end{equation}
\end{pro}
\begin{proof}
The result
is obtained by
solving the ODE.
\end{proof}

Our
algorithm iteratively approximates $f$ with a quadratic function
$q_j=a_j+b_j x+c_j x^2$ and replaces the gradient flow associated
with $f$ with that of $q_j$.
Accordingly, 
we use \eqref{quadgfsol} with $\mu(0)=\mu_j$ and $\sigma(0)=\sigma_j$
to obtain 
\begin{equation}
\label{quadgfsoli}
\mu_{j+1}=
\frac{ b_j(e^{-2 c_j T_j}- 1)}{2 c_j}+ \mu_je^{-2 c_j T_j}, 
\quad 
\sigma_{j+1}=e^{-2 c_j T_j}\sigma_j. 
\end{equation}
The construction of the function $q_j$ is discussed in the next section.
The choice of the time step $T_j$ relies on the error estimates examined in Section \ref{gfsec}. 
%

\section{Sampling and approximation}
\label{sampsec}

Our algorithm samples points from a
Gaussian distribution $\Gamma_{\mu_j, \sigma_j}$ to build a
quadratic model of $f$.
We can achieve
this by selecting $ n_j$ independent samples at each iteration. However, this requires many function evaluations; see Section \ref{numresults}.
Therefore, we  propose to use the rejection sampling \cite{MR1999614} strategy we discuss next. 
	Our methodology systematically reuses previous samples, reducing function evaluations significantly.
	 Moreover, 
	despite inter-iteration dependencies from sample re-use, they do not
	affect the algorithm's convergence properties (see Section
	\ref{tsa}).

\subsection{Rejection sampling}
\label{rejesec}

Consider an infinite sequence of triplets $(x_k, \mu_k, \sigma_k)$, where $x_k$ is drawn from a Gaussian distribution $\Gamma_{\mu_k, \sigma_k}$. Each triplet represents a sample generated from the respective Gaussian in earlier iterations. 
We generate $ n_j$ samples from this distribution at
iteration $j$. Before sampling a new point, 
in rejection sampling, we check if any $x_k$ in the sequence can be used as a sample from the current Gaussian, $\Gamma_{\mu_j, \sigma_j}$ according to the following procedure. We define
$$
M_k= \sup_x \frac{
	\Gamma_{\mu_j, \sigma_j}(x)}{\Gamma_{\mu_{k}, \sigma_{k}}(x)}.
$$
If $\sigma_j<\sigma_k$, the supremum exists and is a critical point of the ratio of the Gaussian distributions.
This results in the expression
\[
M_k=\frac{ \sigma_k}{\sigma_j}e^{-(\mu_j - \mu_k)^2/(2 (\sigma_j^2 - \sigma_k^2))}.
\]
Otherwise, if $\sigma_j\geq \sigma_k$, the supremum does not exist, and we set $M_k=+\infty$.

The ratio $M_k$ is used to define
\begin{equation}
\label{pik}
\pi_k = \frac{
	\Gamma_{\mu_j, \sigma_j}(x_k)}{M_k\Gamma_{\mu_{k}, \sigma_{k}}(x_k)}
\end{equation}
with $0\le \pi_k\le 1$. We accept $x_k$ as a sample from the current Gaussian $\Gamma_{\mu_j, \sigma_j}$ with probability $\pi_k$. 
After examining all triplets $(x_k,\mu_k, \sigma_k)$, we end up with  a sequence with distribution $\Gamma_{\mu_j, \sigma_j}$. 

Two issues need to be addressed: insufficient samples and reduced independence between iteration steps due to sample reuse. To tackle the first issue, we sample new points from $\Gamma_{\mu_j, \sigma_j}$ whenever the number of points selected is less than $n_j$. The second issue occurs when previous time steps leading to the current $\mu_j$ and $\sigma_j$ are small, causing $\pi_k$ values to be close to 1. This results in relying heavily on previous samples and fewer independent samples. In our experience, this is suboptimal. To mitigate this, we introduce a multiplicative parameter $0<p<1$ and accept each $x_k$ with independent probability $\tilde \pi_k$ given by
\begin{equation}
	\label{selprob}
\tilde \pi_k=p \pi_k=	p\frac{
		\Gamma_{\mu_j, \sigma_j}(x_k)}{M_k\Gamma_{\mu_{k}, \sigma_{k}}(x_k)}.
\end{equation}
The sequence obtained by this procedure is iid and has a common 
distribution $\Gamma_{ \mu_j,  \sigma_j}$.
More about rejection sampling in \cite{MR1999614}.

In summary, our sampling procedure begins with an empty list:
\begin{equation}
	\label{lambda1}
	\Lambda_0=\{
	\}.
\end{equation}
At iteration $n>0$, we run through $\Lambda_{j-1}$ and independently select points with probability given by \eqref{selprob}. If the number of selected points is equal or exceeds $ n_j$, we randomly choose $n_j$ points from this list. Otherwise, we sample additional points from $\Gamma_{\mu_j,\sigma_j}$. The selected and newly sampled points form the sample of $ n_j$ points at iteration $j$. 
Then, we build $\Lambda_{j}$ by appending to $\Lambda_{j-1}$  
the quadruplets \[(x_k, f(x_k), \mu_{ n_j}, \sigma_{ n_j})\] 
corresponding to newly sampled points. 
 Thus, $\Lambda_j$ contains all unique quadruplets up to iteration $j$ of the algorithm.

The rejection sampling procedure yields a list of $ n_j$ iid random variables. However, 
sample overlap can lead to systematic errors and dependence across iterations. 
The parameter $p$ in \eqref{selprob} mitigates this by rejecting some prior samples with probability $1-p$. Nevertheless, our convergence proof in Section \ref{tsa} does not require independence between iterations.

\subsection{Least squares approximation}
\label{lsas}

At each iteration, we aim to find a quadratic function $q_j$ that minimizes the integral in \eqref{msexact} with $\mu=\mu_j$ and $\sigma=\sigma_j$.
This requires the computation of the integrals in \eqref{msexact} or in \eqref{nocexact} exactly, which
 we try to avoid here. Instead, we use a Monte Carlo method to approximate $q_j$. 
 For that, we sample $ n_j$ points, $x_i$, according to the distribution $\Gamma_{\mu_j, \sigma_j}$.
 Then, we replace the minimization problem in  \eqref{msexact} with its Monte Carlo approximation; that is, the problem of minimizing 
\begin{equation}
	\label{mcls}
	\frac 1 { n_j}\sum_{i=1}^{ n_j} |q(x_i)-f(x_i)|^2
\end{equation}
among all quadratic functions. 
The minimizer, $q_j(x)=a_j+b_j x+c_j x^2$, satisfies the least squares condition: for any polynomial $p$ of degree less than or equal to $2$,
\begin{equation}
	\label{mcoc}
	\sum_{i=1}^{ n_j}  p(x_i) (f(x_i)-q_j(x_i))=0.
\end{equation}
This condition is the Monte Carlo approximation to \eqref{nocexact}.

\section{Error estimates}
\label{gfsec}

Here, we discuss the error estimates in our approximations and explain the 
time step selection criterion. 
We first analyze the difference
between the gradient flow of $f$ and $q$, and establish a bound on
this difference up to a time $T$.
Next, we show how to  approximate this error estimate using Monte Carlo integration, Taylor series, and importance sampling techniques. Lastly, we discuss the appropriate selection of time step $T_j$ at iteration $j$, constrained by the desired error bound and the validity of the approximations.  The time step choice plays a crucial role. It must strike a balance, being neither excessively large, which could lead to accumulation of errors, nor exceedingly small to maintain an exponential contraction per iteration in equation \eqref{quadgfsoli}. Furthermore,
if the time steps $T_j$ were chosen too small, it could happen that $\sum_{j=1}^\infty T_j$ converges to a finite time $T^*<\infty$; thus,  in this case, our algorithm would not capture the asymptotic behavior of the gradient flow. However, with high probability, this is not the case,
as follows from Proposition \ref{513} and Theorem \ref{mainteo}.

\subsection{Error estimates}
\label{errorsec}

 Let $q$ be a second-order polynomial, $q(x) = a + bx + cx^2$, where $a, b, c \in \mathbb{R}$. We aim to compare the gradient flow associated with $f$, corresponding to $\nabla F$, with the one corresponding to $q$, represented as $\nabla F^q$. We do not assume here that $q = q_j$, the least squares approximation of $f$, since the general case is required for some improvements of our algorithm, specifically for the sparse  sampling discussed in Section \ref{sparsesec}.

Let $\theta=(\mu, \sigma)$ and set
\begin{equation}
\label{bif}
B(\theta, x)=\begin{bmatrix}
	B_1\\
	B_2	
\end{bmatrix}
=
\begin{bmatrix}
	\frac{x-\mu}{\sigma^2}   \\
	\frac{(x-\mu -\sigma ) (x-\mu + \sigma)}{\sigma^3} 
\end{bmatrix}. 
\end{equation}
Then, 
\[
\nabla F(\theta)=\int_{\Rr}  f(x)B(\theta, x) \Gamma_{\theta}(x)dx, \qquad \nabla F^q(\theta)=\int_{\Rr}  q(x)B( \theta, x)\Gamma_{\theta}(x)dx. 
\]

We start at iteration $j$ with the initial condition $\theta^f(0)=\theta^q(0)=\theta_j$ and would like to track the difference in the evolution
of the two gradient flows:
\[
\dot \theta^f=-\nabla F(\theta^f),\qquad \dot \theta^q=-\nabla F^q(\theta^q).
\]
 
Now, we examine the difference between these two differential
 equations by letting $f=q+e$. Then,
\begin{align}
	\label{diffeq}
	\dot \theta^f-\dot \theta^q
	&=-\int_{\Rr}  \left[f(x) B(\theta^f, x) \Gamma_{\theta^f}(x) - q(x) B(\theta^q, x)\Gamma_{\theta^q}(x)\right]
dx\\\notag
	&=-\int_{\Rr}  q(x) \left[B(\theta^f, x)\Gamma_{\theta^f}(x) - B(\theta^q, x)\Gamma_{\theta^q}(x)\right]dx-\int_{\Rr}  e(x)B( \theta^f, x)\Gamma_{\theta^f}(x)dx\\\notag
	&=-2 c (\theta^f-\theta^q) -\int_{\Rr}  e(x)B(\theta^f, x)\Gamma_{\theta^f}(x)dx, 
\end{align}	
taking into account that, because $q$ is a quadratic polynomial,
\[
\int_{\Rr}  q(x) \left[B( \theta^f, x) \Gamma_{\theta^f}(x) - B(\theta^q, x)\Gamma_{\theta^q}(x) \right]dx=2c (\theta^f-\theta^q), 
\]
by direct integration as in \eqref{gfq}.

Next, we estimate $\int_{\Rr}  e(x)B(\theta^f, x)\Gamma_{\theta^f}(x)dx$.
\begin{pro} 	
For $i=1,2$, we have
\begin{align}
\label{errorest}
&\left|\int_{\Rr}  e(x)B_i(\theta^f, x)\Gamma_{\theta^f}(x) dx\right|\\\notag
&\qquad \leq 
\left(\int_{\Rr}  e^2\Gamma_{\theta_j}( x) dx\right)^{1/2} 
\left(\int_{\Rr}  
\frac{\left(B_i(\theta^f, x)\Gamma_{\theta^f}( x)-
	B_i(\theta^q, x)\Gamma_{\theta^q}(x)\right)^2}{\Gamma_{\theta_j}( x)}dx
\right)^{1/2} 
 \\\notag
&\qquad+ 
\left|
	\int_{\Rr}  e(x) B_i(\theta^q, x) 
\Gamma_{\theta^q}( x)dx
\right|.
\end{align}
\end{pro}
\begin{proof}
We write
\begin{align*}
	\int_{\Rr}  e(x) B_i( \theta^f, x)\Gamma_{\theta^f}(x) dx=&\int_{\Rr}  e(x) \left[
	B_i(\theta^f, x) \Gamma_{\theta^f}( x)-
	B_i(\theta^q, x) \Gamma_{\theta^q}( x)\right]dx \\&+
	\int_{\Rr}  e(x) B_i(\theta^q, x)\Gamma_{\theta^q}(x)dx.
\end{align*}
We use Cauchy's inequality for the first term to get 
\begin{align*}
&\left|\int_{\Rr}  e(x) \left[
B_i(\theta^f, x) \Gamma_{\theta^f}( x)-
B_i(\theta^q, x) \Gamma_{\theta^q}( x)\right]dx\right|
\\&\qquad\leq  \left(\int_{\Rr}  e^2\Gamma_{\theta_j}( x)dx\right)^{1/2}
\left(\int_{\Rr}  
\frac{\left(B_i(\theta^f, x)\Gamma_{\theta^f}( x)-
	B_i(\theta^q, x)\Gamma_{\theta^q}(x)\right)^2}{\Gamma_{\theta_j}( x)}
dx\right)^{1/2}. 
\end{align*}
\end{proof}

Next, we apply the previous estimates to obtain a bound on the error between the quadratic and the non-quadratic flows. 

\begin{pro}
\label{properror}
Let
\begin{equation}
\label{e0}
R=\left(\int_{\Rr}  e^2\Gamma_{\theta_j}( x)dx\right)^{1/2},
\end{equation}
\begin{equation}
\label{qi}
Q_i=\left(\int_{\Rr}  
\frac{\left(B_i(\theta^f, x)\Gamma_{\theta^f}( x)-
	B_i(\theta^q, x)\Gamma_{\theta^q}(x)\right)^2}{\Gamma_{\theta_j}( x)}dx
\right)^{1/2},
\end{equation}
and
\begin{equation}
\label{bi}
\beta_i=\left|
\int_{\Rr}  e(x) B_i(\theta^q, x)  
\Gamma_{\theta^q}( x)dx
\right|.
\end{equation}
Fix $T>0$ and  for $i=1,2$, consider the upper bound on the error term \eqref{errorest}
\begin{equation}
\label{epsiloni}
\epsilon_i=\sup_{0\leq t\leq T } \left(RQ_i+\beta_i\right). 
\end{equation}
Let
$z_1=|\mu^f-\mu^q|^2$ and $z_2=|\sigma^f-\sigma^q|^2$.
Then, for all $0\leq t \leq T$,
\begin{equation}
\label{errorprop}
\dot z_i\leq 2 \epsilon_i \sqrt{z_i}-4 c z_i, 
\end{equation}
for $i=1,2$. 
Furthermore, 
\begin{equation}
\label{error}
\sqrt{z_i(t)}\leq 
\begin{cases}
\frac{\epsilon_i  (1-e^{-2 c t} )}{2 c}\quad &c\neq 0\\
t \epsilon_i&c=0
\end{cases}
\end{equation}
for all $0\leq t \leq T$. 
\end{pro}
\begin{remark}
The previous proposition's bound is implicit. For a given time $T$, the error estimates for $|\mu^f-\mu^q|$ and $|\sigma^f-\sigma^q|$ depend on these quantities' errors over the interval $0\leq t\leq T$, particularly through  \eqref{qi}.
 We tackle this challenge in Section \ref{timestepsec}, where we discuss how to select the appropriate time step $T$.	
\end{remark}

\begin{proof}
The inequality in \eqref{errorprop} follows by multiplying \eqref{diffeq} 
by $2 (\mu^f-\mu^q)$, for $z_1$ and by $2 (\theta^f-\theta^q)$, for $z_2$,  and
using the estimate \eqref{errorest} in the preceding proposition.


From \eqref{errorprop}, it follows that  $\sqrt{z_i}$ solves the linear differential inequality
	\[
	\frac{d}{dt}\sqrt{z_i}\leq \epsilon_i -2c \sqrt{z_i}.
	\]
	Solving the associated equality with $z_i(0)=0$ gives
\[	
	\sqrt{z_i(t)}=\frac{\epsilon_i  (1-e^{-2 c t} )}{2 c}, 
\]
	which provides the bound in the statement,  \eqref{error} for $c\neq 0$. The case $c=0$
	is obtained with a similar procedure. 
\end{proof}

\subsection{Monte Carlo estimation of errors}
\label{mee}

Here,  we discuss how to approximate $R$, $Q_i$, and
$\beta_i$ so that we can apply the estimate in Proposition \ref{properror}.

The term $R$ in \eqref{e0} can be estimated by Monte Carlo integration \cite{MR1999614}. For this, the most direct way would be to
consider the sample, $x_k$, of $ n_j$ points sampled at iteration $j$.  
These points are sampled
according to the probability distribution $\Gamma_{\theta_j}$. Then, the Monte Carlo estimator for $R$, $\hat R$, would be
\begin{equation}
\label{hatR}
\hat R = \left[\frac {\sum_{k=1}^{ n_j} (f(x_k)-q(x_k))^2}{ n_j}\right]^{1/2}.
\end{equation}
However, one of the algorithmic improvements we introduce, see Section \ref{sparsesec}, 
eliminates the  need  for sampling and least squares approximation in every iteration. 
Consequently, we may have to rely on a previous sample and its corresponding least squares approximation, associated with a different Gaussian, $\Gamma_{\bar \theta}$.

Accordingly, we rewrite the expression for $R$ as
$$
R=\left(\int_\Rr e^2\Gamma_{\theta_j}( x)dx\right)^{1/2} =\left(\int_\Rr e^2\frac{\Gamma_{\theta_j}(x)}{\Gamma_{\bar \theta}(x)}\Gamma_{\bar \theta}(x)dx\right)^{1/2} .
$$
Suppose we have  
a sample of $\bar n$ points sampled from $\Gamma_{\bar \theta}$ and the corresponding values of $f$ and $q$. The corresponding 
Monte Carlo estimator is
$$
\hat R =\left[\frac {\sum_{k=1}^{\bar n} (f(x_k)-q(x_k))^2 \ell(x_k)}{\bar n}\right]^{1/2},
$$
in which $\ell$ is the likelihood multiplier
\begin{equation}
\label{likemult}
\ell(x)=\frac{\Gamma_{\theta_j}(x)}{\Gamma_{\bar \theta}(x)}.
\end{equation}
Moreover, note that $\sum_{k=1}^{\bar n}  \ell(x_k)/{\bar n}$ is the Monte Carlo estimator for the integral
$$
\int_\Rr \frac{\Gamma_{\theta_j}(x)}{\Gamma_{\bar \theta}(x)}\Gamma_{\bar \theta}(x)dx = 1.
$$
Then, we can write
$$
\bar n \approx \sum_{k=1}^{\bar n}  \ell(x_k). 
$$
Thus, we introduce another estimator for $R$:
\begin{equation}
\label{hatrl}
\hat R =\left[\frac {\sum_{k=1}^{\bar n}  (f(x_k)-q(x_k))^2 \ell(x_k)}{\sum_{k=1}^{\bar n}  \ell(x_k)}\right]^{1/2}.
\end{equation}
This new estimator for $R$ has the advantage of exactly integrating constants.
This approach is called importance sampling \cite{MR1999614}.

The second term, $Q_i$, as defined in \eqref{qi}, is well-defined as long as $(\sigma^q_j)^2 < 2\sigma_j^2$ and $(\sigma^f_n)^2 < 2\sigma_j^2$. 
It can be computed exactly since it involves the product of polynomials with degree 4 or lower and Gaussians. 
Although the full expression of $Q_i$ is computable, we omit it here due to its length. Instead, we derive the second-order Taylor series near $\theta_j$, resulting in a more manageable expression for error bounds. Accordingly, the expressions used for error control are given by:
\begin{equation}
\label{q1}
Q_1\simeq 
\left[\frac{2}{\sigma_j^4} (\mu_q-\mu_f)^2+\frac{6}{\sigma_j^4} (\sigma_q-\sigma_f)^2\right]^{1/2}
\end{equation}
and
\begin{align}
	\label{q2}
	Q_2\simeq &\left[\frac{6}{\sigma_j^4} (\mu_q-\mu_f)^2+\frac{26}{\sigma_j^4} (\sigma_q-\sigma_f)^2
	\right]^{1/2}.
\end{align}

%
%

Finally, we estimate the supremum $\beta_i$ given by \eqref{bi}  in $[0,T]$
by its value at $t=0$; that is, we assume $\beta_i$ does not change substantially in the time intervals being used. 
We have the following Monte Carlo estimator
\begin{equation}
\label{hatbeta}
\hat \beta_i=\left|\frac {\sum_{k=1}^{\bar n}  (f(x_k)-q(x_k)) B_i(\theta_j, x_k)\ell(x_k)}{\sum_{k=1}^{\bar n}  \ell(x_k)}\right|,
\end{equation}
which has a standard deviation estimate
\begin{equation}
\label{sigmahatbeta}
\hat \sigma_{\beta_i}^2=
\frac {\sum_{k=1}^{\bar n}  (f(x_k)-q(x_k))^2 B_i^2(\theta_j, x_k)\ell(x_k)}{{\sum_{k=1}^{\bar n}  \ell(x_k)}}-\hat \beta_i^2.
\end{equation}
To estimate an integral using Monte Carlo sums, it is essential to consider the confidence interval defined by the empirical mean and the empirical standard deviation. For functions  that do not change sign, the empirical standard deviation, which scales as $\frac{1}{\sqrt{\bar n}}$, can often be disregarded, and the empirical mean can be used as the primary estimate. This approach was applied to the variable $R$. However, as $\beta_i$ represents the integral of a function with changing signs. In particular,  the least squares condition results in $\hat \beta_i$ becoming zero when $\bar \theta = \theta_j$. Consequently, to calculate an upper bound for $\beta_i$, both \eqref{hatbeta} and \eqref{sigmahatbeta} are required. We suggest the following estimate:
\begin{equation}
	\label{betai}
	\beta_i	\lesssim\bar \beta_i=\hat \beta_i+m\frac{\hat \sigma_{\beta_i}}{\sqrt{\bar n}},
\end{equation}
where $m$ determines the confidence interval for our estimates.

\subsection{Time step choice}
\label{timestepsec}

We now discuss the choice of the time step, $T_j$, at iteration $j$, which is constrained by two factors. The first factor involves the desired error bound for each iteration, while the second concerns the validity of the approximations for these bounds, as discussed in the previous sections.
We select the maximum permissible time step compatible with those bounds to optimize the algorithm's speed. 
This selection is pivotal for the algorithm's convergence.
	
Concerning the error bounds on iteration $j$, we proceed as follows. 
The error at each iteration is measured as a fraction of $\sigma_j$: 
\begin{equation}
	\label{gammabound}
	|\mu^f(T_j)-\mu^q(T_j)|\leq \gamma_1 \sigma_j,\qquad 
	|\sigma^f(T_j)-\sigma^q(T_j)|\leq \gamma_2 \sigma_j,
\end{equation}
where $\gamma_1$ and $\gamma_2$ are positive parameters.
 As $\sigma_j \to 0$, the error tolerances reduce, resulting in a higher precision as the algorithm converges.  
Moreover, the error scaling is convenient in the face of the denominators in 
\eqref{q1} and \eqref{q2}.
Using the aforementioned bounds, we obtain $Q_1, Q_2 \simeq O(\sigma_j^{-1})$. 
In our error bounds (i.e., \eqref{epsiloni}), these quantities appear multiplied by $R$, which can be bounded by $O(\sigma_j^3)$ and subsequently produce terms that vanish as $\sigma_j \to 0$ (see Section \ref{tsa}).

For the bounds \eqref{q1} and \eqref{q2} to be valid, we need
\begin{itemize}
	\item $(\sigma^f)^2\leq 2 \sigma_j^2$ and $(\sigma^q)^2\leq 2 \sigma_j^2$ so that the integral expression for $Q_i$ is well defined;
	\item $|\sigma^f-\sigma_j|$, $|\mu^f-\mu_j|$,  $|\sigma^q-\sigma_j|$, and $|\mu^q-\mu_j|$
	are small so that the Taylor approximation in \eqref{q1} and \eqref{q2} can be used. 
\end{itemize}
To enforce the second condition, we require 
\begin{equation}
	\label{mb}
	|\mu^q(T_j)-\mu_j|\leq \upsilon_1 \sigma_j. 
\end{equation}
and
\begin{equation}
	\label{sb}
	|\sigma^q(T_j)-\sigma_j|\leq \upsilon_2 \sigma_j, 
\end{equation}
with $\upsilon_1, \upsilon_2>0$. If \eqref{gammabound} and the previous two conditions hold, the triangle inequality gives
\[
|\mu^f(T_j)-\mu_j|\leq (\upsilon_1+\gamma_1)\sigma_j, 
\]
and 
\[
|\sigma^f(T_j)-\sigma_j|\leq (\upsilon_2+\gamma_2)\sigma_j.
\]
Thus, if $\upsilon_i$ and $\gamma_i$ are small enough  $|\sigma^f-\sigma_j|$, $|\mu^f-\mu_j|$,  $|\sigma^q-\sigma_j|$, and $|\mu^q-\mu_j|$ are also small. 
Moreover, requiring $1+\gamma_2+ \upsilon_2\leq \sqrt{2}$, we have 
$\sigma_f^2\leq 2 \sigma_j^2$ and $\sigma_q^2\leq 2 \sigma_j^2$. Thus, it is enough to 
select $T_j$ so that \eqref{gammabound}, \eqref{mb}, and \eqref{sb} hold.

We begin by addressing the bounds in \eqref{mb} and \eqref{sb} and the corresponding bounds for $T_j$. 
\begin{lem}
	Let
	\begin{equation}
		\label{tc1}
		T_\mu
		=
		\begin{cases}
			\frac{\upsilon_1 \sigma_j}{|b_j|} \qquad &\text{if}\quad  c_j=0\wedge b_j\neq 0\\
			-\frac{\log\left(1-\frac{2 |c_j|\sigma_j}{|b_j+2 c_j\mu_j|}\upsilon_1\right)}{2c_j}\qquad
			&\text{if} \quad c_j\neq 0\wedge b_j+2 c_j\mu_j\neq 0\wedge 1-\frac{2 |c_j|\sigma_j}{|b_j+2 c_j\mu_j|}\upsilon_1>0\\
					+\infty \qquad &\text{otherwise} 
		\end{cases}
	\end{equation}
and
	\begin{equation}
		\label{tc2}
		T_\sigma=
		\begin{cases}
			\frac{-\log (1-\upsilon_2\sgn(c_j))}{2c_j}\qquad &\text{if}\quad 1-\upsilon_2\sgn(c_j)>0\wedge c_j\neq 0\\
			+\infty\qquad &\text{otherwise}
		\end{cases}
	\end{equation}
	Suppose that 
	\begin{equation}
	\label{ubs}
	T_j\leq  \min\{T_\mu, T_\sigma \}, 
	\end{equation}
	then \eqref{mb} and \eqref{sb} hold for $0\leq t\leq  T_j$.

\end{lem}

\begin{proof}
	At iteration $j$,
	we determine a step size $T_j$
	such that  \eqref{mb}, \eqref{sb} hold.
	Let $(\mu, \sigma)$ solve \eqref{quadgf} with initial conditions $\mu_j$ and $\sigma_j$. Since \eqref{mb} and \eqref{sb} hold at $t=0$, there exists a $0<\bar T\leq +\infty$ where these conditions are valid for $0\leq t < \bar T$. At the maximal time $\bar T$, if finite, equality holds in \eqref{mb} and \eqref{sb}.
	To determine $T_j$, we solve the equality cases in \eqref{mb} and \eqref{sb}. Each positive solution provides an upper bound for $T_j$. If no positive solutions exist, the corresponding identity does not constrain $T_j$, and either \eqref{mb} or \eqref{sb} remains strict for all positive times.	
	
	Since the inequality  \eqref{mb} involves an absolute value, we need to consider two cases corresponding to 
	$T^1_j$ and $T^2_j$ that we determine next. 
	When $c_j=0$ in \eqref{quadgf}, 
	\eqref{quadgfsol} yields $\mu(t)=b_j t+\mu_j$. Accordingly, 
	the equality case corresponding to \eqref{mb} becomes
	\[
	|b_j T_j^1|= \upsilon_1 \sigma_j, 
	\]
	which gives 
	\begin{equation*}
	T_j^1= \frac{\upsilon_1 \sigma_j}{|b_j|}, 
	\end{equation*}
	if $b_j\neq 0$, $T_j^1=+\infty$ if $b_j=0$.  For convenience, we set $T_j^2=T^1_j$ since here we do not need to consider two cases. 
	
	If $c_j\neq 0$, the equality case corresponding to \eqref{mb} becomes 
	\begin{equation}
		\label{ineq1}
	\left|\frac{ b(e^{-2 c_j T_j}- 1)}{2 c_j}+ \mu_j (e^{-2 c_j T_j}-1)\right|= \upsilon_1\sigma_j.
	\end{equation}
	
	Solving the preceding equality results in two values for $T_j$:
	\[
	T_j^1=\frac{1}{2c_j}\log \left(\frac{b_j+2 c_j \mu_j}{b_j+2 c_j \mu_j+ 2 c_j \sigma_j  \upsilon_1 }\right),
	\]
	if
	\[
	\frac{b_j+2 c_j \mu_j}{b_j+2 c_j \mu_j+ 2 c_j \sigma_j  \upsilon_1 }> 0,
	\]
	and 
	\[
	T_j^2=\frac{1}{2c_j}\log \left(\frac{b_j+2 c_j \mu_j}{b_j+2 c_j \mu_j- 2 c_j \sigma_j  \upsilon_1 }\right)
	\]
	if
	\[
	\frac{b_j+2 c_j \mu_j}{b_j+2 c_j \mu_j- 2 c_j \sigma_j  \upsilon_1 }> 0.
	\]
	For each of these values, we have three possibilities:
	\begin{enumerate}
		\item $T_j^i$ is not defined because the logarithm does not exist;
		\item $T_j^i<0$, in which case the equality is not achieved for a positive time and there is no upper bound for $T_j^i$;
		\item $T_j^i>0$, in which case the equality is achieved at a finite positive time. 
	\end{enumerate}	
	In the first two cases, we redefine $T_j^i=+\infty$. In the last case, we keep $T_j^i$ unchanged to be the obtained positive value. Finally, considering only the constraints on $\mu$, we take
	$$
	T_{\mu} =\min \left\{T_j^1,T_j^2\right\}.
	$$
	
	Similarly, using \eqref{quadgfsol}, condition \eqref{sb} becomes
	\[
	\left|e^{-2 c_j T_j}-1\right|\leq \upsilon_2.
	\]
	If $c_j=0$, there is no constraint, and $T_\sigma=+\infty$. Otherwise, 
	solving the equality case as before gives
	\[
	T_\sigma=\frac{-\log (1-\upsilon_2\sgn(c_j))}{2c_j}
	\]
	if $1-\upsilon_2\sgn(c_j)>0$; otherwise, we also set $T_\sigma=+\infty$.	
	
If \eqref{ubs} holds,  $T_j\leq T_\mu$ and 	 $T_j\leq T_\sigma$. Hence,
	 \eqref{mb} and \eqref{sb} hold for $0\leq t\leq T_j$.
	\end{proof}

	We now address the bounds in \eqref{gammabound}.
	\begin{lem}
Let $\epsilon_i$ be given by  \eqref{epsiloni}.  Define
			\begin{equation}	
		\label{tei}
		T_{\epsilon_i}=
		\begin{cases}
			\frac{-\log \left(1-2c_j \frac {\gamma_i\sigma_j}{\epsilon_i}
				\right)}{2c_j},\qquad & c_j\neq 0\wedge 1-2c_j \frac {\gamma_i\sigma_j} {\epsilon_i}>0\\
			\frac{\gamma_i\sigma_j}{\epsilon_i}	& c_j=0\\
			+\infty&c_j\neq 0\wedge 1-2c_j \frac {\gamma_i\sigma_j} {\epsilon_i}\leq 0.
		\end{cases}	
	\end{equation}
	Suppose that 
		\begin{equation}
			\label{bt}
			T_j \leq \min(T_{\epsilon_1}, T_{\epsilon_2}),
		\end{equation}
	then \eqref{gammabound} holds for $0\leq t\leq T_j$. 	
	\end{lem}
	\begin{proof}
		We use this bound in \eqref{epsiloni}, the expression \eqref{error}, and 
		the condition $\sqrt{z_i(T_j)}= \gamma_i\sigma_j$ to get the result.
	\end{proof}

In our algorithm, we cannot compute $\epsilon_i$ from  \eqref{epsiloni} since the bound there is implicit. 
Thus, we use the following approximations. 
First, 	we use the Taylor approximations for $Q_1$ and $Q_2$. Accordingly, while \eqref{gammabound}
holds, we have 
\begin{equation}
	\label{bqi}
	\hat Q_1=\frac{\sqrt{2 \gamma_1^2+6 \gamma_2^2}}{\sigma_j}\qquad 
	\hat Q_2=\frac{\sqrt{6 \gamma_1^2+26 \gamma_2^2
	}}{\sigma_j}. 
\end{equation}
Then, we replace
$R$  with $\hat R$ from \eqref{hatrl}, and 
$\beta_i$ with $\bar \beta_i$ from \eqref{betai}; that is, we define
\begin{equation}
	\label{hatepsilon}
	\hat \epsilon_i=\hat R \hat Q_i+\bar \beta_i, 
\end{equation}
and compute an approximation to $T_{\hat \epsilon_i}$ of  $T_{\epsilon_i}$.

Finally,  we select the time step as
	\begin{equation}
		\label{timestep}
		T_j=\min(T_\mu, T_\sigma, T_{\hat \epsilon_1}, T_{\hat \epsilon_2}).
	\end{equation}
The subsequent section provides proof of our algorithm's convergence using this approximation for
 $T_j$.

\section{Convergence}
\label{tsa}

We investigate the convergence of $\mu_j$ and
$\sigma_j$ as $j\to \infty$ at a non-degenerate minimum of $f$,
i.e., a minimum point where $f$ has a positive second derivative.
Without loss of generality, 
we suppose that this minimum is at the origin and that $f$ vanishes there. 
Accordingly, we assume the following.
\begin{hyp}
	\label{chyp}
	$f\in C^2_p(\Rr)$ satisfies
	\begin{itemize}
		\item  $f(0)=0$, $f'(0)=0$,  and $f''(0)=2c>0$;
		\item  for all $x\in \Rr$,  $|f(x)-c x^2|\leq C |x|^3$ for some constant $C>0$. 
	\end{itemize}
\end{hyp}

The second condition is a global bound for $f$. This bound simplifies the analysis as it allows to estimate various quantities.
Moreover, we do not apply acceleration techniques detailed in Sections \ref{adaptivitysec} and \ref{sparsesec}. Thus, we 
 maintain  $n_j=n$ as a constant number of sample points.
 Our result does not require the independence of samples
 between iterations.  Therefore, it applies to the case where 
we sample at every step, using the rejection sampling strategy outlined in Section
\ref{rejesec}.

The key  to establishing convergence is to 
show that, with high probability, both $b_j$ and $c_j$ are close to $0$ and
$c$, respectively. This follows from the norm estimates in Proposition \eqref{abcest}.
Because  $b_j$ is close to zero and $c_j$ is
close to $c$,  \eqref{quadgfsoli} is close to 
\[
\mu_{j+1}= \mu_j e^{-2 c T_j}, \quad \sigma_{j+1}=\sigma_j e^{-2 c T_j}, 
\]
whose solutions converge exponentially to zero if $T_j$ remains bounded below. 
We show this is the case in Proposition \ref{513}, using 
the error estimate in Proposition \ref{511}.
Hence, in Proposition \ref{p410}, we obtain $\mu_j, \sigma_j\to 0$. The precise 
convergence result is the following theorem. 

\begin{teo}
	\label{mainteo}
	Suppose $f$ satisfies Assumptions \ref{a2} and \ref{chyp}. Let $\upsilon_1$ and $\upsilon_2$ be as in \eqref{mb} and \eqref{sb}. Let $p\geq  1$ and suppose $n> 6p+2$. Set $r=\frac{2p}3$.
	Fix $K>0$
	and
	\[
	\Theta= \max\left\{\left(1-\frac{\upsilon_1}{4K}\right), 1-\upsilon_2\right\}<1.
	\]
	Then, there is a constant $\bar C$ such that for any $\delta$ small enough, 
	with
	probability larger than
	\[
	1- \bar C\delta^r
	\]
	for any initial condition satisfying 
	\[
	\frac{|\mu_0|}{\sigma_0}\leq \frac{K}{2}, \qquad \sigma_0\leq \delta,
	\]   
   the sequence    $(\mu_j, \sigma_j)$ obtained from Algorithm \ref{algo} (without restart, adaptivity or sparse  sampling) 
	satisfies
	\[
	\frac{|\mu_j|}{\sigma_j}\leq K, \qquad \sigma_j\leq \delta \Theta^j
	\]
	for all $j$.
	
\end{teo}

We present the proof of this theorem in section \ref{mainteoproof} after some preliminary results. Before that, we remark that 
because $0<\Theta<1$, the preceding theorem implies that $\sigma_j\to 0$. But then, 
the condition $	\frac{|\mu_j|}{\sigma_j}\leq K$ gives $\mu_j\to 0$ as well. 


To simplify the notation, throughout this section, we write the same letter $C$ to denote any positive constant  depending only  on the data; that is, depending only on
the constants in Assumption \ref{chyp}, on the parameters $\upsilon_1$ and $\upsilon_2$, 
on the constants $K$, $p$ and $n$  in the theorem statement, but independent on $j$, $\mu_j$, $\sigma_j$ or on the particular
sample $\omega$ in the underlying probability space $\Omega$.
In particular, 
the letter $C$ might represent different constants in consecutive
expressions.

\subsection{Least squares convergence estimates}

In our algorithm, the coefficients $a_j$, $b_j$, and $c_j$ in the least squares approximation of $f$ are random variables. In particular, with small probabilities, they can deviate substantially 
from the values corresponding to the quadratic part of $f$.  
We aim to show that for small $\mu_j$ and $\sigma_j$, the coefficients $a_j$, $b_j$, and $c_j$ converge towards the values corresponding to the quadratic part of $f$, specifically, $a_j$ and $b_j$ approach zero and $c_j$ approaches $c$.

We introduce the following notation. 
For a random variable $Y$ and $p\geq 1$, we denote $\|Y\|_p=\left(E\left[|Y|^p\right]\right)^{1/p}$. Similarly, 
for a function $g(x_1, \hdots x_n)$,
$\|g\|_{p}$ represents the ${L^p}$-norm with respect to the measure $\Gamma_{\mu,\sigma}(x_1)\Gamma_{\mu,\sigma}(x_2)\ldots \Gamma_{\mu,\sigma}(x_n)$; that is,
$$
\|g\|_{p}^p = \int |g |^p \Gamma_{\mu,\sigma}(x_1)\Gamma_{\mu,\sigma}(x_2)\ldots \Gamma_{\mu,\sigma}(x_n) dx_1dx_2\ldots dx_n.
$$
For a deterministic quantity,  $\bar Y$, we write $\bar Y=O(h)$ if there exists a constant $C$ such that $|\bar Y|\leq C|h|$ for all sufficiently small $h$. We say that a function $s:\Rr^n\to \Rr$ is symmetric if it is
invariant under any permutation of the components of its argument.  

Firstly, we note that the least squares problem, given sample points $x_1, \hdots, x_n$, is equivalent to solving the following linear system of equations.
\begin{equation}
\label{sys1}
\begin{bmatrix}
n&\sum_{i=1}^{n} x_i&\sum_{i=1}^{n} x_i^2\\
\sum_{i=1}^{n} x_i&\sum_{i=1}^{n} x_i^2&\sum_{i=1}^{n} x_i^3\\
\sum_{i=1}^{n} x_i^2&\sum_{i=1}^{n} x_i^3&\sum_{i=1}^{n} x_i^4
\end{bmatrix}
\begin{bmatrix}
a_j\\b_j\\c_j
\end{bmatrix}
=
\begin{bmatrix}
\sum_{i=1}^{n} f(x_i)\\
\sum_{i=1}^{n} x_i f(x_i)\\
\sum_{i=1}^{n} x_i^2 f(x_i)
\end{bmatrix}.
\end{equation}
Let $x=(x_1,x_2,\ldots,x_n)\in \Rr^n$. 
Assume that 
\begin{equation}
\label{det}
d(x)=\det\begin{bmatrix}
	{n}&\sum_{i=1}^{n} x_i&\sum_{i=1}^{n} x_i^2\\
	\sum_{i=1}^{n} x_i&\sum_{i=1}^{n} x_i^2&\sum_{i=1}^{n} x_i^3\\
	\sum_{i=1}^{n} x_i^2&\sum_{i=1}^{n} x_i^3&\sum_{i=1}^{n} x_i^4
\end{bmatrix}
\end{equation}
does not vanish.
If $f$ is a quadratic function, $f=\hat a+\hat b x +\hat c x^2$, the solution of \eqref{sys1} is $a_j=\hat a$, $b_j=\hat b$, and $c_j=\hat c$. 
Here, we are interested in the case where
$f=c x^2+\tilde f$, where $\tilde f=O(|x|^3)$, according to Assumption \ref{chyp}. 
Thus, 
because \eqref{sys1} is linear, the solution is of the form
 $a_j=\tilde a_j$, $b_j=\tilde b_j$ and $c_j=c+\tilde c_j$, where the coefficients
 $\tilde a_j$, $\tilde b_j$, and $\tilde c_j$ solve \eqref{sys1} with $f$ replaced by 
 $\tilde f$.  
 Our goal is to estimate $\|a_j\|_p=\|\tilde a_j\|_p$, $\|b_j\|_p=\|\tilde b_j\|_p$, and $\|c_j-c\|_p=\|\tilde c_j\|_p$. For this, we need
 quantitative estimates on the solution of \eqref{sys1} when $f$ is replaced by $\tilde f$. 
 These estimates depend on $n$. For instance, when $n<3$, $d(x)$ vanishes and \eqref{sys1} lacks a unique solution.  Consequently, quantitative bounds cannot be established in this case. The condition $n>6p+2$ in Theorem \ref{mainteo} ensures that $n$ is large enough for quantitative bounds on the coefficients to hold. 
 
We start by investigating how  $\tilde a_j$, $\tilde b_j$ and $\tilde c_j$ are calculated from least squares. 
Let $C_{kl}$ be the cofactor $(kl)$ of the matrix in \eqref{sys1}.  
For example, 
\begin{align*}
	C_{21} & = - \sum_{i=1}^n x_i \sum_{i=1}^n x_i^4+\sum_{i=1}^n x_i^3\sum_{i=1}^n x_i^2.
\end{align*}
By direct inspection, we see that the $C_{kl}$ is a homogeneous polynomial
(in the sample points)
 of degree $8-k-l$.  
For further reference, we record the explicit form of the
solutions:   
$$
\tilde a_j = \frac{\sum_{i=1}^n \tilde f(x_i) C_{11}+\sum_{i=1}^n x_i\tilde f(x_i)C_{21}+ \sum_{i=1}^n x_i^2\tilde f(x_i)C_{31}} {d(x)},
$$ 
$$
\tilde b_j = \frac{\sum_{i=1}^n  \tilde f(x_i) C_{12}+\sum_{i=1}^n x_i\tilde f(x_i)C_{22}+ \sum_{i=1}^n x_i^2\tilde f(x_i)C_{32}} {d(x)},
$$ 
$$
\tilde c_j = \frac{\sum_{i=1}^n \tilde f(x_i) C_{13}+\sum_{i=1}^n x_i\tilde f(x_i)C_{23}+ \sum_{i=1}^n x_i^2\tilde f(x_i)C_{33}} {d(x)}.
$$ 
Then, we have the following estimates.
\begin{lem}
\label{lem52}
Suppose $p\geq 1$. 
Let $\frac 1 p =\frac 1{p_1}+\frac 1 {p_2}+\frac{1}{p_3}$. Then
\begin{align*}
\|\tilde a_j\|_p\leq &\left\|\frac 1 {d}\right\|_{p_1}\left(
\left\|\sum_{i=1}^n \tilde f(x_i)\right\|_{p_2}\|C_{11}\|_{p_3}+\right.\\
&\left.+\left\|\sum_{i=1}^n x_i \tilde f(x_i)\right\|_{p_2}\|C_{21}\|_{p_3}+
\left\|\sum_{i=1}^n x_i^2\tilde  f(x_i)\right\|_{p_2}\|C_{31}\|_{p_3}
\right),
\end{align*}
\begin{align*}
	\|\tilde b_j\|_p\leq &\left\|\frac 1 {d}\right\|_{p_1}\left(
	\left\|\sum_{i=1}^n \tilde f(x_i)\right\|_{p_2}\|C_{12}\|_{p_3}+\right.\\
	&\left.+\left\|\sum_{i=1}^n x_i \tilde f(x_i)\right\|_{p_2}\|C_{22}\|_{p_3}+
	\left\|\sum_{i=1}^n x_i^2 \tilde f(x_i)\right\|_{p_2}\|C_{32}\|_{p_3}
	\right),
\end{align*}
and
\begin{align*}
	\|\tilde c_j\|_p\leq &\left\|\frac 1 {d}\right\|_{p_1}\left(
	\left\|\sum_{i=1}^n \tilde f(x_i)\right\|_{p_2}\|C_{13}\|_{p_3}+\right.\\
	&\left.+\left\|\sum_{i=1}^n x_i \tilde f(x_i)\right\|_{p_2}\|C_{23}\|_{p_3}+
	\left\|\sum_{i=1}^n x_i^2 \tilde f(x_i)\right\|_{p_2}\|C_{33}\|_{p_3}
	\right).
\end{align*}

\end{lem}
\begin{proof}The inequalities follow from H\"older's inequality. 
\end{proof}

The subsequent lemma provides an explicit formula for the determinant of the least squares matrix in \eqref{sys1}. 

\begin{lem} Let $d$ be given by \eqref{det}. For $x_1,x_2,\ldots, x_n \in \Rr$,
\begin{equation}
\label{dp}
d(x_1,x_2,\ldots,x_n)=s(x_1,x_2,\ldots,x_n),
\end{equation}
where 
\[
s(x_1,x_2,\ldots,x_n) \equiv \frac 1 6 \sum_{k,l, m=1}^n (x_k-x_l)^2(x_k-x_m)^2(x_l-x_m)^2.
\]
\end{lem}
\begin{proof}
The determinant, $d$, is a symmetric polynomial of degree 6 in $x_1,x_2,\ldots, x_n$, as is $s(x_1,x_2,\ldots,x_n)$. Thus, it suffices to show that
\[
d(x_1,\ldots,x_6,0,\ldots0)=s(x_1,\ldots,x_6,0,\ldots0).
\]
The prior identity is equivalent to
\begin{equation}
\label{det6}
\det\begin{bmatrix}
n&\sum_{i=1}^6 x_i&\sum_{i=1}^6 x_i^2\\
\sum_{i=1}^6 x_i&\sum_{i=1}^6 x_i^2&\sum_{i=1}^6 x_i^3\\
\sum_{i=1}^6 x_i^2&\sum_{i=1}^6 x_i^3&\sum_{i=1}^6 x_i^4\\
\end{bmatrix}=s(x_1,\ldots,x_6,0,\ldots0).
\end{equation}
Setting $x_i=0$ for $i>6$ in $s(x_1,x_2,\ldots,x_n)$ yields
\[
s(x_1,\ldots,x_6,0,\ldots0) = \frac 1 6 \sum_{k,l, m=1}^6 (x_k-x_l)^2(x_k-x_m)^2(x_l-x_m)^2+ 3 (n-6) \sum_{k,l=1}^6 x_k^2 x_l^2 (x_k-x_l)^2,
\]
which can be checked directly to be equal to the determinant in \eqref{det6}.
\end{proof}

Next, we discuss $L^p$ bounds for $\frac 1 d$ when $\mu=0$ and $\sigma=1$. Lemma \ref{l55} addresses the general case using a scaling argument. 

\begin{pro} Let $p\geq 1$. Suppose that $n> 6 p+2$ and let $d$ be given by \eqref{det}.
	Then, 
\begin{equation}
\label{detp}
\int_{\Rr^n} \frac{1}{d(x)^p} \Gamma(x_1)\hdots \Gamma(x_n)dx_1\hdots d x_n=\bar C_{n,p}<\infty,
\end{equation}
 where $\Gamma=\Gamma_{0,1}$ and $\bar C_{n,p}$ is a positive constant dependent on $n$ and $p$.
\end{pro}
\begin{proof}
Given the proposition's conditions, we assume $n\geq 3$ throughout this proof. Note that for $n<3$, the determinant $d(x)$ is identically zero, rendering the integral in \eqref{detp} infinite.	

Let $Z\subset \Rr^n$ be the set where $d$ vanishes.
 By identity \eqref{dp}, $x\in Z$ if and only if the following property holds:
for any choice of three indices, $k,l,m$, with $1\leq k< l< m\leq n$, $x_k=x_l$ or $x_k=x_m$ or $x_l=x_m$. 
Consequently, any point in $Z$ has at least two identical coordinates. Therefore, $Z$  is contained in the union of the finite collection of hyperplanes defined by $x_k=x_l$ for some $1\leq k<l\leq n$. Thus, $Z$ has measure zero.

Observe that 
\[
Z^c=\left\{
x\in \Rr^n: \exists_{k,l, m: 1\leq k< l< m\leq n}: |x_k-x_l|\neq 0\wedge |x_k-x_m| \neq 0\wedge|x_l-x_m| \neq 0
\right\}.
\]
We organize $Z^c$ in the following sets. For $r\in\Nn$, let
\begin{align*}
G_r=&\left\{
x\in \Rr^n: \exists_{k,l, m: 1\leq k< l< m\leq n}: |x_k-x_l|\geq 2^{-r}\wedge |x_k-x_m| \geq 2^{-r}\wedge|x_l-x_m| \geq 2^{-r}
\right\}.
\end{align*}
For any $x\in Z^c$, there exists $r$ such that $x\in G_r$. 
Furthermore,  in $G_r$, we have the following lower bound for $d$:
\begin{equation}
\label{lbdet}
d\geq C_n 2^{-6r},
\end{equation}
where $C_n$ is a positive constant that depends on $n$.

We note that $G_r$ is an increasing sequence of sets; that is, $G_{r}\subset G_{r+1}$. 
Accordingly, we define $\tilde G_0=G_0$ and, for $r>0$, set
\begin{equation}
\label{tj}
\tilde G_r=G_r \cap G^c_{r-1}.
\end{equation}
The sets $\tilde G_r$ are disjoint and cover $Z^c$. Thus, because $Z$ has measure zero,  the Monotone Convergence Theorem gives
\begin{align*}
\int_{\Rr^n} \frac{1}{d(x)^p} \Gamma(x_1)\hdots \Gamma(x_n)dx_1\hdots d x_n
&=\sum_{r=0}^{+\infty}
\int_{\tilde G_r} \frac{1}{d(x)^p} \Gamma(x_1)\hdots \Gamma(x_n)dx_1\hdots d x_n.
\end{align*}
Using \eqref{lbdet}, we have
\begin{align}
\label{sum}
\int_{\Rr^n} \frac{1}{d(x)^p} \Gamma(x_1)\hdots \Gamma(x_n)dx_1\hdots d x_n
&\leq \tilde C_n 
\sum_{r=0}^\infty2^{6rp} \int_{\tilde G_r} \Gamma(x_1)\hdots \Gamma(x_n)dx_1\hdots d x_n.
\end{align}

Note that by \eqref{tj}, we have
\[
\int_{\tilde G_r} \Gamma(x_1)\hdots \Gamma(x_n)dx_1\hdots d x_n
\leq 
\int_{G_{r-1}^c} \Gamma(x_1)\hdots \Gamma(x_n)dx_1\hdots d x_n.
\]
Because
 \begin{align*}
 G^c_{r-1}=&\left\{
 x\in \Rr^n: \forall_{k,l, m: 1\leq k< l< m\leq n}: \right.\\
 &\left.
  |x_k-x_l|<2^{-(r-1)}\vee |x_k-x_m|<2^{-(r-1)}\vee|x_l-x_m|<2^{-(r-1)}
 \right\},
 \end{align*}
we decompose $G^c_{r-1}$ as the union of two sets $A_{r-1}$ and $B_{r-1}$ defined as follows. Let $ A_{r-1}$ be
 \begin{align*}
 A_{r-1}=&\left\{
 x\in G^c_{r-1}: \forall_{k: 1< k\leq n}: |x_1-x_k|<2^{-(r-1)}
 \right\};
 \end{align*}
 that is, $A_{r-1}$ is the set of all $x\in G^c_{r-1}$ with all coordinates close to $x_1$. Let $B_{r-1}$ be the set of all $x\in G^c_{r-1}$ not in $A_{r-1}$:
 \begin{align*}
 B_{r-1}=&\left\{
 x\in G^c_{r-1}:\exists_{l: 1<l\leq n, |x_1-x_l|\geq 2^{-(r-1) }} \forall_{k: 1< k\leq n,k\ne l}:\right.\\
 &\left. |x_1-x_k|<2^{-(r-1)}\vee |x_l-x_k|<2^{-(r-1)}
 \right\}.
 \end{align*}
Then, because $A_{r-1}\subset\left\{
x\in \Rr^n: \forall_{k: 1< k\leq n}: |x_1-x_k|<2^{-(r-1)}
\right\}$, we have
\begin{align}
\label{Aest}
\notag
\int_{A_{r-1}}  \Gamma(x_1)\hdots \Gamma(x_n)dx
&\leq \int_{\Rr}\int_{|x_1-x_2|<2^{-(r-1)}} \hdots \int_{|x_1-x_n|<2^{-(r-1)}} 
\Gamma(x_1)\hdots \Gamma(x_n)dx
\\
&\leq C 2^{-(r-1)(n-1)},
\end{align}
since, for each $k$ with $1<k\leq n$, 
\[
\int_{|x_1-x_k|<2^{-(r-1)}}\Gamma(x_k)dx_k\leq C 2^{-(r-1)},
\]
where the constants above are positive and depend only on $n$.

For $l>1$, we define
\begin{align*}
B^l_{r-1}=\left\{
x\in \Rr^n: \forall_{k: 1< k\leq n,k\ne l}:
|x_1-x_k|<2^{-(r-1)}\vee |x_l-x_k|<2^{-(r-1)}
\right\}.
\end{align*}
Note that $B_{r-1}\subset \cup_{l=2}^nB^l_{r-1}$. Then,
\begin{align}
\label{Bsplit}
\int_{B_{r-1}}  \Gamma(x_1)\hdots \Gamma(x_n)dx
&\leq \sum_{l=2}^n \int_{B^l_{r-1}}  \Gamma(x_1)\hdots \Gamma(x_n)dx.
\end{align}

Let $H[x_1,x_l] = \{
y\in \Rr: 
|x_1-y|<2^{-(r-1)}\vee |x_l-y|<2^{-(r-1)}\}$ and denote by $\tilde {x} = (x_2,\hdots, x_{l-1},x_{l+1},\hdots,x_n)$. Then,
\begin{align}
\label{Best}
\notag
\int_{B^l_{r-1}}  \Gamma(x_1)\hdots \Gamma(x_n)dx & =
\int_{\Rr} \int_{\Rr} \int_{H[x_1,x_l]} \hdots \int_{H[x_1,x_l]} \Gamma(x_1)\hdots \Gamma(x_n)d\tilde{x} dx_l dx_1\\
&\leq C 2^{-(r-1)(n-2)}.
\end{align}
Combining \eqref{Aest}, \eqref{Bsplit} and \eqref{Best}, we obtain
\[
\int_{\tilde G_r} \Gamma(x_1)\hdots \Gamma(x_n)dx_1\hdots d x_n
\leq 
C 2^{-r (n-2)}, 
\]
for a suitable positive constant $C$ that depends only on $n$. 

Finally, using the preceding estimate in \eqref{sum},  
\begin{align*}
\int_{\Rr^n} \frac{1}{d(x)^p} \Gamma(x_1)\hdots \Gamma(x_n)dx_1\hdots d x_n
&\leq C
\sum_{r=0}^\infty 2^{6rp-r(n-2)} .
\end{align*}
Thus, the sum is finite provided $n> 6 p+2$, which means \eqref{detp} holds.
\end{proof}

\begin{lem}
\label{l55}	
Let $p\geq 1$. Suppose that $n> 6 p+2$ and let $d$ be given by \eqref{det}.
Then, 
\[
\left\|\frac{1}{d}\right\|_{p}\leq \frac{C}{\sigma^6},
\]
where  $\|\cdot \|_{p}$ represents the $L^p$ norm with respect to $\Gamma_{\mu, \sigma}$ and $C$ is a positive real constant depending only on $n$.
\end{lem}
\begin{proof}
We recall that 
\[
\left\|\frac{1}{d}\right\|^p_{p} = \int_{\Rr^n} \left[\frac{\frac 1 6}{\sum_{k,l, m=1}^n (x_k-x_l)^2(x_k-x_m)^2(x_l-x_m)^2}\right]^p \Gamma_{\mu,\sigma}(x_1)\hdots \Gamma_{\mu,\sigma}(x_n)dx_1\hdots d x_n,
\]
where we used \eqref{dp}. Two changes of variables, namely $x_i \rightarrow x_i-\mu$ followed by $x_i \rightarrow \sigma x_i$, result in
\begin{align*} 
\left\|\frac{1}{d}\right\|^p_{p}  &=
\int_{\Rr^n} \left[\frac{\frac 1 6}{\sum_{k,l, m=1}^n (x_k-x_l)^2(x_k-x_m)^2(x_l-x_m)^2}\right]^p \Gamma_{0,\sigma}(x_1)\hdots \Gamma_{0,\sigma}(x_n)dx_1\hdots d x_n\\
&= \frac{1}{\sigma^{6p}}
\int_{\Rr^n} \left[\frac{\frac 1 6}{\sum_{k,l, m=1}^n (x_k-x_l)^2(x_k-x_m)^2(x_l-x_m)^2}\right]^p \Gamma(x_1)\hdots \Gamma(x_n)dx_1\hdots d x_n.
\end{align*}
Finally, from \eqref{detp}, we have that 
\[
\left\|\frac{1}{d}\right\|_{p} \leq \frac{C}{\sigma^6},
\]
where $C$ is a positive constant dependent on $n$ and $p$.
\end{proof}

The following proposition is 
a corollary of Lemma \ref{hplem} and yields estimates needed to apply
Lemma \ref{lem52} in the proof of Proposition \ref{abcest}. 

\begin{pro}
\label{bcor}
We have, for all $p\geq 1$,
 \[
 \|C_{ik}\|_{p}=O( |\mu_j|^\alpha +\sigma_j^\alpha),
 \]
 where $\alpha=8-i-j$, and, for $\beta=0,1,2$, 
 \[
 \left\| \sum x_i^\beta \tilde f(x_i)\right\|_{p}=O( |\mu_j|^{3+\beta}+\sigma_j^{3+\beta}),
 \]
where  $\|\cdot \|_{p}$ represents the $L^p$ norm with respect to $\Gamma_{\mu_j, \sigma_j}$.
\end{pro}
\begin{proof}
The estimates for $C_{ik}$  in the corollary follow directly from the fact that 
each of these cofactors is homogeneous of degree $\alpha=8-i-j$. The estimates for
$ \sum x_i^\beta \tilde f(x_i)$ follow from the homogeneous bound 
$\left| \sum x_i^\beta \tilde f(x_i) \right|\leq  C\sum |x_i|^{\beta+3}$, which results from Asssumption \ref{chyp}. 
\end{proof}

Finally, we collect the previous bounds and obtain the main estimates for the least squares
coefficients. 
\begin{pro}
\label{abcest}
Suppose $f$ satisfies Assumptions \ref{a2} and \ref{chyp}. Let $p\geq 1$ and suppose that  $n>6p+2$. 
At iteration $j$, consider 
a corresponding sample of $n$ points drawn out of $(\mu_j, \sigma_j)$.
 Then
\begin{equation}
	\label{aest}
	\|a_j\|_p=
	O\left(\frac{|\mu_j|^9}{\sigma_j^6}+\sigma_j^3\right), 
\end{equation}
\begin{equation}
	\label{best}
	\|b_j\|_p=
	O\left(\frac{|\mu_j|^8}{\sigma_j^6}+\sigma_j^2\right), 	
\end{equation}
and
\begin{equation}
	\label{cest}
	\|c_j-c\|_p=	
	O\left(\frac{|\mu_j|^7}{\sigma_j^6}+\sigma_j\right),
\end{equation}
where  $\|\cdot \|_{p}$ represents the $L^p$ norm with respect to $\Gamma_{\mu_j, \sigma_j}$.
\end{pro}
\begin{proof}
The result follows from Lemma \ref{lem52}, combined with Lemma
\ref{l55} and 
the estimates in Proposition  \ref{bcor}. 
Accordingly, we have
\begin{align*}
\|a_j\|_p=O\left(\frac{1}{\sigma_j^6} \right)
&\left[
O\left(|\mu_j|^3+\sigma_j^3 \right)O\left(|\mu_j|^6+\sigma_j^6 \right)
+
O\left(|\mu_j|^4+\sigma_j^4 \right)O\left(|\mu_j|^5+\sigma_j^5 \right)\right.\\
&\left.+
O\left(|\mu_j|^5+\sigma_j^5 \right)O\left(|\mu_j|^4+\sigma_j^4 \right)
\right], 
\end{align*}
\begin{align*}
	\|b_j\|_p=O\left(\frac{1}{\sigma_j^6} \right)
	&\left[
	O\left(|\mu_j|^3+\sigma_j^3 \right)O\left(|\mu_j|^5+\sigma_j^5 \right)
	+
	O\left(|\mu_j|^4+\sigma_j^4 \right)O\left(|\mu_j|^4+\sigma_j^4 \right)\right.\\
	&\left.+
	O\left(|\mu_j|^5+\sigma_j^5 \right)O\left(|\mu_j|^3+\sigma_j^3\right)
	\right], 
\end{align*}
and
\begin{align*}
	\|c_j-c\|_p=O\left(\frac{1}{\sigma_j^6} \right)
	&\left[
	O\left(|\mu_j|^3+\sigma_j^3 \right)O\left(|\mu_j|^4+\sigma_j^4 \right)
	+
	O\left(|\mu_j|^4+\sigma_j^4 \right)O\left(|\mu_j|^3+\sigma_j^3 \right)\right.\\
	&\left.+
	O\left(|\mu_j|^5+\sigma_j^5 \right)O\left(|\mu_j|^2+\sigma_j^2\right)
	\right].
\end{align*}
Now, we 
observe 
that the terms that multiply $O\left(\frac{1}{\sigma_j^6}\right)$ can be bounded by homogeneous
functions in $\mu_j$ and $\sigma_j$ which have degrees of $9$, $8$ and $7$, respectively. 
Accordingly, using Lemma \ref{pseudonorm}, they can be bounded by $O\left(|\mu_j|^9+\sigma_j^9 \right)$, 
$O\left(|\mu_j|^8+\sigma_j^8\right)$, and $O\left(|\mu_j|^7+\sigma_j^7 \right)$.
Thus, 
\eqref{aest}, \eqref{best},  and \eqref{cest} follow.
\end{proof}

\subsection{Additional convergence estimates}

We now consider the random variables in the time-step selection.
In particular, we examine the various components of \eqref{hatepsilon} which determine
$T_{\hat \epsilon_i}$. Our goal is to show that $\hat \epsilon_i$ converges to zero quadratically as 
$\sigma_j\to 0$. Hence, according to \eqref{tei}, with high probability, $T_{\hat \epsilon_i}=+\infty$. 
This, in turn, implies that the time step is asymptotically determined by $T_\mu$ and $T_\sigma$, 
which is used in the following section to simplify the analysis of the iterations in \eqref{quadgfsoli}.

\begin{pro}
\label{511}
Suppose $f$ satisfies Assumptions \ref{a2} and \ref{chyp}.  
Consider a sample of $n$ points drawn from $(\mu_j, \sigma_j)$.  Suppose that $n>6p+2$. 
Let $\hat \epsilon_i$ be given by  \eqref{hatepsilon}.  Then 
\begin{equation}
	\label{epsiest}
	\|\hat \epsilon_i\|_p=O\left(\frac{|\mu_j|^{9}}{\sigma_j^{7}}+\sigma_j^2\right),
\end{equation}
where  $\|\cdot \|_{p}$ represents the $L^p$ norm with respect to $\Gamma_{\mu_j, \sigma_j}$.
\end{pro}
\begin{proof}
We consider the terms in \eqref{hatepsilon} separately.
First, 
let $\hat R$ be given by \eqref{hatR}. We claim that 
\begin{equation}
	\label{hatrest}
	\|\hat R\|_p=O(|\mu_j|^3+\sigma_j^3).
\end{equation}
To establish \eqref{hatrest}, we begin by noting that 	 from 
Assumption \ref{chyp} we have
$
|f(x)-c x^2|\leq C |x|^3.
$
Thus, because the polynomial $c x^2$ is sub-optimal in the least squares problem in \eqref{mcls}, we have 
\[
\hat R^2 \leq \frac{C}{n}\sum_{i=1}^n x_i^6.
\]
Using Lemma \ref{normie}, we have 
\[
|\hat R|^p \leq C\sum_{i=1}^n |x_i|^{3p}.
\]
Using now Lemma \ref{hplem}, we see that the expected value of the right-hand side can be bounded using the estimate $\int_{\Rr} |x_i|^{3p} \Gamma_{\mu_j, \sigma_j}dx_i= O(|\mu_j|^{3p}+\sigma_j^{3p})$. 
Thus, \eqref{hatrest} holds.

Next,  we examine $\bar \beta_i$ given by \eqref{betai}.  
Since sparse sampling is not employed (see Section \ref{sparsesec}), we have 
$\ell\equiv 1$. 
Accordingly, 
 the least squares condition \eqref{mcoc} in \eqref{hatbeta} gives 
\begin{equation}
\label{betai0}
\hat \beta_i=0. 
\end{equation}
Thus, we only need to consider
$\hat \sigma_{\beta_i}$ given by \eqref{sigmahatbeta}. 
We claim that 
\begin{equation}
	\label{sigmabetest}
	\|\hat \sigma_{\beta_i}\|_p=O\left(\frac{|\mu_j|^{9}}{\sigma_j^{7}}+\sigma_j^2\right).
\end{equation}
To establish \eqref{sigmabetest}, we use  \eqref{betai0}
in 
 \eqref{sigmahatbeta}, to get
 \[
 \hat \sigma_{\beta_i}^2=
 \frac 1 n \sum_{k=1}^{n}  (f(x_k)-q(x_k))^2 B_i^2(\theta_j, x_k).
 \]
Thus, using Lemma \ref{normie}, we have
 \[
 \hat \sigma_{\beta_i}^p\leq C
\sum_{k=1}^{n}  |f(x_k)-q(x_k)|^p |B_i(\theta_j, x_k)|^p.
 \]
By permutation invariance, all terms in the previous sum are identical in distribution. Hence,
\[
	E\left[\hat \sigma_{\beta_i}^p\right]\leq C E \left[|f(x_1)-q_j(x_1)|^p |B_i(\theta_j, x_1)|^p\right].
\]
Further, we have
\[
|f(x)-q_j(x)|^p\leq C |f(x)-c x^2|^p +C |q_j(x)-c x^2|^p.
\] 
Therefore, using Assumption \ref{chyp},   
\begin{equation}
E\left[\hat \sigma_{\beta_i}^p\right]\leq C\int_{\Rr}  |x_1|^{3p} |B_i|^p \Gamma_{\mu_j, \sigma_j}dx_1+
CE \left[|q_j(x_1)-c x_1^2|^p |B_i(\theta_j, x_1)|^p\right].\label{hatsigmabetaest}
\end{equation}

Now, we consider the expressions for $B_i$ in \eqref{bif} as follows. 
For $B_1$, we have 
\begin{align*}
	\int_{\Rr}  |x_1|^{3p} |B_1|^p \Gamma_{\mu_j, \sigma_j}dx_1&=\frac{1}{\sigma_j^p}
	\int_{\Rr}  |y+\mu_j|^{3p} \left|\frac{y}{\sigma_j}\right|^p \Gamma_{0, \sigma_j}dy\\
	&\leq 
	\frac{C}{\sigma_j^p}\left[
	\int_{\Rr}  |y|^{3p} \left|\frac{y}{\sigma_j}\right|^p \Gamma_{0, \sigma_j}dy
	+|\mu_j|^{3p} \int_{\Rr}  \left|\frac{y}{\sigma_j}\right|^p \Gamma_{0, \sigma_j}dy
	\right]\\&=O\left(\frac{|\mu_j|^{3p}}{\sigma_j^{p}}+\sigma_j^{2p}\right).
\end{align*}
Similarly, for 
$B_2$, we have 
\begin{align*}
	\int_{\Rr}  |x_1|^{3p} |B_2|^p \Gamma_{\mu_j, \sigma_j}dx_1&=\frac{1}{\sigma_j^p}
	\int_{\Rr}  |y+\mu_j|^{3p} \left|\frac{y^2}{\sigma_j^2}-1 \right|^p \Gamma_{0, \sigma_j}dy\\
	&\leq \frac{C}{\sigma_j^p}\left(
	\int_{\Rr}  |y|^{3p} \left|\frac{y^2}{\sigma_j^2}-1 \right|^p \Gamma_{0, \sigma_j}dy+
		|\mu_j|^{3p} \int_{\Rr}  \left|\frac{y^2}{\sigma_j^2}-1 \right|^p \Gamma_{0, \sigma_j}dy
	\right)\\
	&=O\left(\frac{|\mu_j|^{3p}}{\sigma_j^{p}}+\sigma_j^{2p}\right).
\end{align*}

To address the second term in \eqref{hatsigmabetaest}, we use the estimates in Proposition \ref{abcest} combined with H\"older's inequality, as we explain next.
Let $t>1$. 
Taking into account \eqref{bif}, we observe that 
\[
\int_{\Rr} |B_1|^t \Gamma_{\mu_j, \sigma_j}dx
=\frac{1}{\sigma_j^t}\int_{\Rr}  \frac{|x|^t}{\sigma_j^t}\Gamma_{0, \sigma_j}dx=\frac{C}{\sigma_j^t}, 
\]
using a change of variables. Similarly, 
\[
\int_{\Rr} |B_2|^t \Gamma_{\mu_j, \sigma_j}dx
=\frac{1}{\sigma_j^t}\int_{\Rr}  \left|\frac{x^2}{\sigma_j^2}-1\right|^t\Gamma_{0, \sigma_j}dx
=\frac{C}{\sigma_j^t}.
\]
Accordingly, 
we have
\begin{equation}
\label{biest}
\|B_i\|_{t}\leq \frac{C_t}{\sigma_j},
\end{equation}
for some constant $C_t$ depending only on $t$.

Next, we have 
\[
E\left[ |q_j-c x^2_1|^p |B_i|^p\right]\leq
C \left(
E\left[ |a_j|^p |B_i|^p\right]+
 E\left[ |b_j x_1|^p |B_i|^p\right]
 +E\left[ |(c_j-c)_j x_1^2|^p |B_i|^p\right]
\right).
\]
We apply H\"older's inequality to each of the terms in the right-hand side of
the preceding expression with exponents, 
 $r, s, t>1$ with
\[
\frac 1 r +\frac 1 s + \frac 1 t=1.
\]
Then, using \eqref{biest},  we have
\[
E\left[ |q_j-c x^2_1|^p |B_i|^p\right]
\leq \frac{C}{\sigma^p_j}\left(\|a_j\|_{p r}^p+\|b_j\|_{pr}^p \|x_1\|^p_{ps}+\|c_j-c\|_{pr}^p \|x_1\|^{2p}_{2 p s}  \right).
\]
Next, we select $r$ close to $1$ so that $n>6 p r+2$ such that we can apply the estimates
in Proposition \ref{abcest}, and in Lemma \eqref{hplem} to get
\[
E\left[  |q_j-c x^2|^p |B_i|^p \right]=O\left(\frac{|\mu_j|^{9p}}{\sigma_j^{7p}}+\sigma_j^{2p}\right).
\]
Therefore, 
\[
E\left[\hat \sigma_{\beta_i}^p\right]= O \left(\frac{|\mu_j|^{3p}}{\sigma_j^p}+\sigma_j^{2p}\right)+
O\left(\frac{|\mu_j|^{9p}}{\sigma_j^{7p}}+\sigma_j^{2p}\right)
=O \left(\frac{|\mu_j|^{3p}}{\sigma_j^p}+\frac{|\mu_j|^{9p}}{\sigma_j^{{7p}}}+\sigma_j^{2p}\right);
\]
 thus establishing \eqref{sigmabetest} by observing that Young's inequality
 gives
 $\frac{|\mu_j|^{3p}}{\sigma_j^p}=O\left(
 \frac{|\mu_j|^{9p}}{\sigma_j^{{7p}}}+\sigma_j^{2p}\right)$.
So, using the preceding identity and \eqref{sigmabetest} in \eqref{betai}, we have
\begin{equation}
\label{barbetaiest}
\|\bar \beta_i\|_p
=O \left(\frac{|\mu_j|^{9}}{\sigma_j^{7}}+\sigma_j^2\right).
\end{equation}
Finally, we take into account \eqref{bqi}, \eqref{barbetaiest}, and \eqref{hatrest} in \eqref{hatepsilon}. This gives us \eqref{epsiest}.
\end{proof}

\subsection{Convergence}

To prove convergence, we use the following strategy. We construct a "good set", $\bar \Xi$, in 
the underlying probability space where $b_j$ is close to $0$, $c_j$ is close to $c$, and
$\hat \epsilon_j$ is small for every $j$.
This set is built inductively by intersecting a sequence of "good sets", $\Xi_j$, 
constructed in Proposition \ref{p48}. In that proposition, we deduce estimates for 
the probability of $\Xi_j^c$ using Chebychev inequality in conjunction with the estimates 
in the preceding sections. 
In $\bar \Xi$, Proposition \ref{513} shows that the time $T_j$ has a lower bound.
Because the convergence estimates involve terms of the form $\frac{\mu_j}{\sigma_j}$, we need to ensure that this quotient remains bounded. This boundedness and
convergence follow from the analysis of \eqref{quadgfsoli} performed in Proposition \ref{p410}.
Finally, in Proposition \ref{p411}, we show that 
 $\bar \Xi^c$, where convergence may not occur, has a small probability.

First, we define sets
$\Xi_j$ where $b_j$ is close to $0$, $c_j$ is close to $c$, and
$\hat \epsilon_j$ is small and estimate the probability of their complements. 
More precisely, because $\mu_j$ and $\sigma_j$ are random variables, 
we estimate $P(\Xi_j^c| (\mu_j, \sigma_j) )$.

\begin{pro}
\label{p48}	
Suppose $f$ satisfies Assumptions \ref{a2} and \ref{chyp}.
Let $p>1$ and suppose that  $n>6p+2$. 
Fix a constant $C>0$ and consider the set $\Xi_j$ of all events for which
\begin{equation}
	\label{bestpoint}
	|b_j|\leq \left(\frac{|\mu_j|^8}{\sigma_j^6}+\sigma_j^2\right)^{2/3},
\end{equation}
\begin{equation}
	\label{cestpoint}
	|c_j-c|\leq \left(\frac{|\mu_j|^7}{\sigma_j^6}+\sigma_j\right)^{1/3},
\end{equation}
and
\begin{equation}
	\label{epsiestpoint}
	\hat \epsilon_i\leq \left( \frac{|\mu_j|^9}{\sigma_j^7}+\sigma_j^2\right)^{2/3}.
\end{equation}
Then, 
\begin{equation}
\label{pbad}
P(\Xi_j^c| (\mu_j, \sigma_j) )\leq C\left( \frac{|\mu_j|^{14}}{\sigma_j^{12}}+\sigma_j^2\right)^{p/3}.
\end{equation}
\end{pro}
\begin{proof}
We observe that	$\Xi_j$ is the intersection of the sets determined by \eqref{bestpoint}, \eqref{cestpoint}, and \eqref{epsiestpoint}. Hence, 
the probability of $\Xi_j^c$ is bounded by the sum of the probabilities of the complements of each of these sets. To obtain \eqref{pbad}, it suffices to show that each of these probabilities is
$O\left(\left(\frac{|\mu_j|^{14}}{\sigma_j^{12}}+\sigma_j^2\right)^{p/3}\right)$.
	
According to Chebychev's inequality, for any $p$-integrable random variable $Y$, we have
\[
P(|Y|> \gamma)\leq \frac{\|Y\|_p^p}{\gamma^p}. 
\]	
We observe that 
by choosing $\gamma=C\|Y\|_p^{\theta}$, we obtain
\[
P(|Y|> \gamma)\leq  C\|Y\|_p^{p(1-\theta)}.
\]
For the complements of the sets in \eqref{bestpoint} and  \eqref{epsiestpoint}, we use $\theta=\frac 2 3$ and for the complement of the set in \eqref{cestpoint}, we use $\theta=\frac 1 3$. 
The corresponding norms are given in Proposition \ref{abcest}
and in Proposition \eqref{511}. 
Finally, we note that by Young's inequality
\[
\left(\frac{|\mu_j|^8}{\sigma_j^6}+\sigma_j^2\right)^{p/3}, \ \ \left(\frac{|\mu_j|^7}{\sigma_j^6}+\sigma_j\right)^{2p/3}, 
\ \text{and }  \ \left( \frac{|\mu_j|^9}{\sigma_j^7}+\sigma_j^2\right)^{p/3}
\leq C\left( \frac{|\mu_j|^{14}}{\sigma_j^{12}}+\sigma_j^2\right)^{p/3}.
\]
Accordingly, \eqref{pbad} follows. 

The choice of 
$\theta=\frac 2 3$ is somewhat arbitrary; 
any $\frac 1 2 <\theta<1$ would result in similar estimates. 
The requirement $\theta>\frac 1 2$ stems from the need for $\frac{b_j}{\sigma_j}$ to vanish as $\sigma_j\to 0$ later in our proof, which is only valid when $\theta>\frac 1 2$.
The choice of $\theta=\frac{1}{3}$ for \eqref{cestpoint} is not critical, but it simplifies the expression in \eqref{pbad}
as it produces a term of the same order as the previous ones. 
\end{proof}

Now, we set
\begin{equation}
	\label{barxi}\bar \Xi=\cap_{j=1}^\infty\Xi_j. 
\end{equation} 
$\bar \Xi$ 
contains all events that 
satisfy the pointwise bounds from Proposition \ref{p48} for all $j$. 
%
We establish convergence of $(\mu_j,\sigma_j)$ in
$\bar \Xi$ and then show that $\bar \Xi^c$ has a small
probability.

To prove Theorem \ref{mainteo}, we introduce the following sequence. Let 
$\Theta$ be given by \eqref{Theta} and
$C_K$ be the constant in \eqref{ck}
determined in the proof of Proposition \ref{p410}, and 
we assume that $\delta$ is small enough so that 
\begin{equation}
\label{kn}
K_{j+1}=\frac{K}{2}+\sum_{k=0}^{j}C_K \delta^{1/3}  (\Theta^{1/3})^k\leq K.
\end{equation}
Then, in Proposition \ref{p410}, we proceed by induction and show that 
 on $\cap_{k=0}^j\Xi_k$, we have
\begin{equation}
	\label{induction}
\frac{|\mu_j|}{\sigma_j}\leq K_j, \quad  \sigma_j\leq \delta\Theta^{j}
\end{equation}
for all $j$.
\begin{remark}
\label{r510}
Note that from \eqref{induction}, we have $|\mu_j|\leq 
K\Theta^{j}\delta$. Furthermore, on $\bar \Xi$, if \eqref{induction} holds, we have  
\[
\hat \epsilon_i, |b_j|\leq C \sigma_j^{4/3},
\]
and $| c_j-c|\leq C \sigma_j^{1/3}$ (which implies $\frac c 2 \leq c_j\leq \frac 3 2 c$ since for 
$\delta$ small $C \sigma_j^{1/3}\leq \frac c 2$), 
 using \eqref{bestpoint}, \eqref{cestpoint}, and \ref{epsiestpoint}. Finally, from \eqref{pbad}, we also have
\[
P(\Xi_j^c| (\mu_j, \sigma_j) )\leq C\left( (1+K^{14})\sigma_j^2\right)^{p/3}\leq  C \delta^r
\Theta^{rj}, 
\]
for $r=\frac{2p}{3}$. 
\end{remark}

The contraction of the algorithm in each step is determined by the factor $e^{-2 c_j T_j}$. The next proposition examines this term.

\begin{pro}
\label{513}
Suppose $f$ satisfies Assumptions \ref{a2} and \ref{chyp}.
Suppose $\mu_j$ and $\sigma_j$ satisfy \eqref{induction} and $\delta$ is small enough. 
Let 
\begin{equation}
\label{Theta}
\Theta=\max\left\{1-\frac{\upsilon_1}{4K}, 1-\upsilon_2\right\}.
\end{equation}
Then, 
on $\cap_{k=0}^j\Xi_k$,
\begin{equation}
\label{expbound}
e^{-2 c_j T_j}\leq \Theta. 
\end{equation}
Moreover,
\begin{equation}
\label{onomorebound}
e^{-2 c_j T_j}\geq 1-\upsilon_2
\end{equation}
\end{pro}
\begin{proof}
On $\cap_{k=0}^j\Xi_k$, the argument of the logarithm in the expression defining $T_{\hat \epsilon_i}$, 
\eqref{tei}, is
\begin{equation*}
1-\frac{2 c_j \gamma_i\sigma_j}{\hat \epsilon_i}\leq 
1-\frac{c \gamma_i}{C} \sigma_j^{-1/3}<0, 
\end{equation*}
for all $j$ if $\delta$ is small enough, considering Remark \ref{r510}. Thus, $T_{\hat \epsilon_i}=+\infty$. Consequently, only $T_\mu$ and $T_\sigma$ constrain the time step,  
thus, according to \eqref{timestep}, $T_j=\min\{T_{\mu}, T_{\sigma}\}$.

Moreover, according to \eqref{tc2},
\[
T_{\sigma}=\frac{-\log(1-\upsilon_2)}{2c_j}.
\]	
Thus, if $T_j=T_\sigma$, we have
\[
e^{-2 c_j T_j}=1-\upsilon_2.
\]
Hence, because $T_j=\min\{T_{\mu}, T_{\sigma}\}$, we obtain \eqref{onomorebound}.

Next, by examining \eqref{tc1}, we have	
\[
T_\mu=-\frac{1}{2c_j}\log\left(1-\frac{2 c_j\sigma_j}{|b_j+2 c_j\mu_j|}\upsilon_1\right).
\]	
We can rewrite the argument of the logarithm as 
\begin{equation}
	\label{logarg}
1- \frac{2 c_j \sigma_j}{|b_j+2 c_j \mu_j|}\upsilon_1=
1- \frac{2 c_j }{\left|\frac{b_j}{\sigma_j}+2 c_j \frac{\mu_j}{\sigma_j}\right|}\upsilon_1.
\end{equation}
Now, we note that $c_j\leq \frac 3 2 c$ and that, by Remark \ref{r510},   $\frac{b_j}{\sigma_j}\leq C\sigma_j^{1/3}$.
Thus,
we have the following estimate for the absolute value in the denominator:
\[
\left|\frac{b_j}{\sigma_j}+2 c_j \frac{\mu_j}{\sigma_j}\right|\leq C\sigma_j^{1/3} +3 c K
\leq 4 c K, 
\]
if $\delta$ is small enough.
Therefore, 
\[
\frac{2 c_j }{\left|\frac{b_j}{\sigma_j}+2 c_j \frac{\mu_j}{\sigma_j}\right|}\geq \frac{1}{4K}.
\]
Thus, combining the preceding inequality with \eqref{logarg}, we get 
\[
T_{\mu}\geq -\frac{1}{2c_j}\log\left(1-\frac{\upsilon_1}{4K}\right).
\]
Accordingly, we obtain the upper bound in \eqref{expbound}.
\end{proof}

Now, we prove exponential convergence in $\bar \Xi$, that is \eqref{induction}. 
\begin{pro}
\label{p410}
Suppose $f$ satisfies Assumptions \ref{a2} and \ref{chyp}.
Then, on  $\cap_{k=0}^j\Xi_k$, we have
\[
\sigma_{j+1}\leq \Theta \sigma_j
\]
and
\[
\frac{|\mu_{j+1|}}{\sigma_{j+1}}\leq K_{j+1}, 
\]
where $K_{j+1}$ is given by \eqref{kn}.
\end{pro}
\begin{proof}
By induction, we assume that $\sigma_j\leq \delta \Theta^j$ and $\frac{|\mu_{j|}}{\sigma_{j}}\leq K_{j}<K$.	

The first statement is trivial by using \eqref{expbound} in the iterative formula for $\sigma_j$, 
 $\sigma_{j+1}=e^{-2 c_j T_j}\sigma_j$ from  \eqref{quadgfsoli}.
Regarding $\mu_j$, we  have
  \[
  \mu_{j+1}=
  \frac{ b_j(e^{-2 c_j T_j}- 1)}{2 c_j}+ \mu_je^{-2 c_j T_j}.
  \]
Using again  \eqref{quadgfsoli}, we obtain
\[
\left|\frac{\mu_{j+1}}{\sigma_{j+1}}\right|\leq 
\left|\frac{ b_j(1-  e^{2 c_j T_j})}{2 c_j\sigma_j }\right|+ \left|\frac{\mu_j}{\sigma_j}\right|.
\]
We note that 
\[
\frac{|b_j|}{2 c_j \sigma_j}\leq 
C\frac 1 {\sigma_j}\left(
\frac{\mu_j^{8}}{\sigma_j^{6}}+\sigma_j^{2}
\right)^{2/3}
\leq C (K^8+1)^{2/3}\sigma_j^{1/3}\leq C (K^8+1)^{2/3}\delta^{1/3}\Theta^{j/3},
\]
using \eqref{bestpoint}, the induction hypothesis, and that $K_j\leq K$. Furthermore, from \eqref{onomorebound}, 
we have 
\[
e^{2 c_j T_j}\leq \frac{1}{1-\upsilon_2}.
\]
Therefore, we obtain
\[
\left|\frac{\mu_{j+1}}{\sigma_{j+1}}\right|\leq K_j+C_K \delta^{1/3} \Theta^{j/3}=K_{j+1}, 
\]
where 
\begin{equation}
\label{ck}
C_K=C (K^8+1)^{2/3}.\qedhere
\end{equation}
\end{proof}

Next, we establish bounds on the probability of $\bar \Xi^c$.
\begin{pro}
\label{p411}
Suppose $f$ satisfies Assumptions \ref{a2} and \ref{chyp}.
Let $p>1$ and suppose that  $n>6p+2$ and that $\delta$ is small enough. 
Then
\[
P(\bar \Xi^c)\leq C\delta^r,
\]
 where $r=\frac{2p}3$, for some constant $C$ depending on $K$.
\end{pro}
\begin{proof}
	Taking into account the definition of $\bar \Xi$ in \eqref{barxi}, 
we decompose $\bar \Xi^c$ as a disjoint union, as follows	
\[
\bar \Xi^c=\cup_j \Xi_j^c=\cup_j\left(\Xi_j^c\cap (\cap_{k=0}^{j-1}\Xi_k)\right).
\]
The advantage of this decomposition is that on $\cap_{k=0}^{j-1}\Xi_k$, 
we have by Proposition \eqref{p410}, 
$\sigma_j\leq \delta \Theta^j $ and $ \frac{|\mu_j|}{\sigma_j}\leq K$.
Thus, denoting by $\Omega$ the underlying probability space, 
\[
 \cap_{k=0}^{j-1}\Xi_k \subset \left\{\omega\in \Omega: \sigma_j\leq \delta \Theta^j \wedge \frac{|\mu_j|}{\sigma_j}\leq K\right\}=\Upsilon_j.
\]
Accordingly, 
\[
P(\Xi_j^c\cap (\cap_{k=0}^{j-1}\Xi_k))\leq 
P(\Xi_j^c\cap \Upsilon_j).
\]
Next, we observe that $\Upsilon_j$ is $(\mu_j, \sigma_j)$-measurable. Hence, 
\begin{align*}
P(\Xi_j^c\cap \Upsilon_j)&=E \left[1_{\Upsilon_j} 1_{\Xi_j^c}\right]\\
	&=E \left[1_{\Upsilon_j} 
	E\left[1_{\Xi_j^c}|(\mu_j, \sigma_j)\right]
	\right]\\
	&\leq C \delta^{r} \Theta^{rj}, 
\end{align*}
%
%
%
 for $r=\frac{2p}3$ using Remark \ref{r510}. 
Thus,
\begin{equation}
	\label{pbad2}
P(\Xi_j^c\cap (\cap_{k=0}^{j-1}\Xi_k))\leq C \delta^{r} \Theta^{rj}. 
\end{equation}	
Consequently,
\[
P(\bar \Xi^c)\leq \sum_j P(\Xi_j^c\cap (\cap_{k=0}^{j-1}\Xi_k))\leq  \sum_jC  \delta^{r} \Theta^{rj}
\leq C \delta^r.\qedhere
\]
\end{proof}	

\subsection{Proof of Theorem \ref{mainteo}}
\label{mainteoproof}

Finally, by gathering the preceding results, we establish the main theorem.
\begin{proof}[Proof of Theorem \ref{mainteo}]
The theorem follows from combining the convergence result in Proposition \ref{p410} with 
the estimate for the probability of $\bar \Xi^c$ in Proposition \ref{p411}.
\end{proof}

\section{Algorithmic improvements}

Here, we discuss several implementation details and algorithmic improvements
for enhancing code performance. In Section \ref{ext},  we explain how $f$ is extended from its original domain to
$\Rr$. 
A few additional improvements
and modifications to the explicit solution of the quadratic flow,  \eqref{quadgfsol}, are discussed in Section \ref{lfi}. Sections \ref{adaptivitysec} and \ref{sparsesec} discuss two strategies to improve performance: adapting the number of sample points or skipping the sampling and quadratic interpolation step.  A detailed discussion of the stopping criteria for the algorithm is presented in Section \ref{stop}. 
There, we discuss termination at interior points and close to the boundary, which must be handled differently.
After the algorithm ends, we implement two strategies to improve the results. 
First, we consider a restart strategy (Section \ref{restartsec}) that involves restarting the algorithm 
if the best point where $f$ was evaluated is not near the terminal value of $\mu$. Second,  we use
a postprocessing stage (Section \ref{postprocesssec}) that takes advantage of the quadratic interpolation to improve the accuracy
of our results.
Finally, because our algorithm uses a relatively small number of function evaluations, it is often 
competitive to run it multiple times. Furthermore, by replacing the initial list $\Lambda_0$ in \eqref{lambda1} with
the final list, $\Lambda_j$, from the previous run of the algorithm, we further reduce 
the number of function evaluations needed for a re-run. 
Thus, using previous function evaluations allows running the algorithm multiple times with only a modest increase in the number of evaluations.
This strategy, which we call boosting,  is discussed in Section \ref{boosting}.

\subsection{Function extension}
\label{ext}

Let the domain of $f$ be the interval $[x_{min}, x_{max}]$.
Our approach uses the gradient flow of \eqref{Fdef}, requiring $f$ to be extended to $\mathbb{R}$ while preserving its minimum. 
For this,  we consider a parameter 
$\nu>0$ and define the extension
\[
\hat f(x) =
\begin{cases}
f(x_{min})+\nu (x_{min}-x)\quad& x<x_{min}, \\
f(x)& x_{min}\leq x\leq x_{max},\\
f(x_{max})+\nu (x-x_{max})\quad& x>x_{max}.
\end{cases}
\]
Accordingly, we have
\[
\min_{x\in \Rr}\hat f(x)=\min_{x_{min}\leq x\leq x_{max}} f(x).  
\]

Outside the interval $[x_{min}, x_{max}]$, $\hat{f}$ has linear growth, ensuring $\hat{f} \in C_p(\mathbb{R})$. Additionally, if $f$ is convex in $[x_{min}, x_{max}]$ and if $\nu > \max(|f'(x_{min})|, |f'(x_{max})|)$, $\hat f$ is convex. However, strict convexity is not preserved by this extension. Although strictly convex extensions can be constructed, our experience indicates that when $f$ is convex with an interior minimum, function evaluations outside the domain are minimal, rendering the choice of extension irrelevant. 

For convenience, we automatically scale the parameter $\nu$ based on the domain length. We fix a parameter $\varpi$ and set 
\[
\nu=\frac{\varpi}{x_{max}-x_{min}}. 
\]

\subsection{Flow}
\label{lfi}

To avoid numerical instabilities due to the exponentials in \eqref{quadgfsol}, we limit the time step to
a maximum value $h_{max}$. 

In degenerate cases, where $f$ is locally flat, i.e., functions that are remarkably well approximated by a constant or linear function,
there is an excellent fit between $q$ and $f$, and the time step is large or unbounded. However, since $c_j$ is zero or near zero, the algorithm only ends when the maximum number of iterations is reached, as there is no contraction in $\sigma_j$. For instance, if $f \equiv 0$, then $b_j = c_j = 0$. In this case, $\mu_j$ and $\sigma_j$ remain unchanged. 
Accordingly, 
when the time step $T_j$ is larger than $h_{max}$ and $c_j\geq 0$,
we contract $\sigma_j$ by an additional factor $0<\vartheta<1$.
That is, whenever
the time step $T_j>h_{max}$, we write
\[
e^{-2 c_j T_j}=e^{-2 c_j h_{max}} e^{-2 c_j (T_j-h_{max})}
\]
and replace the factor $e^{-2 c_j (T_j-h_{max})}$ with a constant $\vartheta$. Accordingly, 
\eqref{quadgfsoli} becomes
\[
\mu_{j+1}=
\frac{ b_j(e^{-2 c_j h_{max}}\vartheta - 1)}{2 c_j}+ \mu_je^{-2 c_j h_{max}}\vartheta, 
\quad 
\sigma_{j+1}=\vartheta e^{-2 c_j h_{max}}\sigma_j. 
\]
Without this additional contraction, the algorithm would continue until reaching the maximum number of iterations, as the primary termination criterion depends on a small $\sigma_j$.
The value of $\vartheta$ is chosen close to $1$ to expedite the algorithm's termination for degenerate problems without substantially changing the iterations in all other cases.

A final enhancement addresses cases where solving \eqref{quadgfsol} yields a value $\mu_{j+1}$ outside the domain of $f$. In such instances, we project $\mu_{j+1}$ onto the nearest boundary point and contract $\sigma_{j+1}$ by a factor of $\vartheta$. Without this contraction, the flow might repeatedly push $\mu_j$ outside the domain without substantially contracting $\sigma_j$. This occurs at boundary points where the quadratic interpolation is suboptimal due to the extension's lack of smoothness. As a result, the time step may become exceedingly small, leading to $\mu_j$ being continuously pushed outside the domain and returned to the same state without any change in $\sigma_j$, causing numerous iterations.

\subsection{Adaptivity in the sample size}
\label{adaptivitysec}
In this section, we discuss adapting the sample size, $n_j$, to decrease function evaluations while maintaining the accuracy of the gradient flow approximation. 
In the choice of the time step in \eqref{timestep}, the expressions for $T_\mu$ and $T_\sigma$
do not directly depend on the number of sample points, but rather indirectly through the coefficients $a_j$, $b_j$, and $c_j$. 
In contrast, the expressions for $T_{\epsilon_i}$ depend on our estimate 
for \eqref{epsiloni}. 
The estimate for $R$ in  \eqref{epsiloni} converges to the limit integral in \eqref{e0} as $n_j\to \infty$, but may fluctuate with $n_j$ without any monotonic behavior. Based on expressions \eqref{q1} and \eqref{q2},  our estimate for $Q_i$ remains independent of $n_j$. 
 Nevertheless, the explicit dependence on $n_j$ in \eqref{betai} shows that the error 
 decreases as $n_j$ increases. Consequently, our estimate for $T_{\epsilon_i}$ grows with increasing $n_j$.

When $\min(T_{\epsilon_1}, T_{\epsilon_2}) > \min(T_\mu, T_\sigma)$, increasing the number of sample points will not affect the time step, as it is determined by either $T_\mu$ or $T_\sigma$. On the other hand, when $\min(T_{\epsilon_1}, T_{\epsilon_2}) < \min(T_\mu, T_\sigma)$, increasing the number of sample points will lead to an increased time step. Consequently, instead of keeping the number of points fixed, we employ a greedy strategy for selecting the number of points at each iteration, starting with an initial sample of $n_0$ points.
Then, we consider 
two integer numbers $n_{min}$ and $n_{max}$, representing  a sample's minimum and maximum size. At  iteration $j$, given the time $T_j$ and the corresponding times, $T_{\epsilon_1}, T_{\epsilon_2}, T_\mu, T_\sigma$, we select $n_{j+1}$ as follows.
We set $n_{j+1}=n_{min}$ if 
$\min(T_{\epsilon_1}, T_{\epsilon_2})>\min(T_\mu, T_\sigma)$
and set $n_{j+1}=n_{max}$ if $\min(T_{\epsilon_1}, T_{\epsilon_2})\leq \min(T_\mu, T_\sigma)$.

\subsection{Sparse  sampling}
\label{sparsesec}

To reduce the number of function evaluations, we employ a strategy of skipping sampling and least squares approximation at certain iterations. This approach is used when the fit between $q_j$ and $f$ at iteration $j$ yields a significantly smaller error bound than \eqref{gammabound} permits.
In this case, the only limiting factor for a larger time step is the saturation of either \eqref{mb} or \eqref{sb}, which were imposed to ensure the validity of  $Q_i$ approximations in \eqref{q1} and \eqref{q2} required for error estimates on the current iteration. 
As a result, we can use $q_j$ without  sampling in iteration $j+1$ and obtain new error estimates based on the new point $(\mu_{j+1}, \sigma_{j+1})$. It is important to note that this is feasible because the error estimates in Section \ref{mee} consider the possibility of the sampling Gaussian being distinct from the one defined by $(\mu_{j+1}, \sigma_{j+1})$.

The process just described can be 
iterated as long as the original bounds in \eqref{gammabound} are not exceeded, as we explain now. If  \eqref{gammabound} is not saturated at the end of one iteration, we do not  sample
 to calculate $q_{j+1}$. Instead, we 
 use the same quadratic interpolation but require a smaller error at the end of the following iteration.
Suppose that $q_j=a_j+b_j x +c_j x^2$ with  $c_j\neq 0$ (the case $c_j=0$ is similar).
 After a time step $T_j$, by \eqref{error}, we have
\[
|\mu^f(T_j)-\mu^q(T_j)|\leq \frac{\epsilon_1 (1-e^{-2 c_j T_j} )}{2 c_j\sigma_j}\sigma_j,\qquad 
|\sigma^f(T_j)-\sigma^q(T_j)|\leq \frac{\epsilon_2(1-e^{-2 c_j T_j} )}{2 c_j\sigma_j} \sigma_j.
\]
Since \eqref{gammabound} allows for a larger error estimate, we can repeat a further time step
as long as 
\begin{align*}
&|\mu^f(T_{j+1})-\mu^q(T_{j+1})|\leq \left(\gamma_1-\frac{\epsilon_1  (1-e^{-2 c_j T_j} )}{2 c_j\sigma_j}\right)\sigma_j,\\
&|\sigma^f(T_{j+1})-\sigma^q(T_{j+1})|\leq \left(\gamma_2-\frac{\epsilon_2 (1-e^{-2 c_j T_j} )}{2 c_j\sigma_j}\right) \sigma_j.
\end{align*}
We can interpret the right-hand side as an error budget that we can still spend without additional sampling. 

If $\sigma_{j+1}\leq \sigma_j$, we can do one (or more) iteration step without  sampling with
$\gamma_i$ replaced by
\[
\tilde \gamma_i=
\begin{cases}
\gamma_i-\frac{\epsilon_i  (1-e^{-2 c_j T_j} )}{2 c_j \sigma_j}\qquad &c_j\neq 0\\
\gamma_i-\frac{\epsilon_i T_j}{\sigma_j}& c_j=0.
\end{cases}
\]
If $\sigma_{j+1}> \sigma_j$, we always  sample. The reason for doing this is that, in this case, $c_j<0$.
Accordingly, the errors are amplified by the quadratic flow rather than dampened.

\subsection{Stopping criteria}
\label{stop}

A key component of our algorithm implementation is the stopping criterion. At iteration $j$, we must decide whether to stop or continue with the iterations. Our stopping criterion is a 
conjunction of several conditions that reflect that certain necessary optimality conditions are met
within specified tolerances.  

We first observe that if $F$ is continuous up to $\sigma=0$, 
$F$ has a global minimum at some point $(\bar \mu, 0)$ (see Proposition \ref{relaxpro}). Moreover, 
no point $(\mu, \sigma)$ with $\sigma>0$ can be a global minimum for $F$ since
$F(\mu, \sigma)>\min f$ for all $\sigma>0$. 
Thus, our first stopping criterion concerns $\sigma_j$ being small enough. For this, we define a target standard deviation,   $\sigma_{target}>0$. Our first stopping condition is met if $\sigma_j< \sigma_{target}$. 

In addition to this criterion, we require additional conditions to be met. 
These differ depending on the proximity of $\mu_j$ to the boundary.
 For a fixed $\kappa>0$ (for example, $\kappa=1$), we say that $\mu_j$ is far from the boundary if the distance from $\mu_j$ to the boundary is larger than $\kappa \sigma_j$; conversely, 
we say that   $\mu_j$ is close to the boundary if the preceding condition is not met. 

When $\mu_j$ is away from the boundary and $\sigma_j$ small, the standard deviation of the sample of $f$,
$\sigma_f$, is of the order of
$|f'(\mu_j)|\sigma_j$. At a minimum point of $f$, $\bar \mu$, we have $f'(\bar \mu)=0$.
This second stopping criterion attempts to detect critical points of $f$. 
Accordingly, we fix a number $\delta_f>0$, small, and
we stop once $\sigma_f\leq \delta_f$. 

To summarize, for points far from the boundary, the algorithm stops if, simultaneously,
\begin{itemize}
\item $ \sigma_j\leq \sigma_{target}$
\item$\sigma_f\leq \delta_f$.  
\end{itemize}

For $\mu_j$ near the boundary,  we proceed as follows. Let $x_b$ be the closest point to $\mu_j$ on the boundary. Note that we may not have $f'(x_b) = 0$ even if $f$ has a minimum at $x_b$. Therefore, $\sigma_f$ is $O(\sigma_j)$, independent of whether $x_b$ is a minimum or not. 
Thus, we do not impose conditions on $\sigma_f$ for points near the boundary. Instead, we observe that if $x_b$ is a boundary minimum, $f'(x_b) \geq 0$ for $x_b = x_{min}$ and $f'(x_b) \leq 0$ for $x_b = x_{max}$. Thus, $f(x_b)$ should be less than or equal to $f$ evaluated at the sample points. To minimize function evaluations, we select $\tilde x_b$ from the current sample within the domain and closest to $x_b$.
 The algorithm stops if, simultaneously,
\begin{itemize}
	\item $\sigma_j\leq \sigma_{target}$
	\item $f(\tilde x_b)\leq f(x_i)$ for all interior points, $x_i$,  in the current sample inside the domain. 
\end{itemize}

To ensure termination, we set a small number 
$\sigma_{min}>0$ and two integers $N^*_i$ and $N^*_f$ corresponding to the maximum number of iterations and function evaluations, respectively. 
 The algorithm stops when either of the following conditions are met.
\begin{itemize}
	\item $\sigma_j<\sigma_{min}$;
	\item maximum number of iterations is achieved;
	\item maximum number of function evaluations is achieved. 
\end{itemize}
If terminated by a fail-safe criterion, the algorithm returns the best point found throughout the process.

\subsection{Restarting}
\label{restartsec}

Upon termination and before
the postprocessing step (see next section), we check if the best point found, $x_{best}$, is near $\mu_j$, i.e., $|x_{best}-\mu_j|<\sigma_j$. If not, the algorithm restarts with $\Lambda_0$ replaced by the current list $\Lambda_j$, and initial condition $\mu_{j+1}=x_{best}$, and $\sigma_{j+1}=\tilde\sigma/2$, where $\tilde\sigma$ is the standard deviation of the distribution from which $x_{best}$ was sampled. This strategy prevents termination at a non-optimal point if a better position was discovered, which can occur if the gradient flow converges to a local minimum but sampling identifies a point near a better local or global minimum.

\subsection{Postprocessing}
\label{postprocesssec}

The last stage in the algorithm is postprocessing. For points away from the boundary (in the sense of Section \ref{stop}), there are three candidates for the minimizer: the best point found by the algorithm, $x_{best}$, the mean $\mu_j$ in the last iteration, and, if $c_j>0$, the minimizer of $q_j$ (and if this last point falls outside the boundary, the nearest boundary point). The objective function is evaluated at
these points, and  the point corresponding to the best value is returned as output. For points  near the boundary, we consider three candidates: $x_{best}$, the mean $\mu_j$ in the last iteration, and the boundary point closest to $\mu_j$.  As before, the objective function is evaluated at
these points; the one corresponding to the best value is returned as output. 

\subsection{Boosting}
\label{boosting}

The main optimization cycle can be iterated using previous function evaluations as initial data. 
We can re-run the algorithm but retain all samples from previous cycles instead of using
$\Lambda_0=\{\}$ as in \eqref{lambda1}. Due to random sampling, this approach may find different minima without significantly increasing the number of evaluations, as illustrated in Section \ref{deppars}.

\section{Numerical results}
 \label{numresults}

In this section, we
assess our algorithm's performance relative to established
algorithms using a variety of test functions.   These functions represent a broad spectrum of optimization challenges discussed in Section \ref{tf}. Our analysis aims to illustrate the algorithm's robustness and performance in multiple problem contexts. Specifically, we focus on the algorithm's probability of finding a global minimum and its efficiency regarding the number of function evaluations needed. The criteria for these comparisons are detailed in Section \ref{comparasiosec}, while Sections \ref{73} and \ref{parsec} explain the chosen initial conditions and parameters. Section \ref{ares} provides a comprehensive discussion of our results. Furthermore, we investigate the algorithm's performance sensitivity to parameters in Section \ref{deppars}. Finally, in Section \ref{noisy}, we illustrate the algorithm behavior in noisy functions, which are relevant in many applications.

\subsection{Test Functions}
\label{tf}

Our performance evaluation comprises a diverse set of 50 test functions, encompassing various optimization challenges and complexities. Specifically, the tests use the following classes of functions:
\begin{itemize}
	\item 
	convex (uniformly convex functions, convex functions, non-smooth convex functions), 
	\item unimodal non-convex, concave functions (for which minima are at the boundary), 
	\item multimodal non-convex (including bimodal functions and several other well-known test functions), \item highly oscillatory functions (including functions with infinitely many local minima and infinitely many global minima), 
	\item discontinuous functions,
	\item degenerate problems (linear and constant functions).
\end{itemize}
These functions are depicted in Figures \ref{niceconvex}-\ref{discontinuous}. 
All test functions are normalized such that $\max f-\min f=1$.  Thus, the comparison and aggregation of errors (Section \ref{ares}) are meaningful. 
Moreover, in Section \ref{noisy}, 
we also present results for noisy functions, such as the ones that arise in mini-batching problems in machine learning.

 \subsection{Comparative Analysis of Algorithms}
\label{comparasiosec}

We compare our algorithm with global minimization methods representing several classes outlined in Section \ref{prior}: Nelder-Mead, Differential Evolution, Random Search, and Simulated Annealing. Our implementation was developed in Mathematica, and these algorithms are also available as built-in versions within the same platform. The built-in algorithms were executed using default parameters for consistency.

A potential comparison metric between algorithms is the average run time, denoted by $\tau$, measured in seconds. However, run time is not ideal for evaluating algorithms in a manner that is independent of function evaluation cost.  Run time depends on the number of function evaluations, their evaluation time,  and function-independent overhead computation time.
 Our algorithm's efficiency advantage -- reduced function evaluations -- may not be evident in the run time for relatively inexpensive test functions. This is because the overhead in our algorithm is computationally expensive.
  Furthermore, our Mathematica code is compared to the platform's internal implementations of alternative algorithms, which may be more optimized for speed, suggesting run time may not be the most suitable comparison metric. Finally,  evaluation metrics should also consider both function evaluations and the algorithm's success probability.
 Moreover, the number of function evaluations and success probability are implementation-independent metrics, unlike run-time, which could change substantially with the choice of programming language or computer hardware.
   These are discussed in detail below.

For each test function, all algorithms  are run 100 times.
We considered a run successful if the objective function's value at the found candidate minimizer $\tilde x$ is close to the value at a global minimizer $\bar x$.  More precisely, we consider an output a success if $|f(\tilde x)-f(\bar x)|\leq 10^{-3}$.

To assess the quality of a candidate minimizer $\tilde x$, we consider two evaluation metrics: the average optimization gap, $\Delta$, which represents the average error $|f(\tilde x)-f(\bar x)|$ over 100 runs, and the average gap conditional on success, $\Delta_c$, also calculated over 100 runs. The gap value $\Delta$ quantifies the deviation of an algorithm output's objective function from the global minimum on average. Lower $\Delta$ values signify algorithms that either select a global minimizer or a point with an objective function value near the global minimum (i.e., favorable local minima). Conversely, higher $\Delta$ values correspond to algorithms that identify local minima or terminate before converging to the global minimum.
The gap $\Delta_c$ takes into account only instances where successful identification of a global minimum occurred. Consequently, when applying a typical local minimizer algorithm to a multimodal function, a large gap $\Delta$ is expected (as finding a global minimum is less likely), while a small conditional gap $\Delta_c$ is anticipated (given that a global minimum was found, the local algorithm exhibits higher precision). On the other hand, an effective global minimizer should exhibit a small gap $\Delta$, even if its $\Delta_c$ value is not as small as the one achieved using a local minimizer.

The number of function evaluations for each algorithm and function class depends on the termination criterion. To ensure a fair comparison between the algorithms, we selected a termination criterion for our algorithm that yields an average $\Delta_c$ value similar to those of the various comparison algorithms (refer to Section \ref{parstop} for a comprehensive discussion).

Algorithms differ in terms of the number of function evaluations and success rates. To effectively compare algorithms, we need a metric that considers both factors. Let $\Pi$ denote the probability of success for an algorithm when applied to a given class of functions, and let $N_f$ represent the average number of function evaluations per run. We examine two algorithms, indexed by $i=1,2$, characterized by $\Pi^i$ and $N_f^i$.

To simplify, we assume that each algorithm utilizes exactly $N_f^i$ function evaluations and has an independent probability of success, $\Pi^i$, for every run. We propose two experiments. In the first experiment, we execute each algorithm repeatedly and terminate it upon achieving success. 
Elementary probability shows that this scenario's expected number of function evaluations is:
\begin{equation}
\label{ns}
N_s^i=\frac{N_f^i}{\Pi^i}.
\end{equation}
In the second experiment, we execute the first algorithm $N_f^2$ times and the second algorithm $N_f^1$ times, utilizing a total of $N_f^1 N_f^2$ function evaluations for each algorithm. Let $\bar i$ be given by $\bar 1=2$ and $\bar 2 =1$. Algorithm $i$ undergoes a sequence of $N_f^{\bar i}$ Bernoulli trials, each having a probability of success $\Pi^i$. The probability of success for algorithm $i$, which indicates at least one successful trial within the $N_f^{\bar i}$ trials, is given by:
\begin{equation}
\label{suc}
1-(1-\Pi^i)^{N_f^{\bar i}}.
\end{equation}
By defining the efficiency index
\begin{equation}
\label{pi100}
\Pi_{100}^i = 1 - (1-\Pi^i)^{\frac{100}{N_f^i}},
\end{equation}
we can rewrite the probability of success,  \eqref{suc}, as
\[
1-(1-\Pi_{100}^i)^{N_f^1 N_f^2/100}.
\]
The index $\Pi_{100}^i$ represents a synthetic probability of success per 100 function evaluations. Therefore, if we conduct multiple independent trials of algorithm $i$, resulting in $100m$ function evaluations, the probability of achieving at least one success is given by $1-(1-\Pi_{100}^i)^m$; in this experiment $m=N_f^1 N_f^2$.

To compare algorithms, we use both metrics, $N_s$ and $\Pi_{100}$, as defined by equations \eqref{ns} and \eqref{pi100}, respectively. The $\Pi_{100}$ 
metric is suitable for problems with a fixed budget of function evaluations. Conversely, 
the $N_s$ metric is particularly 
relevant when the number of function evaluations is not predetermined.
%
%
%

\subsection{Initial Conditions}
\label{73}

The algorithm starts with a user-defined value $\mu_0$ and $\sigma_0$. If these are not provided, $\mu_0$ is chosen randomly in $[x_{min}, x_{max}]$ and $\sigma_0=x_{max}-x_{min}$.
In the initial step, $n_0$  samples are taken where $n_0$ is either user provided or takes the default value $n_0=10$. 
Without adaptivity, the number of sample points per iteration is always $n_0$. 
By Theorem \ref{mainteo}, for local convergence, we need a number of points larger than $6p+2$, where $p>1$. This number $n_0$ satisfies this condition for $p$ close to $1$. The smallest number of points required for the Theorem to hold would be $n_0=9$. 

\subsection{Parameters}
\label{parsec}

In this section, we present the chosen values for each parameter and provide justifications for these selections. These values are organized on a section-by-section basis and can be found in Table \ref{defaults0}. By default, the proposed algorithm uses rejection sampling (Section \ref{rejesec}), adaptivity (Section \ref{adaptivitysec}), sparse  sampling (Section \ref{sparsesec}), restarting (Section \ref{restartsec}), and postprocessing (Section \ref{postprocesssec}).

\subsubsection{Parameters in Section \ref{rejesec}}
The rejection sampling parameter $p$ should be set below $1$ to prevent excessive dependence between consecutive samples. In our experiments, a value of $p=0.75$ has proven reasonable, as it strikes a balance between minimizing the number of function evaluations and ensuring sufficient renewal of the sample points.

\subsubsection{Parameters in Section  \ref{mee}}
The parameter $m$, determining the confidence interval for the estimator of $\beta_i$, \eqref{bi}, is set to $m=1$.

\subsubsection{Parameters in Section \ref{timestepsec}}
The time stepping parameters,  $\gamma_i$ and $\upsilon_i$, 
	were chosen close to their maximal values as this reduces the number of iterations without decreasing performance illustrated in Section \ref{deppars}.

\subsubsection{Parameters in Section \ref{ext}} 
The coefficient $\varpi$ associated with function extension beyond the domain has minimal impact on most simulations, as function evaluation outside the domain is rare. In our experiments, we set $\varpi=10$ and observed that performance  does not change with 
alternative reasonable values for $\varpi$.

\subsubsection{Parameters in Section \ref{lfi}} 
For quadratic and linear functions, where exact interpolation leads to unbounded time steps, or near the end of the algorithm where smooth functions exhibit a high-quality quadratic approximation near a minimum, the time step is $h_{max}$. In most other cases, the time step determined by the discussion in Section \ref{timestepsec} is significantly smaller than $h_{max}$. As such, our choice of $h_{max}$ and the contraction parameter $\vartheta$ was guided to enhance convergence performance for quadratic or linear functions, without substantially impacting the performance for other functions.

\subsubsection{Parameters in Section \ref{adaptivitysec}} 
The values for the maximum and minimum number of samples, $n_{min}=6$ and $n_{max}=10$
and the initial sample size $n_0=10$,  were determined based on the following considerations. For uniformly convex functions like the parabola $f(x)=x^2$, our algorithm uses about 50 function evaluations. Consequently, significantly increasing $n_0$ or $n_{max}$ would negatively impact performance. On the other hand, performing least squares requires at least three points, and obtaining meaningful error bounds requires four or more points. Therefore, we chose $n_{min}=6$.

\subsubsection{Parameters in Section \ref{stop}}
\label{parstop}

To accommodate varying domain magnitudes for different functions, we set two parameters, $\hat \sigma_{target}$ and $\hat \sigma_{min}$, and define $ \sigma_{target}=\hat \sigma_{target} (x_{max}-x_{min})$ and $ \sigma_{min}=\hat \sigma_{min} (x_{max}-x_{min})$. We choose termination parameters $\hat \sigma_{target}$ and $\delta_f$ such that our algorithm's average error conditional on success, $\Delta_c = 1.4\times 10^{-5}$, is comparable to the average $\Delta_c = 1.7\times 10^{-5}$ for four competing algorithms (see table \ref{allresults}). This enables meaningful comparison of function evaluations between algorithms. The default value for $\kappa$ is $\kappa=1$.
The parameters $N^*_f$ and $N^*_i$ are chosen somewhat arbitrarily, as they are rarely reached in our experiments and are included here for completeness.

\begin{table}[h]
	\begin{tabular}{|l|l|l|l|}
		\hline
		{\bf Par.}	&{\bf Description}& {\bf Default}  \\ \hline
		$\varpi$&Section \ref{ext}& $10$\\\hline
		$n_{min}$&Section \ref{adaptivitysec}&$6$ \\ \hline
		$n_{max}$&Section \ref{adaptivitysec}&$10$ \\ \hline
		$h_{max}$&Section \ref{lfi}&$1000$\\ \hline
		$N^*_f$&Section \ref{stop}s&$1000$\\ \hline
		$N^*_i$&Section \ref{stop}&$1000$\\ \hline
		$\hat \sigma_{target}$&Section \ref{stop}&$5 \times 10^{-5}$\\ \hline
		$\hat \sigma_{min}$&Section \ref{stop}&$10^{-8}$\\ \hline
		$\delta_f$ & Section \ref{stop}&$1.25\times 10^{-6}$\\ \hline
		$\kappa$ & Section \ref{stop}&$1$\\ \hline
		$p$&Section \ref{rejesec}&0.75\\ \hline
		$\gamma_1$&Section \ref{timestepsec}&$0.2$\\ \hline
		$\gamma_2$&Section \ref{timestepsec}&$0.2$\\ \hline
		$\upsilon_1$&Section \ref{timestepsec}&$0.2$\\ \hline
		$\upsilon_2$&Section \ref{timestepsec}&$0.2$\\ \hline
		$m$&Section \ref{mee}&$1.$\\ \hline
		$\vartheta$&Section \ref{lfi}&$0.95$\\ \hline
	\end{tabular}
	\medskip
	\caption{Default parameter values}
	\label{defaults0}
\end{table}

\subsection{Numerical results}
\label{ares}

Table \ref{allresults} compares the results of our algorithm with those of the four competitors, presenting averaged values across all tested function classes. 
	All the numerical tests were performed on a Mac Pro (2019), 
	3.5GHz 8-Core Intel Xeon W,  32 Gb RAM.
As previously explained in Section \ref{comparasiosec}, $N_f$ represents the number of function evaluations, $\Pi$ denotes the success probability, $N_s$ is the average number of function evaluations per success given in \eqref{ns}, $\Pi_{100}$ is the efficiency index in \eqref{pi100}, $\tau$ is the average run time per optimization problem in seconds, $\Delta$ is the optimization gap, and $\Delta_c$ is the optimization gap conditional on success. Our algorithm outperforms the competing algorithms using 11 to 26 times fewer  function evaluations. Furthermore, it achieves the highest efficiency index and the smallest average number of function evaluations per success. 

Because our stopping criterion was selected to ensure that our algorithm's average value of $\Delta_c$ is similar to the average values of the other algorithms, our algorithm naturally ranks between the competing algorithms in this criterion. Regarding success probability, only Differential Evolution achieves better values for $\Pi$ and $\Delta$, but it requires approximately 25 times more function evaluations.

Table \ref{boostingtable} displays the averaged results using boosting (Section \ref{boosting}), with $N_i$ representing the average total number of iterations. As seen in that table, a single boosting cycle increases the success probability $\Pi$ to a value higher than that of Differential Evolution, using around 15 times fewer function evaluations. Additional boosting cycles further improve the probability of success, despite a slight decrease in the efficiency index $\Pi_{100}$ and an increase in average number of function evaluations per success, $N_s$. Notably, the number of function evaluations, $N_f$, grows sublinearly for each boosting cycle.  The use of boosting is particularly interesting in situations where function evaluations are costly but further exploration of the space is desirable. Thus, for example, we can run a single boosting cycle, increasing the probablity of finding a minimum,  without doubling the number of function evaluations, which would happen if we simply ran the algorithm again.

It is important to mention that the run time of our algorithm is larger but of the same order of magnitude as the other algorithms, even though it has not been optimized for speed and the test functions are computationally inexpensive. For computationally expensive functions, our algorithm is expected to be substantially faster because it uses fewer function evaluations. 

\begin{TBL}
	{"allresults",  (normalizeassociation/@combinedreport[{}//selecttestfunctions])// displayreportastable//First//#[[All, {1, 2, 4, 5, 6,7, 8}]]&}
\end{TBL}
	%
	
	\begin{center}
		\begin{small}
			\begin{table}


\begin{tabular}{|c|c|c|c|c|c|c|c|}
	\hline
	Optimizer & \(N_f\) & \(\Pi\) & \(N_s\) & \(\Pi_{100}\) & \(\tau\) & \(\Delta\) & \(\Delta_c\) \\
	\hline
	Nelder Mead & \(1639.9\) & \(0.67\) & \(2447.6\) & \(0.065\) & \(0.12\) & \(0.13\) & \(0.000017\) \\
	\hline
	Random Search & \(3902.6\) & \(0.9\) & \(4336.2\) & \(0.057\) & \(0.16\) & \(0.032\) & \(0.000012\) \\
	\hline
	Differential Evolution & \(3689.4\) & \(0.96\) & \(3843.1\) & \(0.085\) & \(0.16\) & \(0.0055\) & \(8.5\times 10^{-6}\) \\
	\hline
	Simulated Annealing & \(2110.0\) & \(0.87\) & \(2425.3\) & \(0.093\) & \(0.085\) & \(0.024\) & \(0.000031\) \\
	\hline
	Algorithm & \(149.8\) & \(0.94\) & \(159.4\) & \(0.84\) & \(0.3\) & \(0.014\) & \(0.000014\) \\
	\hline
\end{tabular}
				\medskip

				\caption{Averaged results for all test functions}
				\label{allresults}
				
			\end{table}	
		\end{small}
	\end{center}

\begin{TBL}
	{"boosting",replist = normalizeassociation/@report[selecttestfunctions[{}],
		{
			Association[],
			Association["boosting" -> 1],
			Association["boosting" -> 2],
			Association["boosting" -> 3],
			Association["boosting" -> 4],
			Association["boosting" -> 5]
		}] // displayreportastable //First//(#[[All, {1, 2, 4,9, 5, 6,7,8}]]&)}
\end{TBL}
	
	\begin{center}
		\begin{small}
			\begin{table}
				\begin{tabular}{|c|c|c|c|c|c|c|c|c|}\hline
Boosting cycles&\(N_f\)&\(\Pi\)&$N_s$&\(\Pi_{100}\)&\(\tau\)&\(N_i\)&\(\Delta\)&\(\Delta_c\)\\\hline
0&\(148.9\)&\(0.94\)&158.4&\(0.85\)&\(0.32\)&\(209.4\)&\(0.014\)&\(0.000012\)\\\hline
1&\(234.5\)&\(0.97\)&241.7&\(0.77\)&\(0.74\)&\(437.8\)&\(0.0063\)&\(0.000011\)\\\hline
2&\(291.2\)&\(0.98\)&297.1&\(0.73\)&\(1.\)&\(610.\)&\(0.0042\)&\(
9.5\times 10^{-6}\)\\\hline
3&\(328.9\)&\(0.98\)&335.6&\(0.71\)&\(1.2\)&\(716.1\)&\(0.0034\)&\(      
9.3 \times 10^{-6}\)\\\hline
4&\(356.9\)&\(0.99\)&360.5&\(0.69\)&\(1.3\)&\(797.5\)&\(0.0026\)&\(0.00001\)\\\hline
5&\(381.0\)&\(0.99\)&384.8&\(0.69\)&\(1.5\)&\(876.9\)&\(0.0021\)&\(      
8.6 \times 10^{-6}\)\\\hline\end{tabular}
				\medskip
				\caption{Performance for different number of boosting cycles}
				\label{boostingtable}
			\end{table}		
		\end{small}
	\end{center}

\subsection{Dependence on parameters}
\label{deppars}

Table \ref{parstable} analyzes the performance of the algorithm for various choices of the parameters $\gamma_i$ and $\upsilon_i$ from Section \ref{timestepsec}.  As evident from the table, the selection of these parameters has minimal impact on the number of function evaluations ($N_f$), algorithm performance ($\Pi$),  efficiency ($\Pi_{100}$) and function evaluations per success ($N_s$), despite a significant increase in the number of iterations ($N_i$) due to decreased time steps. This behavior is due to the following.  Since the variables $\mu_j$ and $\sigma_j$ also exhibit less change from iteration to iteration, the probability $\pi_k$ in \eqref{pik} approaches 1, enabling the reuse of more prior samples.
It is worth noting, however, that both $\Pi$ and $\Pi_{100}$ seem to decrease for very small values of $\gamma_i$ and $\upsilon_i$. This observation may be attributed to the fact that larger errors associated with higher values of $\gamma_i$ and $\upsilon_i$ facilitate better exploration of the state space.

Table \ref{comparison} illustrates the effect of various options. Firstly, it is important to note that rejection sampling considerably decreases in  function evaluations. However, rejection sampling also is a primary source of algorithmic complexity. In fact, the execution time (which is a proxy for algorithmic complexity) without rejection sampling is comparable to Simulated Annealing, as seen in Table \ref{allresults}. This means that the run time overhead can be mainly attributed to rejection sampling. 
Secondly, even in the absence of rejection sampling, our algorithm outperforms all competing algorithms, requiring 2 to 5 times fewer function evaluations. This superior performance can be attributed to our highly optimized method, which allows for the largest time step within the prescribed error bounds. Thirdly, due to the increase in function evaluations, the value of $\Delta$ without rejection sampling is smaller than the corresponding $\Delta$ with rejection sampling, as can be seen  in Table \ref{parstable}. Lastly, we see that both adaptivity and sparse  sampling contribute to reducing the execution time ($\tau$) and the number of iterations ($N_i$).

Finally, in Table \ref{restarttable}, we can see the effect of the restart strategy. While this strategy
substantially increases  function evaluations, it also improves performance and reduces
the optimization gap $\Delta$.

\begin{TBL}
	{"pars",replist = normalizeassociation/@report[selecttestfunctions[{}],
		{
			Association[],
			Association[{"gamma1" -> 0.02, "gamma2" -> 0.02,"u1"->0.02, "u2"->0.02}],
			Association[{"gamma1" -> 0.002, "gamma2" -> 0.002,"u1"->0.002, "u2"->0.002}]
		}] // displayreportastable //First//(#[[All, {1, 2, 4,9, 5, 6,7,8}]]&)}
\end{TBL}

\begin{center}
	\begin{small}
		\begin{table}
			\begin{tabular}{|c|c|c|c|c|c|c|c|c|}\hline
Parameters&\(N_f\)&\(\Pi\)&$N_s$&\(\Pi_{100}\)&\(\tau\)&\(N_i\)&\(\Delta\)&\(\Delta_c\)\\\hline
$\gamma_i=\upsilon_i=0.2$&\(150.0\)&\(0.94\)&159.6&\(0.84\)&\(0.36\)&\(220.5\)&\(0.015\)&\(0.000013\)\\\hline
$\gamma_i=\upsilon_i=0.02$&\(179.2\)&\(0.94\)&190.6&\(0.8\)&\(1.4\)&\(1059.6\)&\(0.013\)&\(0.000016\)\\\hline
$\gamma_i=\upsilon_i=0.002$&\(180.7\)&\(0.85\)&212.6&\(0.66\)&\(7.6\)&\(6885.9\)&\(0.016\)&\(0.000034\)\\\hline\end{tabular}			
			\medskip
			\caption{Dependence on parameters}
			\label{parstable}
		\end{table}		
	\end{small}
\end{center}

\begin{TBL}
	{"comparisonresults2",  normalizeassociation/@report[selecttestfunctions[{}], {
			Association["resamplingmethod" -> "plain", "nmin" -> 10,  "nmax" -> 10, "sparseresampling" -> False],
			Association["resamplingmethod" -> "plain","nmin" -> 10, "nmax" -> 10, 
			"sparseresampling" -> True],
			Association["resamplingmethod" -> "plain", "sparseresampling" -> False],			
             Association["resamplingmethod" -> "plain", "sparseresampling" -> True],			
			Association[ "nmin" -> 10,  "nmax" -> 10, "sparseresampling" -> False],
			Association["nmin" -> 10, "nmax" -> 10, "sparseresampling" -> True],
			Association["sparseresampling" -> False],							
			Association[]
		}] // displayreportastable//First//(#[[All, {1, 2, 4,9, 5, 6,7,8}]]&)}
\end{TBL}
\begin{center}
	\begin{small}
		\begin{table}
			\begin{tabular}{|c|c|c|c|c|c|c|c|c|}
	\hline
\(options\)&\(N_f\)&\(\Pi\)&$N_s$&\(\Pi_{100}\)&\(\tau\)&\(N_i\)&\(\Delta\)&\(\Delta_c\)\\\hline
Plain&\(755.0\)&\(0.95\)&794.7&\(0.33\)&\(0.092\)&\(75.3\)&\(0.00084\)&\(0.000038\)\\\hline
Plain+SR&\(495.3\)&\(0.97\)&510.6&\(0.52\)&\(0.12\)&\(115.8\)&\(0.0022\)&\(0.000014\)\\\hline
Plain+A&\(595.6\)&\(0.97\)&614.0&\(0.44\)&\(0.1\)&\(84.9\)&\(0.00084\)&\(0.000023\)\\\hline
Plain+SR+A&\(374.4\)&\(0.97\)&380.0&\(0.62\)&\(0.12\)&\(128.4\)&\(0.005\)&\( 
9.1\times10^{-6}\)\\\hline
RS&\(224.1\)&\(0.97\)&231.0&\(0.79\)&\(0.92\)&\(344.9\)&\(0.0056\)&\(0.000013\)\\\hline
RS+SR&\(188.2\)&\(0.96\)&196.0&\(0.83\)&\(0.64\)&\(303.8\)&\(0.0074\)&\(0.000012\)\\\hline
RS+A&\(167.2\)&\(0.95\)&176.0&\(0.82\)&\(0.56\)&\(248.7\)&\(0.012\)&\(0.000013\)\\\hline
RS+SR+A
&\(147.8\)&\(0.93\)&158.9&\(0.84\)&\(0.32\)&\(212.5\)&\(0.014\)&\(0.000014\)\\\hline\end{tabular}
			\medskip
			\caption{Comparison of various options -
				Plain= no Rejection Sampling, 		RS = Rejection Sampling, A=Adaptivity, SS=Sparse  Sampling}
			\label{comparison}
		\end{table}		
	\end{small}
\end{center}


\begin{TBL}
	{"restart",replist = normalizeassociation/@report[selecttestfunctions[{}],
		{
			Association[],
			Association[{"restart"->False}]
		}] // displayreportastable //First//(#[[All, {1, 2, 4,9, 5, 6,7,8}]]&)}
\end{TBL}
%

\begin{center}
	\begin{small}
		\begin{table}
			\begin{tabular}{|c|c|c|c|c|c|c|c|c|}\hline
Restart&\(N_f\)&\(\Pi\)&$N_s$&\(\Pi_{100}\)&\(\tau\)&\(N_i\)&\(\Delta\)&\(\Delta_c\)\\\hline
True&\(148.7\)&\(0.93\)&160.0&\(0.84\)&\(0.3\)&\(205.4\)&\(0.015\)&\(0.000016\)\\\hline
False&\(105.6\)&\(0.82\)&128.9&\(0.81\)&\(0.27\)&\(145.\)&\(0.027\)&\(0.000036\)\\\hline\end{tabular}
			
			\medskip
			\caption{Effect of the restart strategy}
			\label{restarttable}
		\end{table}		
	\end{small}
\end{center}

%


%

\subsection{Noisy functions}
\label{noisy}

In certain applications, the objective function is random.
This is the case in  mini-batching in machine learning.
To model this category of functions, which we call noisy functions, we employ additive Gaussian noise. Consequently, we substitute the deterministic objective function $f$ with
\[
\hat f(x)=f(x)+\zeta X,
\]
where $X$ denotes a Gaussian random variable with zero mean and unit variance, and $\zeta\in \Rr$ represents the noise level.

We ran our algorithm using various function selections and noise levels.
We disabled restarting, sparse  sampling, and adaptivity features. Further, 
because $\hat f$ is a random variable,  not a deterministic function, we had to modify the  stopping criteria.
We employed a single stopping criterion: the algorithm stops when  
 $\sigma_j\leq \hat \sigma_{target}$, where $\hat \sigma_{target}=5\times 10^{-5}$. 
 The functions under consideration do not possess minima at their boundaries, so there is not a distinct stopping criteria for points at the boundary. 

Because $\hat f$ is a random variable, to define the success of the optimization process, it is not
suitable to look at the value of $\hat f$ at the algorithm's output $\tilde x$.
Instead, we compare the minimizer $\bar x$ of $f$ with $\tilde x$: the algorithm is successful if
$|\bar x-\tilde x|\leq 0.05 (x_{max}-x_{min})$. 
In our numerical results, Tables \ref{noisyDeJong1table}, \ref{noisySchwefeltable}, and \ref{noisyDeltadprime10table}, we show the average error
$\Delta x=E|\bar x-\tilde x|$ and the error conditional on success 
$\Delta_c x=E\big[|\bar x-\tilde x| \ \big|\   |\bar x-\tilde x|\leq 0.05 (x_{max}-x_{min})\big]$. 

Tables \ref{noisyDeJong1table}, \ref{noisySchwefeltable}, and \ref{noisyDeltadprime10table} present the results for each of the functions in \ref{noisyDeJongimages}, \ref{noisySchwefelimages}, and \ref{noisyDeltadprime10images}, respectively. 
As we see in the numerical results, even for large noise levels
the results are substantially better than chance. 
If the algorithm were to produce a uniformly distributed output, approximately $10\%$ of the runs would be considered successful, since for a uniform distribution in
$[x_{min}, x_{max}]$, 
$10\%$ of the points would fall within a relative error of $\pm 5\%$ of $\bar x$. 
The  algorithm's performance under low noise ($\zeta = 1\%$) is comparable to the performance without noise for all test functions.
In the medium noise scenario ($\zeta = 10\%$), the algorithm performs well, generating points near the global minimum or suitable local minima. While the quality of the proposed minima significantly deteriorates in the high-noise case ($\zeta = 50\%$), the proposed minima still lie closer to the global minimum with a noticeably higher probability than if the points were chosen randomly. Lastly, we observe a moderate increase in the number of function evaluations as the noise level rises.


\begin{TBL}
	{"noisyDeJong1",normalizeassociation /@ 
		noisyreport[
		selecttestfunctions[
		"DeJong1"], {Association["noisevector" -> {0, .0}], 
			Association["noisevector" -> {0, .01}], 
			Association["noisevector" -> {0, .1}], 
			Association["noisevector" -> {0, 0.5}]}] // displayreportastable//First//(#[[All, {1, 2, 4,9, 5, 6,7,8}]]&)}
\end{TBL}
\begin{center}
	\begin{small}
		\begin{table}
			\begin{tabular}{|c|c|c|c|c|c|c|c|c|}\hline
Noise Level&\(N_f\)&\(\Pi\)&$N_s$&\(\Pi_{100}\)&\(\tau\)&\(N_i\)&\(\Delta x \)&\(\Delta_c x\)\\\hline
\(0\)&\(125.6\)&\(1.\)&$125.6$&\(1.\)&\(0.076\)&\(45.1\)&\( 
5.1\times 10^{-6}\)&\(      
5.1 \times 10^{-6}\)\\\hline
\(0.01\)&\(174.5\)&\(1.\)&$174.5$&\(1.\)&\(0.48\)&\(247.1\)&\(0.0094\)&\(0.0094\)\\\hline
\(0.1\)&\(181.8\)&\(0.73\)&$249.0$&\(0.51\)&\(0.51\)&\(261.8\)&\(0.034\)&\(0.022\)\\\hline
\(0.5\)&\(190.7\)&\(0.37\)&$515.4$&\(0.22\)&\(0.59\)&\(302.5\)&\(0.078\)&\(0.021\)\\\hline
\end{tabular}
			\medskip
			\caption{Performance for the function in Figure \ref{noisyDeJongimages}, $\Delta x$, and
			$\Delta_c x$ range between 0 and 1 and are scaled with the domain. }
			\label{noisyDeJong1table}
		\end{table}		
	\end{small}
\end{center}

\begin{TBL}
	{"noisySchwefel",normalizeassociation /@ 
		noisyreport[
		selecttestfunctions[
		"Schwefel"], {Association["noisevector" -> {0, .0}], 
			Association["noisevector" -> {0, .01}], 
			Association["noisevector" -> {0, .1}], 
			Association["noisevector" -> {0, 0.5}]}] // displayreportastable//First//(#[[All, {1, 2, 4,9, 5, 6,7,8}]]&)}
\end{TBL}
\begin{center}
	\begin{small}
		\begin{table}
				\begin{tabular}{|c|c|c|c|c|c|c|c|c|}\hline Noise Level&\(N_f\)&\(\Pi\)&$N_s$&\(\Pi_{100}\)&\(\tau\)&\(N_i\)&\(\Delta x\)&\(\Delta_c x\)\\\hline
$0$&\(141.4\)&\(0.66\)&$214.2$&\(0.53\)&\(0.21\)&\(134.6\)&\(0.25\)&\(0.012\)\\\hline
$0.01$&\(178.\)&\(0.72\)&$247.2$&\(0.51\)&\(0.49\)&\(272.8\)&\(0.21\)&\(0.0089\)\\\hline
$0.1$&\(184.7\)&\(0.48\)&$384.8$&\(0.3\)&\(0.57\)&\(313.8\)&\(0.31\)&\(0.011\)\\\hline
$0.5$&\(191.9\)&\(0.2\)&$959.5$&\(0.11\)&\(0.59\)&\(322.5\)&\(0.39\)&\(0.015\)\\\hline
\end{tabular}
			\medskip
			\caption{Performance for the function in Figure \ref{noisySchwefelimages}.}
			\label{noisySchwefeltable}
		\end{table}		
	\end{small}
\end{center}

\begin{TBL}
	{"noisyDeltadprime10",normalizeassociation /@ 
		noisyreport[
		selecttestfunctions[
		"Deltadprime10"], {Association["noisevector" -> {0, .0}], 
			Association["noisevector" -> {0, .01}], 
			Association["noisevector" -> {0, .1}], 
			Association["noisevector" -> {0, 0.5}]}] // displayreportastable//First//(#[[All, {1, 2, 4,9, 5, 6,7,8}]]&)}
\end{TBL}
\begin{center}
	\begin{small}
		\begin{table}
				\begin{tabular}{|c|c|c|c|c|c|c|c|c|}\hline
Noise Level&\(N_f\)&\(\Pi\)&$N_s$&\(\Pi_{100}\)&\(\tau\)&\(N_i\)&\(\Delta x\)&\(\Delta_c x\)\\\hline
$0$&\(140.5\)&\(0.97\)&$144.8$&\(0.92\)&\(0.18\)&\(107.9\)&\(0.0076\)&\(0.0043\)\\\hline
$0.01$&\(169.4\)&\(0.93\)&$182.1$&\(0.79\)&\(0.41\)&\(212.2\)&\(0.012\)&\(0.0036\)\\\hline
$0.1$&\(187.1\)&\(0.77\)&$243.0$&\(0.54\)&\(0.55\)&\(278.7\)&\(0.052\)&\(0.005\)\\\hline
$0.5$&\(195.2\)&\(0.38\)&$513.7$&\(0.22\)&\(0.61\)&\(311.6\)&\(0.14\)&\(0.0073\)\\\hline
\end{tabular}
			\medskip
			\caption{Performance for the function in Figure \ref{noisyDeltadprime10images}.}
			\label{noisyDeltadprime10table}
		\end{table}		
	\end{small}
\end{center}

\begin{figure}
	\subfloat[]{\includegraphics[width = .25\textwidth]{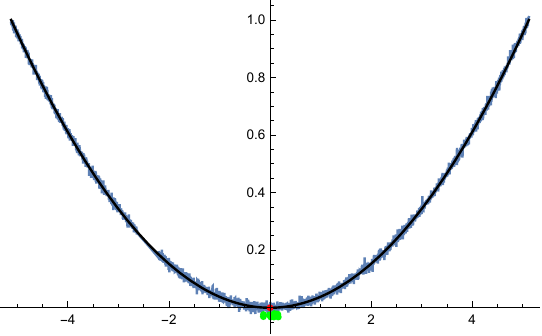}} 
	\hspace{20pt}
	\subfloat[]{\includegraphics[width = .25\textwidth]{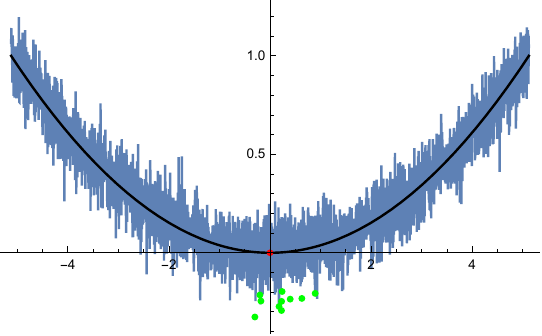}}
	\hspace{20pt}
	\subfloat[]{\includegraphics[width = .25 \textwidth]{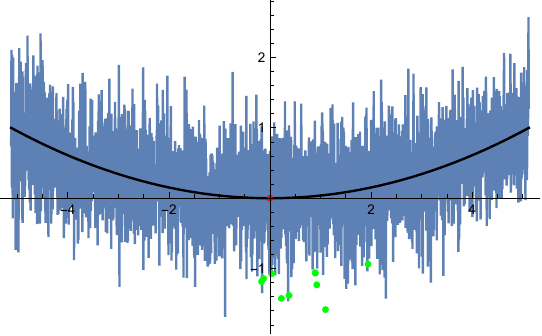}}
	\caption{Noisy versions of $f(x)=x^2$ with 1\%, 10 \% and 50\% noise.
		Here, $f$ is scaled so that the domain
		amplitude is $1$. In red, the global minimum of $f$, in green, 10 realizations of the minimum of the noisy version of $f$.}
	\label{noisyDeJongimages}
\end{figure}

\begin{figure}
	\subfloat[]{\includegraphics[width = .25\textwidth]{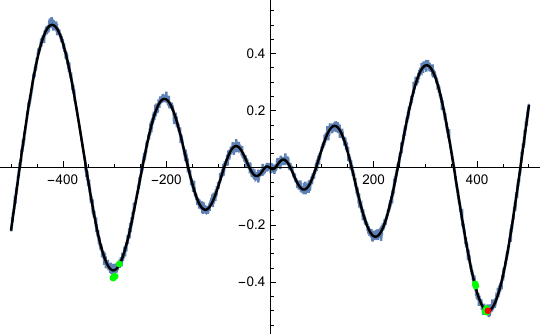}} 
	\hspace{20pt}
	\subfloat[]{\includegraphics[width = .25\textwidth]{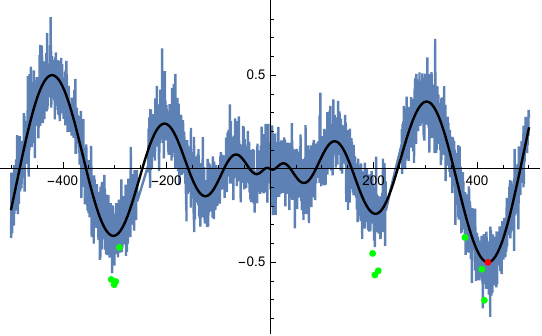}}
	\hspace{20pt}
	\subfloat[]{\includegraphics[width = .25 \textwidth]{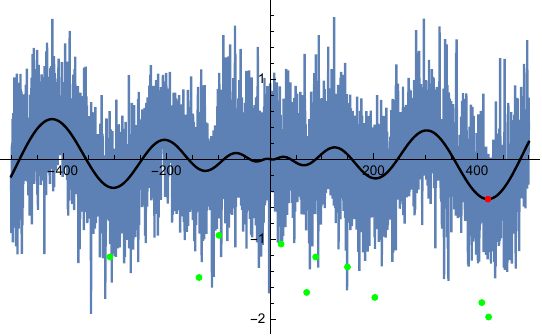}}
	\caption{Noisy versions of $f(x)=-x \sin \left(\sqrt{| x| }\right)$ with 1\%, 10 \% and 50\% noise. Here, $f$ is scaled so that the domain
		amplitude is $1$.  In red, the global minimum of $f$, in green, 10 realizations of the minimum of the noisy version of $f$.}
	\label{noisySchwefelimages}
\end{figure}

\begin{figure}
	\subfloat[]{\includegraphics[width = .25\textwidth]{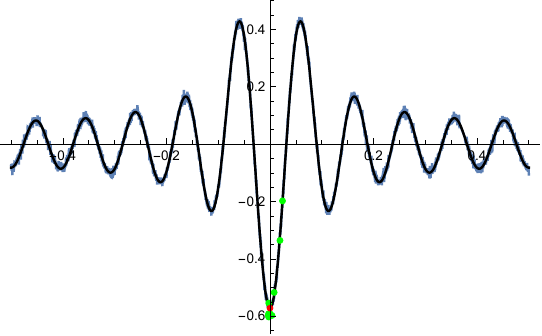}} 
	\hspace{20pt}
	\subfloat[]{\includegraphics[width = .25\textwidth]{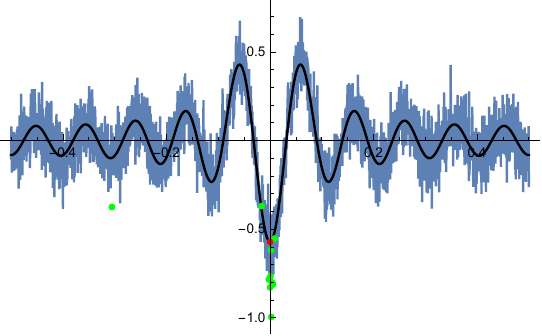}}
	\hspace{20pt}
	\subfloat[]{\includegraphics[width = .25 \textwidth]{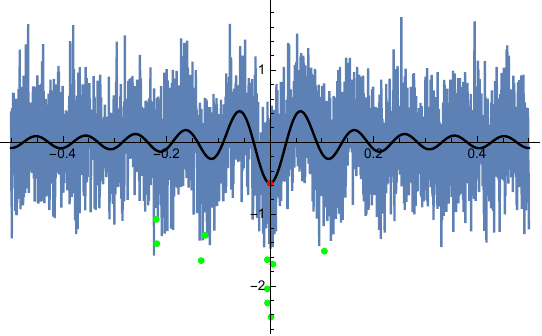}}
	\caption{Noisy versions of $f(x)=-\sum_{k=1}^10 (2\pi k)^2 \cos(2 \pi k)$ with 1\%, 10 \% and 50\% noise. Here, $f$ is scaled so that the domain
		amplitude is $1$.  In red, the global minimum of $f$, in green, 10 realizations of the minimum of the noisy version of $f$.}
	\label{noisyDeltadprime10images}
\end{figure}

\section{Conclusions and future work}
\label{extensionsec}

Our algorithm outperforms standard algorithms in a broad class of functions both in terms of function evaluations and success rate. This performance is consistently good
across all classes tested, which ranges from easy optimization problems (eg smooth, uniformly convex functions) to functions that challenge standard
algorithms or require a substantial number of function evaluations (eg non-convex, non-smooth multimodal functions), 
see Appendix \ref{breakdown}.
While the current implementation is in Mathematica, we are developing a Python implementation of our code to give access to a broader audience.

We plan to extend our work to the
higher-dimensional case.
In this case,  there are multiple strategies to consider.  
First
approaches include considering Gaussians with scalar ($\sigma I$) or
diagonal covariances.  For general covariance ($\Sigma$), the role of natural gradients (see Section \ref{prior}) may be crucial, as the time scales can significantly vary across the principal directions of $\Sigma$. 
Given the complexity of multidimensional
geometry, we are interested in exploring covariance adaptation
similar to CMA-ES (see Section \ref{prior}).
Moreover, calculating the quadratic approximation using least squares may be infeasible in high dimensions. 
 As a result, alternative formulas for the quadratic approximation $q$ and variance reduction methods may be essential. 
Additionally, creating strategies to use the
estimate $q_{j-1}$ for better estimates of $q_j$ may prove
advantageous. On the more theoretical side, it would be interesting to investigate further global convergence properties that explain the high success rate of our algorithm.

%

\appendix

\section{Integral inequalities for homogeneous functions}
\label{homogeneousA}

Here, we present various results and integral inequalities 
on homogeneous functions
that simplify various bounds in our proofs.  A function $g:\Rr^n\to \Rr$ is homogeneous of degree $\alpha$ if  $g(\lambda x)  = \lambda^{\alpha} g(x)$ for all $\lambda \in \Rr$.
We say that $g$ is positively homogeneous if the preceding 
identity holds for all $\lambda>0$. 
\begin{lem}
	\label{pseudonorm}
	Let $h\in C(\Rr^n)$ satisfy:
	\begin{enumerate}
		\item $h \geq 0;$
		\item if $h(x) = 0$, then $x=0$;
		\item $h$ is positively 1-homogeneous;
		\item the set  $\{x\in \Rr^n: h(x)=1\}$ is compact.
	\end{enumerate}
	Then, if $g\in C(\Rr^n)$ is positively homogeneous of degree $\alpha>0$, 
	there exists a positive constant depending only on $h$ such that
	\begin{equation}
		\label{equivnorms}
		|g(x)| \leq C \left[h(x)\right]^{\alpha}.
	\end{equation}
\end{lem}	
\begin{proof} 
	It is
	sufficient to check \eqref{equivnorms} for $x\neq 0$ as
	homogeneity ensures $g(0)=h(0)=0$.
	Since $g$ is continuous and the set  where $h(x)=1$
	is compact, we have
	\[
	\sup_{h(x) =1} |g(x)| = C < \infty.
	\]
	Moreover, from the $\alpha$-homogeneity of $g$, we have
	\[
	|g(x)| = \left|g\left(h(x)\frac{x}{h(x)}\right) \right| = \left[h(x)\right]^{\alpha} \left|g\left(\frac{x}{h(x)}\right) \right| \leq C  \left[h(x)\right]^{\alpha},
	\]
	since $|g(x/h(x))|$ is bounded by $C$ because $h(x/h(x))=1$.
\end{proof}

Note that if $h:\Rr^n\to \Rr$ is a norm, then it satisfies the conditions in the lemma. 

\begin{lem}
	\label{normie}
	For $x=(x_1,\hdots,x_n)\in \Rr^n$ and any norm $\|\cdot\|$ in $\Rr^n$, we have
	\begin{equation}
		\label{normandcoord}
		\|x\|^{\alpha} \leq C_{\alpha}\left(|x_1|^{\alpha} +\hdots +|x_n|^{\alpha}  \right),
	\end{equation}
	for any real positive $\alpha$.
\end{lem}
\begin{proof}
	Since any
	norm is homogeneous of degree 1,
 $g(x)= \|x\|^{\alpha}$ is homogeneous of degree $\alpha$. Moreover, $ h(x)=\left(|x_1|^{\alpha} +\hdots +|x_n|^{\alpha}  \right)^{1/\alpha}$ verifies all the properties in  Lemma \ref{pseudonorm}. Then, we use \eqref{equivnorms} to find
	\[
	\|x\|^{\alpha} \leq C_{\alpha} \left(|x_1|^{\alpha} +\hdots +|x_n|^{\alpha}  \right).\qedhere
	\]
\end{proof}

As a result of the preceding two lemmas, we obtain the following estimates
for integrals with respect to Gaussians of homogeneous functions. 

\begin{lem}
	\label{hplem}
	Let $g\in C(\Rr^n)$ be positively homogeneous of degree $\alpha$. Then,
	\[\int |g(x_1, \hdots x_n)| \Gamma_{\mu,\sigma}(x_1)\hdots \Gamma_{\mu,\sigma}(x_n)dx_1\hdots dx_n=O(|\mu|^{\alpha}+\sigma^{\alpha}).
	\] 
\end{lem}
\begin{proof}
	We begin by changing variables $x_i \rightarrow y_i=x_i-\mu$ in the given integral.  That is, we write $y=x-m$, with $y,m\in\Rr^n$ and $m=(\mu,\hdots,\mu)$. Accordingly, 
	\begin{align*}
		\int |g(x)| \Gamma_{\mu,\sigma}(x_1)\hdots \Gamma_{\mu,\sigma}(x_n)dx&=
		\int |g(m+y)| \Gamma_{0,\sigma}(y_1)\hdots \Gamma_{0,\sigma}(y_n)dy\\
		& \leq \int C_1 \|m+y\|^{\alpha} \Gamma_{0,\sigma}(y_1)\hdots \Gamma_{0,\sigma}(y_n)dy,
	\end{align*}
	where we used \eqref{equivnorms} with $h(x) = \|x\|$ and $\|\cdot \|$ representing a fixed norm in $\Rr^n$.
	
	Now, using the triangular inequality,  we have
	\[
	\|m+y\|^{\alpha}  \leq \left(\|m\|+\|y\|\right)^{\alpha}= \left[\|\left(\|m\|,\|y\|\right)\|_{\ell^1}\right]^{\alpha} \leq C_2\left(\|m\|^{\alpha} + \|y\|^{\alpha}\right),
	\]
	where we used \eqref{normandcoord} in the last step. Therefore,
	\begin{align*}
		\int |g(x)| \Gamma_{\mu,\sigma}(x_1)\hdots \Gamma_{\mu,\sigma}(x_n)dx
		& \leq 	\int C_3\left(\|m\|^{\alpha} + \|y\|^{\alpha}\right) \Gamma_{0,\sigma}(y_1)\hdots \Gamma_{0,\sigma}(y_n)dy\\
		& \leq C_4 |\mu|^{\alpha} + C_3 \int \|y\|^{\alpha}\Gamma_{0,\sigma}(y_1)\hdots \Gamma_{0,\sigma}(y_n)dy,
	\end{align*}
	since \eqref{normandcoord} results in $\|m\|^{\alpha}\leq C |\mu|^{\alpha}$. Finally, by changing variables again, $y = \sigma z$,  we transform the Gaussians into standard Gaussian distributions. Thus,
	\begin{align*}
		\int |g(x)| \Gamma_{\mu,\sigma}(x_1)\hdots \Gamma_{\mu,\sigma}(x_n)dx & \leq C_4 |\mu|^{\alpha} + C_5 \sigma^{\alpha}\int \|z\|^{\alpha}\Gamma(z_1)\hdots \Gamma(z_n)dz\\
		& \leq C_4 |\mu|^{\alpha} + C_6 \sigma^{\alpha}\int \left(|z_1|^{\alpha} +\hdots +|z_n|^{\alpha}  \right)\Gamma(z_1)\hdots \Gamma(z_n)dz\\
		& \leq C_4 |\mu|^{\alpha} + C_7\sigma^{\alpha}\\
		& = O(|\mu|^{\alpha} + \sigma^{\alpha}),
	\end{align*}
	because $|z_i|^{\alpha}$ are integrable with respect to $\Gamma(z_i)$.
\end{proof}

\section{Test functions}
\label{testfunctions}

Here, we present the test functions used. These functions were selected to represent attributes and challenges relevant to optimization. The considered attributes and functions count per attribute
are shown in Table \ref{attributes}. 
 Note that each function can possess multiple attributes; for instance, uniformly convex functions are both strictly convex and convex. As a result, some functions belong to more than one of the individual classes we examine. 
 The plots of the functions organized by classes are displayed in figures 
	\ref{niceconvex}-\ref{discontinuous}.  In these figures, as in the following ones, the red dot represents the minimum of the function. 
	Table \ref{funlist} lists the
	corresponding expressions. To make the table easy to read, we display the unnormalized versions of the functions; in our experiments, all functions were multiplied by a factor so that they have oscillation 1. 
	
\begin{TBL}
	{"attributes", Prepend[{#, Length@selecttestfunctions[#]} & /@ fattributes[] ,{"attribute","number of functions"}]}
\end{TBL}
\begin{center}
	\begin{small}
		\begin{table}
			\begin{tabular}{|c|c|}\hline Attribute& Number of functions\\\hline 
Bimodal &\(6\)\\\hline
 Boundary min&\(8\)\\\hline
 Concave&\(6\)\\\hline
 Convex&\(9\)\\\hline
 Discontinuous&\(7\)\\\hline
 Interior min&\(44\)\\\hline
 Linear&\(2\)\\\hline
 Multimodal&\(34\)\\\hline
 Multiple global min&\(9\)\\\hline
 Non-convex&\(41\)\\\hline
 Non-smooth&\(12\)\\\hline
 Oscillatory&\(16\)\\\hline
 Quadratic&\(1\)\\\hline
 Smooth&\(38\)\\\hline
 Strictly convex&\(6\)\\\hline
 Uniformly convex&\(5\)\\\hline
 Unimodal&\(16\)\\\hline
 \end{tabular}			
			\medskip
			\caption{Function count per attribute.}
			\label{attributes}
		\end{table}
	\end{small}
\end{center}

\begin{center}
	\begin{small}
\begin{table}
	\label{funlist}
	\begin{tabular}{|l|l|l|}
		\hline
		\textbf{Figure} & \textbf{Expression} & \textbf{Range} \\				
		\hline
		\ref{niceconvex} (A) & $x^2$ & $-5.12 \leq x \leq 5.12$ \\
		\hline
		\ref{niceconvex} (B) & $(-5 + 24x - 16x^2) e^{-x}$ & $1.9 \leq x \leq 3.9$ \\
		\hline
		\ref{niceconvex} (C)& $-x^{\frac{2}{3}} - (1 - x^2)^{\frac{1}{3}}$ & $0.001 \leq x \leq 0.99$ \\
		\hline		
		\ref{niceconvex} (D)& $1.25 x^2 + 0.0625 x^4$ & $-5 \leq x \leq 10$ \\
		\hline	
		\ref{niceconvex} (E) & $x^8$ & $-2 \leq x \leq 2$ \\
		\hline		
		\ref{nsconvex} (A), \ref{discontinuous} (E)  &$ \frac{1}{1 - x} + \frac{1}{x}$ & $0.01 \leq x \leq 0.99$ \\
		\hline		
		\ref{nsconvex} (B) & $|0.5 - x|$ & $-2 \leq x \leq 2$ \\
		\hline	
		\ref{bconvex} (A)& $ x$ & $-3 \leq x \leq 3$ \\
		\hline
		\ref{bconvex} (B)& $0$ & $-3 \leq x \leq 3$ \\
		\hline		
		\ref{ncsu} (A)
		& $1 - \cos(x^5)$ & $-\pi \leq x \leq \pi$ \\
		\hline			
		\ref{ncsu} (B)&$-\sin(x)\sin^{20}(x^2/\pi)$&$0\leq x\leq\pi$\\
		\hline
		\ref{ncsu} (C)
		 & $
		\begin{cases} 
			(x - 2)^2 & \text{if } x < 3 \\
			2 \log(x - 2) + 1 & \text{otherwise} 
		\end{cases}$ & $0 \leq x \leq 6$ \\
		\hline		
		\ref{ncnsu} (A) & $\sqrt{|x|}$ & $-3 \leq x \leq 2$ \\
		\hline		
		 \ref{ncnsu} (B), \ref{discontinuous} (G)  & $\begin{cases} \frac{1}{2}|x - 5| & \text{if } |x-5| < 1 \\ 1 & \text{otherwise} \end{cases}$ & $0 \leq x \leq 10$ \\
		\hline				
	    \ref{ncsmA} (A)& $-\sum_{k=1}^{10} \cos[(2\pi k x)$ & $-0.5 \leq x \leq 0.5$ \\
		\hline
		\ref{ncsmA} (B)& $-\sum_{k=1}^{10} 4 \pi^2 k^2 \cos(2 \pi k x)$ & $-0.5 \leq x \leq 0.5$ \\
		\hline
		\ref{ncsmA} (C) & $\sum_{k=1}^{10} 2 \pi k \sin(2 \pi k x)$ & $-0.5 \leq x \leq 0.5$ \\
		\hline
		\ref{ncsmA} (D) & $-x^2 + x^4$ & $-2 \leq x \leq 2$ \\
		\hline
		\ref{ncsmA} (E)& $- (2 - 6x)^2 \sin(4 - 12x)$ & $0 \leq x \leq 1$ \\
		\hline
		\ref{ncsmA} (F)& $1 + \frac{x^2}{4000} - \cos(x)$ & $-600 \leq x \leq 600$ \\
		\hline
		\ref{ncsmB} (A) & $ x^2 \sin^2(\frac{1}{x})$ & $-3 \leq x \leq 2$ \\
		\hline
		\ref{ncsmB} (B) & $ \sin(x) + \sin(3.33333x)$ & $-2.7 \leq x \leq 7.5$ \\
		\hline
		\ref{ncsmB} (C) & $\sum_{j=1}^6 j \sin (j +(j+1) x)$
		& $-2.7 \leq x \leq 7.5$ \\
		\hline	
		\ref{ncsmB} (D) & $(-1.4 + 3x) \sin(18x)$ & $0 \leq x \leq 1.2$ \\
		\hline
		\ref{ncsmB} (E) & $ e^{-x^2} (-x - \sin(x))$ & $-10 \leq x \leq 10$ \\
		\hline
		\ref{ncsmB} (F) & $ 3 - 0.84x + \log(x) + \sin(x) + \sin(\frac{10x}{3})$ & $2.7 \leq x \leq 7.5$ \\
		\hline
		\ref{ncsmC} (A) & $-\sum_{k=1}^{6} k \cos((k+1)x+k)$ & $-10 \leq x \leq 10$ \\
		\hline
		\ref{ncsmC} (B) & $\sin(\frac{2x}{3}) + \sin(x)$ & $3.1 \leq x \leq 20.4$ \\
		\hline
		\ref{ncsmC} (C) & $-x \sin(x)$ & $0 \leq x \leq 10$ \\
		\hline
		\ref{ncsmC} (D) & $2 \cos(x) + \cos(2x)$ & $-\frac{\pi}{2} \leq x \leq 2\pi$ \\
		\hline
		\ref{ncsmC} (E) & $\cos(x)^3 + \sin(x)^3$ & $0 \leq x \leq 2\pi$ \\
		\hline
		\ref{ncsmC} (F) & $-e^{-x} \sin(2\pi x)$ & $0 \leq x \leq 4$ \\
		\hline
		\ref{ncsmD} (A) & $\frac{6 - 5x + x^2)}{1 + x^2}$ & $-5 \leq x \leq 5$ \\
		\hline
		\ref{ncsmD} (B)  & $e^{-x^2} (-x + \sin(x))$ & $-10 \leq x \leq 10$ \\
		\hline
		\ref{ncsmD} (C)  & $x \cos(2x) + x \sin(x)$ & $0 \leq x \leq 10$ \\
		\hline
		\ref{ncsmD} (D)  & $ e^{-3x} - \sin^3(x)$ & $0 \leq x \leq 20$ \\
		\hline
		\ref{ncsmD} (E)  & $-x \sin(\sqrt{|x|})$ & $-500 \leq x \leq 500$ \\
		\hline
		\ref{ncsmD} (F)  & $x^2 - \cos(10x)$ & $-3 \leq x \leq 3$ \\
		\hline
		\ref{ncsmD} (G)  & $\frac{x}{4} - x^2 + x^4$ & $-1.5 \leq x \leq 1.5$ \\
		\hline
		\ref{ncnsm} (A) & $ x^2 + \sin^2(\frac{1}{x})$ & $-2 \leq x \leq 3$ \\
		\hline
		\ref{ncnsm} (B)  & $|x|\Pi_{j=1}^5|x-(-1)^j j/10|^{1/2}$ & $-1 \leq x \leq 1$ \\
		\hline
		\ref{ncnsm} (C), \ref{ncbm} (D), \ref{discontinuous} (A)  & $\lfloor5 (\sin^2(2x) + \sin^2(5x))\rfloor$ & $0 \leq x \leq \pi$ \\
		\hline
		\ref{ncnsm} (D), \ref{discontinuous} (B)    & $ x + \frac{1}{5} \lfloor-5 x^2\rfloor$ & $0 \leq x \leq 2$ \\
		\hline
		\ref{ncnsm} (E), \ref{discontinuous} (D)   & $ \lfloor5 x^2\rfloor$ & $-1 \leq x \leq 2$ \\
		\hline
		\ref{ncnsm} (F), \ref{discontinuous} (F)   & $\begin{cases} 0 & \text{if } |x-5| < 1 \\ 1 & \text{otherwise} \end{cases}$ & $0 \leq x \leq 10$ \\
		\hline
		\ref{ncbm} (A) 
		 & $x - x^2 - 0.01 x^4$ & $-3 \leq x \leq 3$ \\
		\hline
		\ref{ncbm} (B)  & $-x - x^2$ & $-3 \leq x \leq 3$ \\
		\hline
		\ref{ncbm} (C)  & $ -x^2 - 0.01 x^)$ & $-3 \leq x \leq 3$ \\
		\hline
		\ref{ncbm} (E), \ref{discontinuous} (C)  & $-x + \frac{1}{5} \lfloor-5 x^2\rfloor$ & $0 \leq x \leq 2$ \\
		\hline
		\ref{ncbm} (F)  & $-|1 + x|$ & $-2 \leq x \leq 2$ \\
		\hline		
	\end{tabular}
	\caption{Unnormalized test functions}
\end{table}
	\end{small}
\end{center}
	
\section{Performance breakdown}
\label{breakdown}

This appendix evaluates our algorithm's performance on various function classes. The results in this appendix demonstrate that our algorithm's efficiency consistently surpasses competing algorithms across all classes, as shown by the values $\Pi_{100}$ and $N_S$.  Moreover, the performance, as measured by the values of $\Pi$ or $\Pi_{100}$, remains relatively stable, ranging from 89 to 100\% and 49 to 100\%, respectively.
Competing algorithms perform substantially differently across the different classes; for example, the values of $\Pi_{100}$ vary from less than 10\% to 100\%. 
The number of function evaluations used by our algorithm is reasonably independent of the function class being minimized, ranging between 61 and 440.  Larger values correspond to degenerate functions (e.g., linear) or discontinuous functions.

\subsection{Convex  functions}
\label{convexres}

Our algorithm and all competing  algorithms consistently find a global minimum for convex functions. In all instances, our algorithm requires fewer function evaluations than competing algorithms. Among these, the best is the Nelder Mead algorithm. Our algorithm performs exceptionally well on uniformly convex and non-smooth convex functions, significantly outperforming Nelder Mead. For degenerate functions, that is,  either strictly but not uniformly convex or linear functions, our algorithm is slightly better than Nelder Mead. Our algorithm also requires substantially fewer function evaluations than Random Search, Differential Evolution, and Simulated Annealing. 
The following sections provide a detailed analysis of 
 the algorithm's behavior in the convex case.  

\subsubsection{Smooth, strictly convex functions}

The algorithm performs 
 exceptionally well for smooth, strictly convex functions with a single interior minimum, such as those depicted in Figure \ref{niceconvex}. The corresponding numerical results are presented in Tables \ref{niceconvextable} (uniformly convex) and \ref{niceconvextableB} (strictly but not uniformly convex). 
In particular, for uniformly convex functions, our algorithm outperforms the best built-in algorithm for this class of functions (Nelder Mead) by a factor of 6. For strictly but not uniformly convex functions, the number of function evaluations is similar to Nelder Mead. It is worth noting that, unlike our algorithm,  the Nelder Mead algorithm is invariant under composition with monotone functions, so its performance is similar for both $x^2$ (Figure \ref{niceconvex} (A)) and $x^8$ (Figure \ref{niceconvex} (E)). As a result, it is not sensitive to high-order degeneracy of a minimum.

\begin{figure}
	\subfloat[]{\includegraphics[width = .25\textwidth]{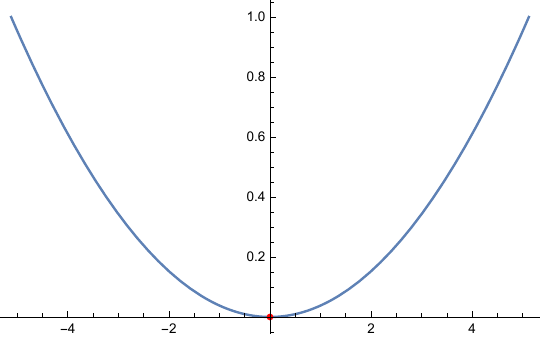}} 
	\hspace{20pt}
	\subfloat[]{\includegraphics[width = .25\textwidth]{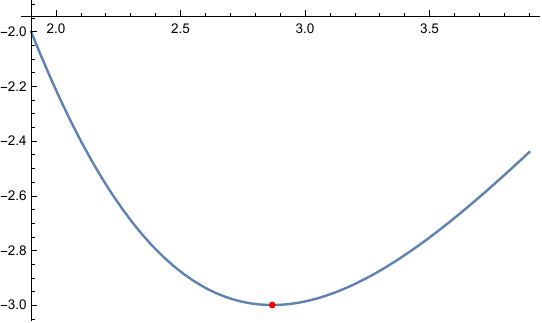}}
	\hspace{20pt}
	\subfloat[]{\includegraphics[width = .25 \textwidth]{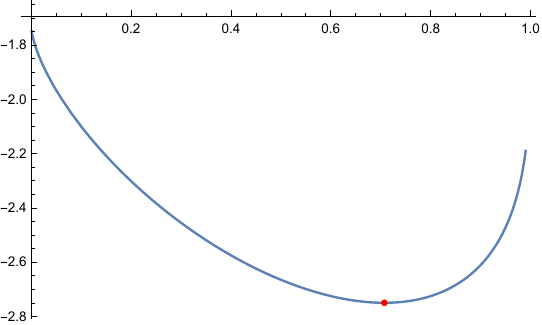}}\\
	\subfloat[]{\includegraphics[width =.25 \textwidth]{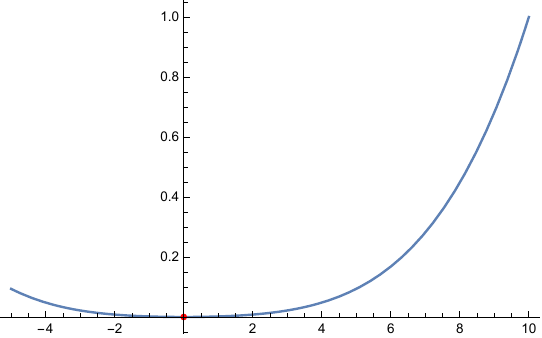}} 
	\hspace{20pt}
	\subfloat[]{\includegraphics[width = .25\textwidth]{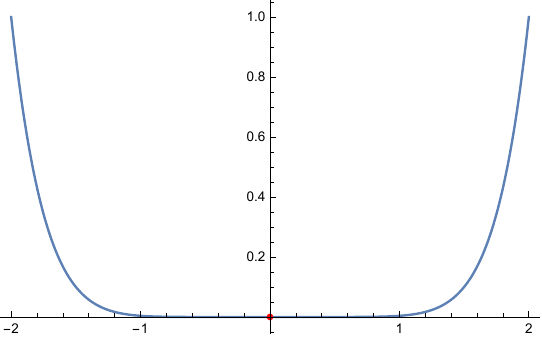}}
	\caption{Smooth strictly convex functions with interior minima. (A)-(D) are uniformly convex and (E) is strictly but not uniformly convex.}
	\label{niceconvex}
\end{figure}
\begin{TBL}
	{"uniformlyconvexresults",  (normalizeassociation/@combinedreport[{"uniformly convex","interior min", "unimodal", "smooth"}//selecttestfunctions]) //displayreportastable//First//#[[All, {1, 2, 4, 5, 6,7, 8}]]&}
\end{TBL}
\begin{center}
	\begin{small}
		\begin{table}
			\begin{tabular}{|c|c|c|c|c|c|c|c|}\hline Optimizer &\(N_f\)&\(\Pi\)&\(N_s\)&\(\Pi_{100}\)&\(\tau\)&\(\Delta\)&\(\Delta_c\)\\\hline
Nelder Mead&\(371.3\)&\(1.\)&\(371.3\)&\(1.\)&\(0.017\)&\(   
1.6\times 10^{-16}\)&\(   
1.6 \times 10^{-16}\)\\\hline
Random Search&\(705.4\)&\(1.\)&\(705.4\)&\(1.\)&\(0.022\)&\(  
3.1 \times 10^{-16}\)&\(  
3.1 \times 10^{-16}\)\\\hline
Differential Evolution&\(2129.2\)&\(1.\)&\(2129.2\)&\(1.\)&\(0.052\)&\( 
1.5 \times 10^{-16}\)&\(  
1.5 \times 10^{-16}\)\\\hline
Simulated Annealing&\(592.2\)&\(1.\)&\(592.2\)&\(1.\)&\(0.018\)&\(  
1.1 \times 10^{-16}\)&\(     
1.1 \times 10^{-16}\)\\\hline
Algorithm&\(62.1\)&\(1.\)&\(62.1\)&\(1.\)&\(0.058\)&\(     
1.1 \times 10^{-15}\)&\(      
1.1 \times 10^{-15}\)\\\hline\end{tabular}
			\medskip
			\caption{{Smooth, uniformly convex functions with interior minima, Figure \ref{niceconvex} (A)-(D).}}
			\label{niceconvextable}
		\end{table}		
	\end{small}
\end{center}

\begin{TBL}
	{"highpowerconvexresults",  (normalizeassociation/@combinedreport[{"High power"}//selecttestfunctions]) //displayreportastable//First//#[[All, {1, 2, 4, 5, 6,7, 8}]]&}
\end{TBL}
\begin{center}
	\begin{small}
		\begin{table}
						\begin{tabular}{|c|c|c|c|c|c|c|c|}\hline
Optimizer&\(N_f\)&\(\Pi\)&\(N_s\)&\(\Pi_{100}\)&\(\tau\)&\(\Delta\)&\(\Delta_c\)\\\hline
Nelder Mead&\(461.2\)&\(1.\)&\(461.2\)&\(1.\)&\(0.016\)&\(   
1.7 \times 10^{-15}\)&\(    
1.7 \times 10^{-15}\)\\\hline
Random Search&\(2994.1\)&\(1.\)&\(2994.1\)&\(1.\)&\(0.052\)&\(    
1.3 \times 10^{-41}\)&\(    
1.3 \times 10^{-41}\)\\\hline
Differential Evolution&\(2125.4\)&\(1.\)&\(2125.4\)&\(1.\)&\(0.049\)&\(   
3. \times 10^{-27}\)&\(    
3. \times 10^{-27}\)\\\hline
Simulated Annealing&\(3087.6\)&\(1.\)&\(3087.6\)&\(1.\)&\(0.031\)&\(     
3.8 \times 10^{-12}\)&\(   
3.8 \times 10^{-12}\)\\\hline
Algorithm&\(440.1\)&\(1.\)&\(440.1\)&\(1.\)&\(0.72\)&\(     
8.9 \times 10^{-10}\)&\(      
8.9 \times 10^{-10}\)\\\hline\end{tabular}
						
			\medskip			
			\caption{{Smooth, strictly convex but not uniformly functions with interior minima,  Figure \ref{niceconvex} (E).}}
			\label{niceconvextableB}
		\end{table}		
	\end{small}
\end{center}


\subsubsection{Non-smooth unimodal convex functions}

For non-smooth, convex functions with an interior minimum, we considered examples such as those depicted in Figure \ref{nsconvex}. This includes a function with a singularity at the boundary (Figure \ref{nsconvex} (A)) and one with an interior singularity (Figure \ref{nsconvex} (B)). The corresponding numerical results are presented in Table \ref{tnsconvex}.

In this case, our algorithm outperforms substantially the competing algorithms significantly. 
Our algorithm requires ten times fewer function evaluations than Nelder Mead,  the
best-performing among the competing algorithms.

\begin{figure}
	\subfloat[]{\includegraphics[width = .25\textwidth]{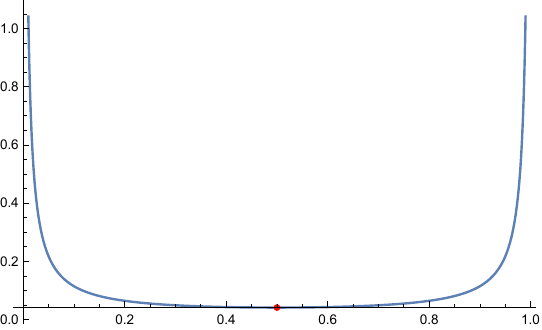}} 
	\hspace{20pt}
	\subfloat[]{\includegraphics[width = .25\textwidth]{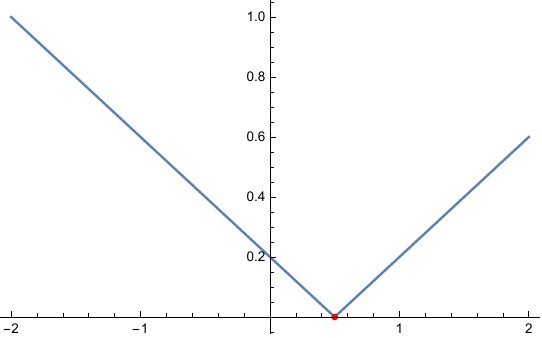}}
	\caption{Non-smooth, unimodal convex functions with interior minima.}
	\label{nsconvex}
\end{figure}
\begin{TBL}
	{"nsmoothconvexresults",normalizeassociation/@combinedreport[{ "non-smooth", "convex"}//selecttestfunctions] // displayreportastable//First//#[[All, {1, 2, 4, 5, 6, 7,8}]]&}
\end{TBL}
\begin{center}
	\begin{small}
		\begin{table}
			\begin{tabular}{|c|c|c|c|c|c|c|c|}\hline
Optimizer&\(N_f\)&\(\Pi\)&\(N_s\)&\(\Pi_{100}\)&\(\tau\)&\(\Delta\)&\(\Delta_c\)\\\hline
Nelder Mead&\(717.2\)&\(1.\)&\(717.2\)&\(1.\)&\(0.08\)&\(     
5.9 \times 10^{-19}\)&\(     
5.9 \times 10^{-19}\)\\\hline
Random Search&\(1835.8\)&\(1.\)&\(1835.8\)&\(1.\)&\(0.1\)&\(    
1.2 \times 10^{-18}\)&\(    
1.2 \times 10^{-18}\)\\\hline
Differential Evolution&\(3258.4\)&\(1.\)&\(3258.4\)&\(1.\)&\(0.17\)&\(    
1.5 \times 10^{-11}\)&\(   
1.5 \times 10^{-11}\)\\\hline
Simulated Annealing&\(2783.4\)&\(1.\)&\(2783.4\)&\(1.\)&\(0.28\)&\(   
1.1 \times 10^{-11}\)&\(     
1.1 \times 10^{-11}\)\\\hline
Algorithm&\(74.7\)&\(1.\)&\(74.7\)&\(1.\)&\(0.068\)&\(   
4.8 \times 10^{-8}\)&\(   
4.8 \times 10^{-8}\)\\\hline\end{tabular}
			
			\medskip
			\caption{Non-smooth unimodal convex functions with interior minima, Figure \ref{nsconvex}.}
			\label{tnsconvex}
		\end{table}		
	\end{small}
\end{center}			
			
			\subsubsection{Linear functions}
			
The algorithm performs well for linear convex functions with a boundary minimum. Our test cases include those depicted in Figure \ref{bconvex}. The corresponding numerical results are presented in Table \ref{tbconvex}. In this case, our algorithm requires twice as few function evaluations as Nelder Mead. A decrease in performance compared to the strictly convex case is expected because these functions are linear, resulting in no contraction in $\sigma$.
			
			\begin{figure}
				\subfloat[]{\includegraphics[width = .25\textwidth]{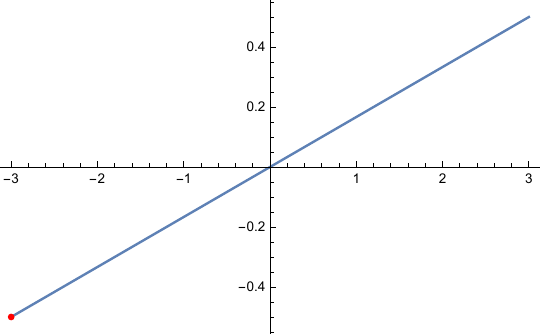}} 
				\hspace{20pt}
				\subfloat[]{\includegraphics[width = .25\textwidth]{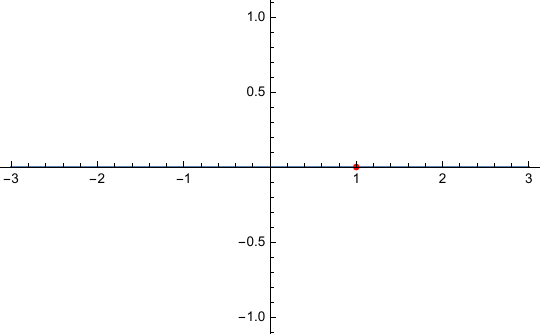}}
				\caption{Convex functions with a boundary minimum.}
				\label{bconvex}
			\end{figure}
			\begin{TBL}
				{"bmmoothconvexresults",normalizeassociation/@combinedreport[ {"boundary min", "convex"}//selecttestfunctions] // displayreportastable//First//#[[All, {1, 2, 4, 5, 6, 7,8}]]&}
			\end{TBL}
			\begin{center}
				\begin{small}
					\begin{table}
						\begin{tabular}{|c|c|c|c|c|c|c|c|}\hline
Optimizer&\(N_f\)&\(\Pi\)&\(N_s\)&\(\Pi_{100}\)&\(\tau\)&\(\Delta\)&\(\Delta_c\)\\\hline
Nelder Mead&\(642.4\)&\(1.\)&\(642.4\)&\(1.\)&\(0.026\)&\(    
5. \times 10^{-18}\)&\(  
5. \times 10^{-18}\)\\\hline
Random Search&\(3710.5\)&\(1.\)&\(3710.5\)&\(1.\)&\(0.073\)&\(0.\)&\(0.\)\\\hline
Differential Evolution&\(2555.2\)&\(1.\)&\(2555.2\)&\(1.\)&\(0.054\)&\( 
1.1 \times 10^{-18}\)&\(    
1.1 \times 10^{-18}\)\\\hline
Simulated Annealing&\(942.9\)&\(1.\)&\(942.9\)&\(1.\)&\(0.024\)&\(    
5. \times 10^{-18}\)&\(    
5. \times 10^{-18}\)\\\hline
Algorithm&\(296.2\)&\(1.\)&\(296.2\)&\(1.\)&\(0.52\)&\(0.\)&\(0.\)\\\hline\end{tabular}
						
						\medskip
						\caption{Convex functions with a boundary minimum, Figure \ref{bconvex}.}
						\label{tbconvex}
					\end{table}		
				\end{small}
			\end{center}
			

			\subsection{Non-convex unimodal functions}
			
In the case of non-convex, unimodal functions, our algorithm, Differential Evolution, and Simulated Annealing consistently locate the minimum. The other two algorithms find the minimum in nearly all instances. However, our algorithm consistently requires fewer function evaluations than the other algorithms, ranging between 14 and 63 fewer function evaluations.

			\subsubsection{Non-convex smooth unimodal functions}
			
For non-convex smooth unimodal functions, displayed in Figure \ref{ncsu}, the corresponding numerical results are presented in Table \ref{tncsu}. Our algorithm demonstrates exceptional performance, finding all minima while requiring few function evaluations. Random Search and Simulated Annealing locate almost all minima (note that the result $\Pi=1$ for Simulated Annealing is due to round-off, resulting in a $\Pi_{100}<1$ value), but they require 63 and 14 additional function evaluations, respectively. Nelder Mead and Differential Evolution successfully identify all minima; however, they require 34 and 58 more function evaluations, respectively.

			\begin{figure}
				\subfloat[]{\includegraphics[width = .25\textwidth]{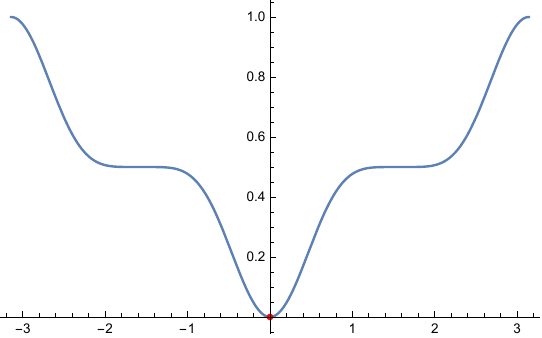}} 
				\hspace{20pt}
				\subfloat[]{\includegraphics[width = .25\textwidth]{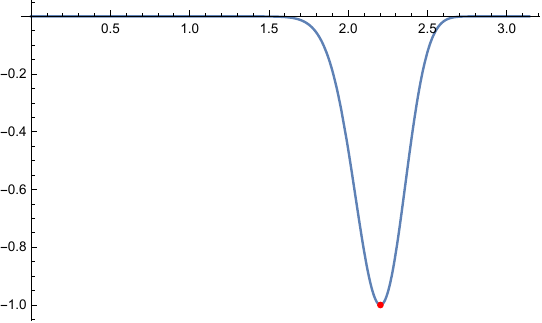}}
				\hspace{20pt}
				\subfloat[]{\includegraphics[width = .25 \textwidth]{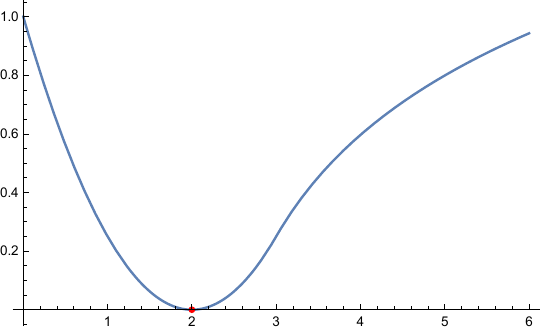}}
				\caption{Non-convex smooth unimodal}
				\label{ncsu}
			\end{figure}
			\begin{TBL}
				{"smoothunimodalresults",normalizeassociation/@combinedreport[{"unimodal", "non-convex", "interior min", "smooth"}//selecttestfunctions] // displayreportastable//First//#[[All, {1, 2, 4, 5, 6, 7,8}]]&}
			\end{TBL}
			\begin{center}
				\begin{small}
					\begin{table}
						\begin{tabular}{|c|c|c|c|c|c|c|c|}\hline
Optimizer&\(N_f\)&\(\Pi\)&\(N_s\)&\(\Pi_{100}\)&\(\tau\)&\(\Delta\)&\(\Delta_c\)\\\hline
Nelder Mead&\(2112.9\)&\(1.\)&\(2112.9\)&\(1.\)&\(0.24\)&\(  
6.2 \times 10^{-16}\)&\(    
6.2 \times 10^{-16}\)\\\hline
Random Search&\(3826.2\)&\(0.98\)&\(3904.3\)&\(0.1\)&\(0.4\)&\(0.017\)&\(     
2.2 \times 10^{-16}\)\\\hline
Differential Evolution&\(3508.9\)&\(1.\)&\(3508.9\)&\(1.\)&\(0.22\)&\(    
7.4 \times 10^{-16}\)&\(    
7.4 \times 10^{-16}\)\\\hline
Simulated Annealing&\(824.1\)&\(1.\)&\(826.8\)&\(0.5\)&\(0.02\)&\(0.0033\)&\(    
7.6\times 10^{-16}\)\\\hline
Algorithm&\(61.\)&\(1.\)&\(61.\)&\(1.\)&\(0.093\)&\(     
9.1 \times 10^{-16}\)&\(    
9.1 \times 10^{-16}\)\\\hline\end{tabular}
						
						\medskip
						\caption{Non-convex smooth unimodal, Figure \ref{ncsu}.}
						\label{tncsu}	
					\end{table}		
				\end{small}
			\end{center}				
					
					\subsubsection{Non-convex non-smooth unimodal functions }
					
We incorporate  non-convex, non-smooth unimodal functions into our test cases, as illustrated in Figure \ref{ncnsu}. The corresponding numerical results can be found in Table \ref{tncnsu}. Our algorithm successfully locates the minimum. Among the competing algorithms, only Differential Evolution and Simulated Annealing manage to find the minimum, but at the expense of 41 and 23 times more function evaluations, respectively.

					\begin{figure}
						\subfloat[]{\includegraphics[width = .25\textwidth]{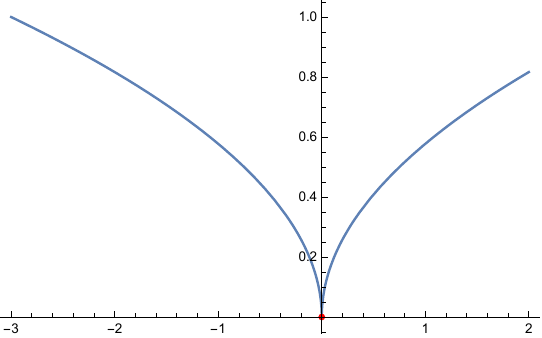}} 
						\hspace{20pt}
						\subfloat[]{\includegraphics[width = .25\textwidth]{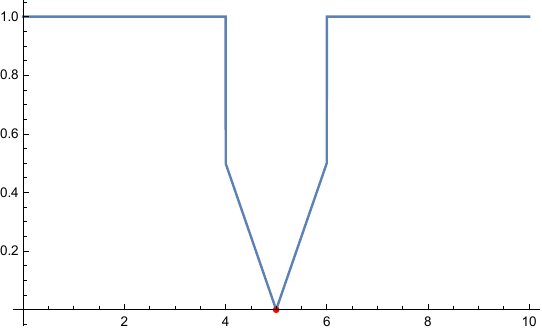}}
						\caption{Non-convex non-smooth unimodal.}
						\label{ncnsu}
					\end{figure}
					\begin{TBL}
						{"nonsmoothunimodalresults",normalizeassociation/@combinedreport[{"unimodal", "non-convex", "interior min", "non-smooth"}//selecttestfunctions] // displayreportastable//First//#[[All, {1, 2, 4, 5, 6, 7,8}]]&}
					\end{TBL}
					\begin{center}
						\begin{small}
							\begin{table}
								\begin{tabular}{|c|c|c|c|c|c|c|c|}\hline
Optimizer&\(N_f\)&\(\Pi\)&\(N_s\)&\(\Pi_{100}\)&\(\tau\)&\(\Delta\)&\(\Delta_c\)\\\hline
Nelder Mead&\(4430.2\)&\(0.69\)&\(6420.6\)&\(0.026\)&\(0.53\)&\(0.31\)&\(     
1.3 \times 10^{-10}\)\\\hline
Random Search&\(1853.9\)&\(0.88\)&\(2106.7\)&\(0.11\)&\(0.097\)&\(0.12\)&\(   
9.9\times 10^{-19}\)\\\hline
Differential Evolution&\(4723.5\)&\(1.\)&\(4723.5\)&\(1.\)&\(0.28\)&\(  
1.4 \times 10^{-7}\)&\(   
1.4 \times 10^{-7}\)\\\hline
Simulated Annealing&\(2662.5\)&\(1.\)&\(2662.5\)&\(1.\)&\(0.13\)&\(   
8. \times 10^{-8}\)&\(  
8. \times 10^{-8}\)\\\hline
Algorithm&\(114.7\)&\(1.\)&\(114.7\)&\(1.\)&\(0.11\)&\(0.000016\)&\(0.000016\)\\\hline\end{tabular}
								
								\medskip
								\caption{Non-convex non-smooth unimodal, Figure \ref{ncnsu}.}
								\label{tncnsu}				
							\end{table}		
						\end{small}
					\end{center}				

							\subsection{Non-convex multimodal functions}

Our algorithm exhibits the highest performance for both  non-convex smooth and non-convex  non-smooth multimodal functions, as can be seen in the $\Pi_{100}$ and $N_s$ of the columns in Table \ref{tncsm} and Table \ref{tncnsm}.
Although Differential Evolution discovers a larger fraction of minima, it does so at the expense of approximately 25 times more function evaluations. 
							
							\subsubsection{Non-convex smooth multimodal functions}

Regarding non-convex, smooth multimodal functions, we include examples of highly oscillatory functions (Figures \ref{ncsmA}(A)-(C)), bimodal functions  (Figures \ref{ncsmA}(D) and \ref{ncsmD}(G)), oscillatory perturbations of convex functions with high-frequency and low-amplitude (Figure \ref{ncsmA}(F)), low-frequency and larger amplitude (Figure \ref{ncsmD}(F)), and differentiable functions with infinitely many global minima (Figure \ref{ncsmB}(A)). These test functions are represented in Figures \ref{ncsmA}, \ref{ncsmB}, \ref{ncsmC}, and \ref{ncsmD}. The corresponding numerical results are displayed in Table \ref{tncsm}. While no algorithm finds all minima, the best algorithm in terms of probability of finding a minimum is Differential Evolution. However, Differential Evolution uses 25 times more function evaluations than our algorithm, which is reflected in the efficiency index $\Pi_{100}$  and $N_s$.
 Therefore,  based on these criteria the proposed algorithm shows the best performance.
							
							\begin{figure}
								\subfloat[]{\includegraphics[width = .25\textwidth]{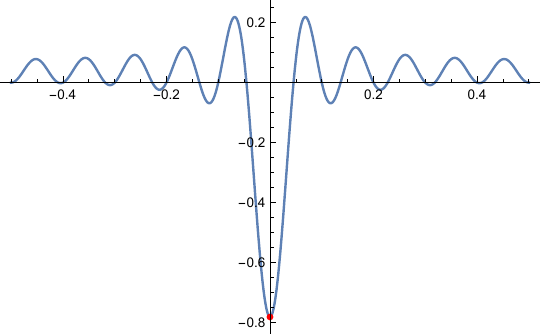}} 
								\hspace{20pt}
								\subfloat[]{\includegraphics[width = .25\textwidth]{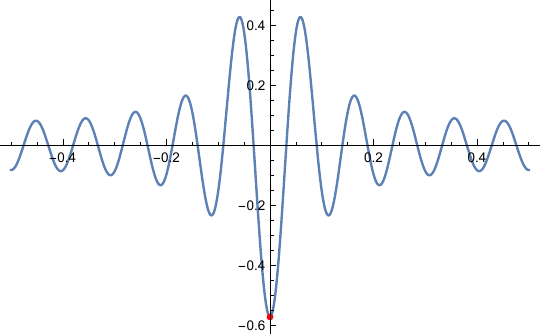}}
								\hspace{20pt}
								\subfloat[]{\includegraphics[width = .25 \textwidth]{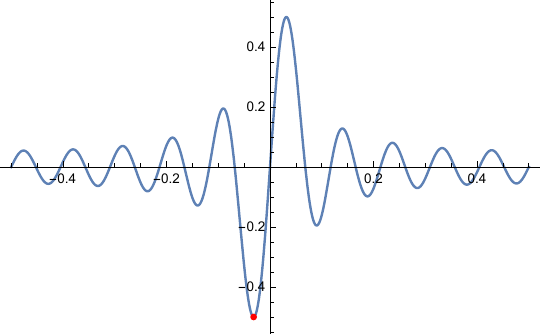}}\\
								\subfloat[]{\includegraphics[width = .25\textwidth]{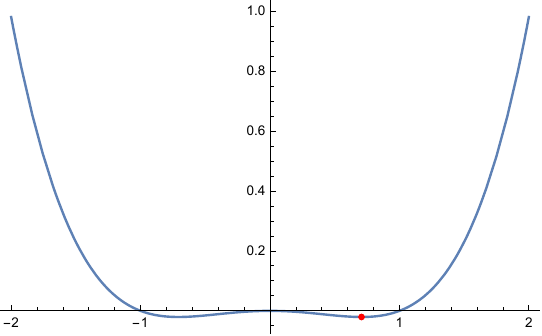}} 
								\hspace{20pt}
								\subfloat[]{\includegraphics[width = .25\textwidth]{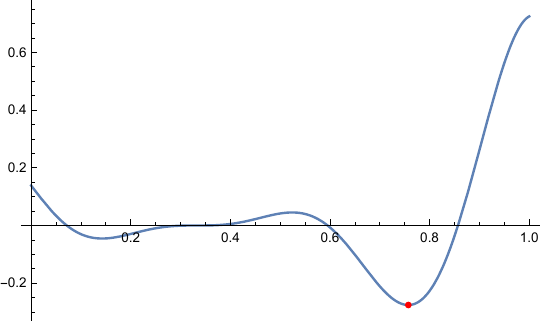}}
								\hspace{20pt}
								\subfloat[]{\includegraphics[width = .25 \textwidth]{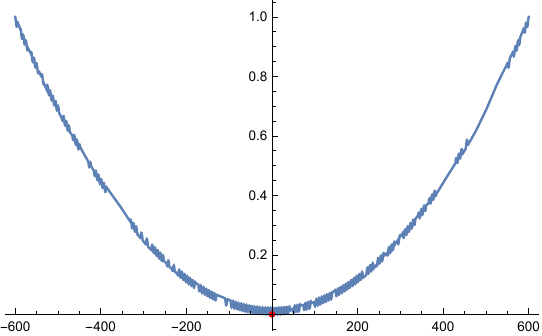}}
								\caption{Non-convex smooth multimodal - I.}
								\label{ncsmA}
							\end{figure}	
							\begin{figure}		
								\subfloat[]{\includegraphics[width = .25\textwidth]{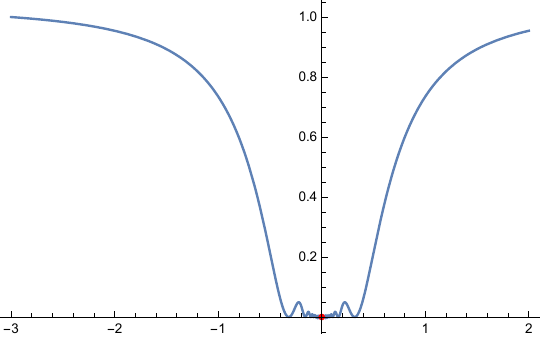}} 
								\hspace{20pt}
								\subfloat[]{\includegraphics[width = .25\textwidth]{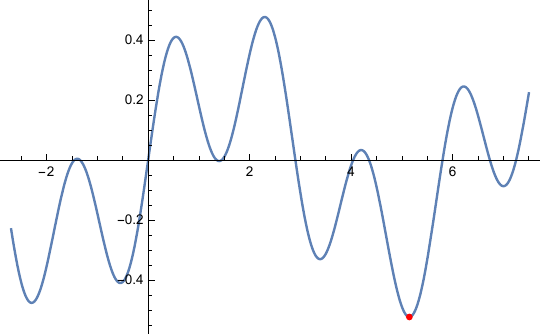}}
								\hspace{20pt}
								\subfloat[]{\includegraphics[width = .25 \textwidth]{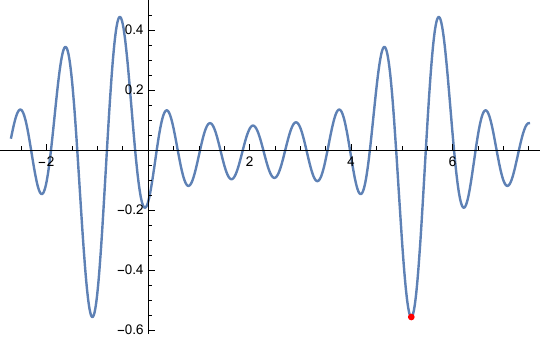}}\\
								\subfloat[]{\includegraphics[width = .25\textwidth]{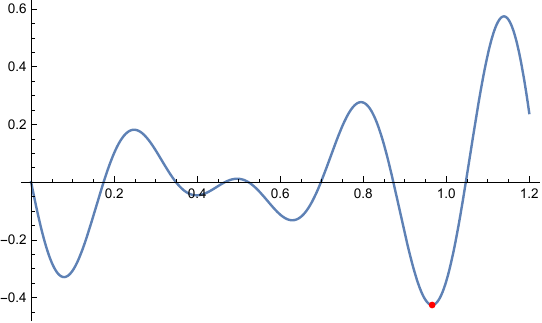}} 
								\hspace{20pt}
								\subfloat[]{\includegraphics[width = .25\textwidth]{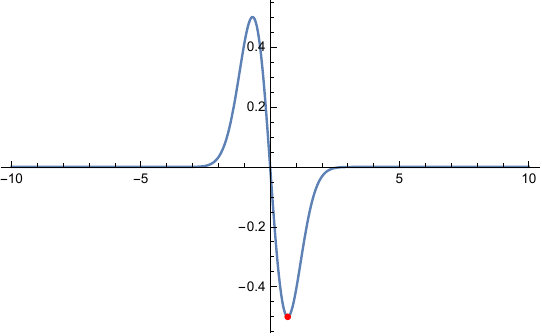}}
								\hspace{20pt}
								\subfloat[]{\includegraphics[width = .25 \textwidth]{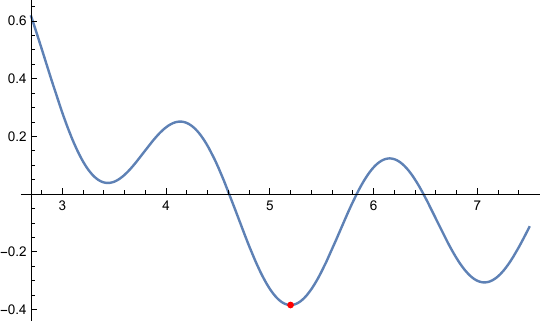}}
								\caption{Non-convex smooth multimodal - II.}
								\label{ncsmB}
							\end{figure}	
							\begin{figure}	
								\subfloat[]{\includegraphics[width = .25\textwidth]{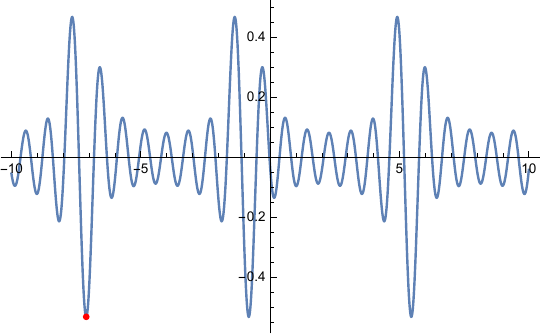}} 
								\hspace{20pt}
								\subfloat[]{\includegraphics[width = .25\textwidth]{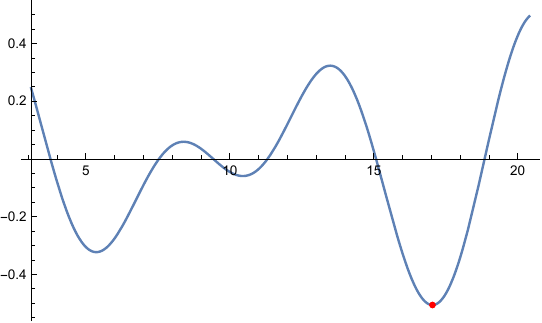}}
								\hspace{20pt}
								\subfloat[]{\includegraphics[width = .25 \textwidth]{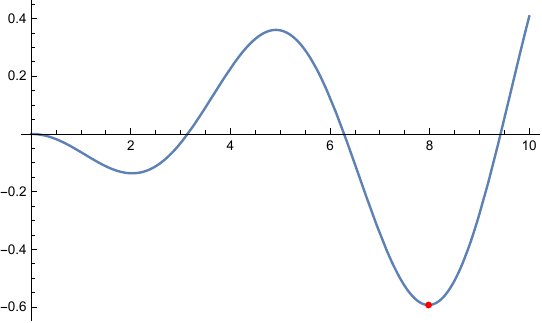}}\\
								\subfloat[]{\includegraphics[width = .25\textwidth]{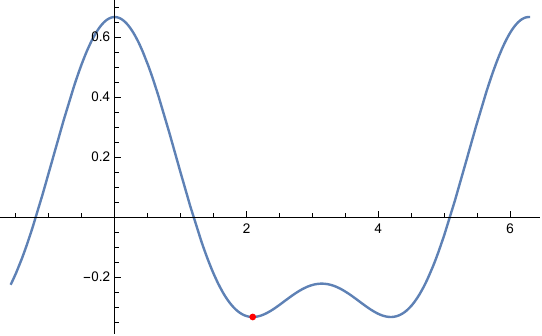}} 
								\hspace{20pt}
								\subfloat[]{\includegraphics[width = .25\textwidth]{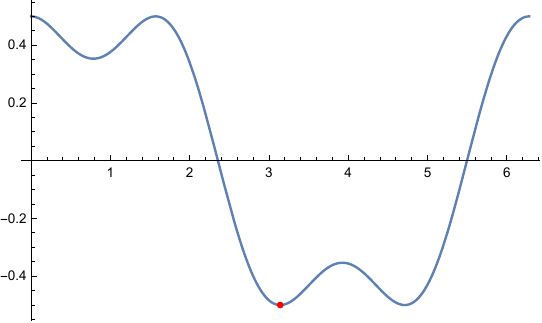}}
								\hspace{20pt}
								\subfloat[]{\includegraphics[width = .25 \textwidth]{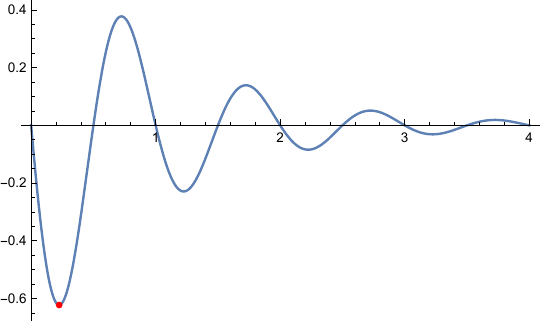}}
								\caption{Non-convex smooth multimodal - III.}
								\label{ncsmC}
							\end{figure}	
							\begin{figure}	
								\subfloat[]{\includegraphics[width = .25\textwidth]{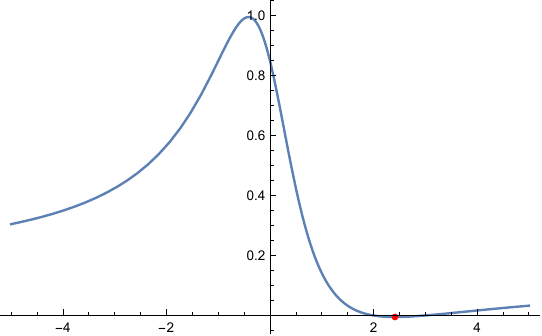}} 
								\hspace{20pt}
								\subfloat[]{\includegraphics[width = .25\textwidth]{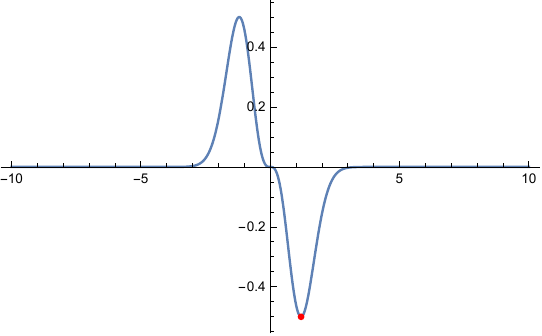}}
								\hspace{20pt}
								\subfloat[]{\includegraphics[width = .25 \textwidth]{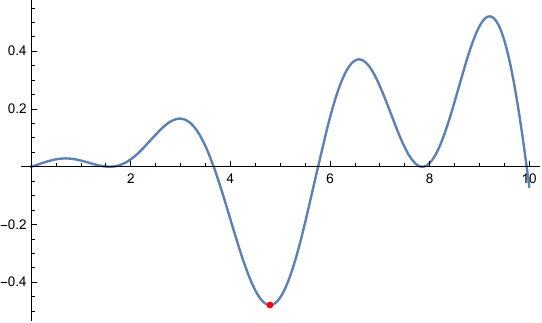}}\\
								\subfloat[]{\includegraphics[width = .25\textwidth]{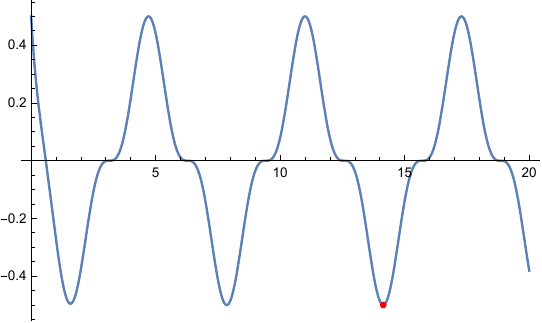}} 
								\hspace{20pt}
								\subfloat[]{\includegraphics[width = .25\textwidth]{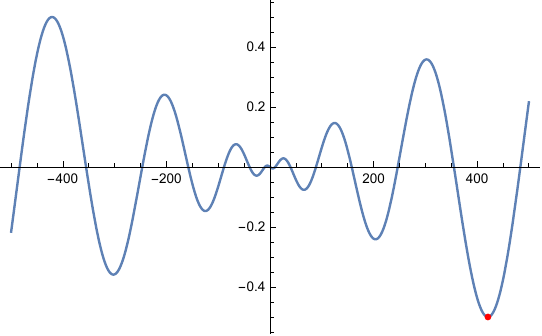}}\\
								\subfloat[]{\includegraphics[width = .25 \textwidth]{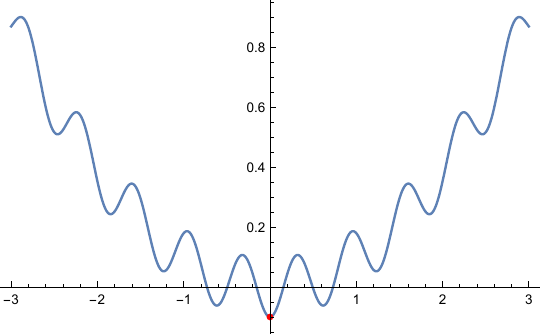}}
								\hspace{20pt}
								\subfloat[]{\includegraphics[width = .25\textwidth]{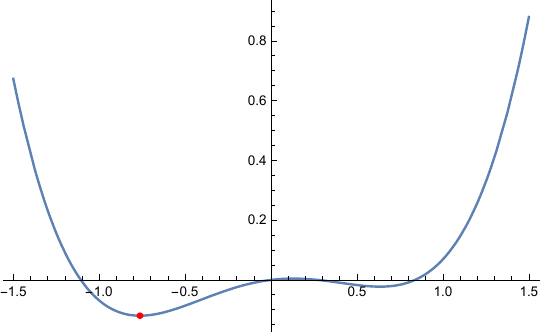}} 
								\caption{Non-convex smooth multimodal - IV.}
								\label{ncsmD}
							\end{figure}
							\begin{TBL}
								{"smoothmultimodalresults",normalizeassociation/@combinedreport[{"multimodal", "interior min", "non-convex", "smooth"}//selecttestfunctions] // displayreportastable//First//#[[All, {1, 2, 4, 5, 6,7, 8}]]&}
							\end{TBL}
							\begin{center}
								\begin{small}
									\begin{table}
										\begin{tabular}{|c|c|c|c|c|c|c|c|}\hline
Optimizer&\(N_f\)&\(\Pi\)&\(N_s\)&\(\Pi_{100}\)&\(\tau\)&\(\Delta\)&\(\Delta_c\)\\\hline
Nelder Mead&\(547.9\)&\(0.47\)&\(1165.7\)&\(0.11\)&\(0.035\)&\(0.18\)&\(0.000015\)\\\hline
Random Search&\(1422.3\)&\(0.84\)&\(1693.2\)&\(0.12\)&\(0.042\)&\(0.042\)&\(0.000016\)\\\hline
Differential Evolution&\(2578.6\)&\(0.93\)&\(2772.7\)&\(0.098\)&\(0.086\)&\(0.011\)&$5.4 \times 10^{-6}$
\\\hline
Simulated Annealing&\(855.9\)&\(0.81\)&\(1056.7\)&\(0.18\)&\(0.026\)&\(0.037\)&\(0.000015\)\\\hline
Algorithm&\(102.2\)&\(0.89\)&\(114.8\)&\(0.89\)&\(0.13\)&\(0.025\)&\(0.000012\)\\\hline
\end{tabular}
										
										\medskip
										\caption{Non-convex smooth multimodal, Figures \ref{ncsmA}-\ref{ncsmD}.}
										\label{tncsm}				
									\end{table}		
								\end{small}
							\end{center}
									%
									
										\subsubsection{Non-convex non-smooth multimodal functions}

Regarding non-convex, non-smooth, multimodal functions, we include examples with an infinite number of local minima (Figure \ref{ncnsm}(A)), cusps  (Figure \ref{ncnsm}(B)), and various examples of discontinuous multimodal functions  (Figures \ref{ncnsm}(C)-(F)). The corresponding numerical results are presented in Table \ref{tncnsm}. Similar to the case of non-convex smooth multimodal functions, Differential Evolution is the algorithm that finds the highest fraction of minima but at the cost of using 26 times more function evaluations. Consequently, our algorithm  has the best performance also for this class of functions.

									\begin{figure}
										\subfloat[]{\includegraphics[width = .25\textwidth]{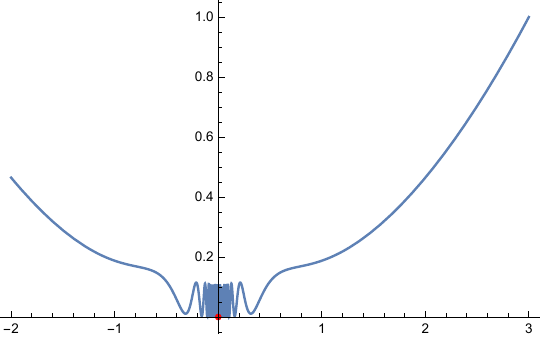}} 
										\hspace{20pt}
										\subfloat[]{\includegraphics[width = .25\textwidth]{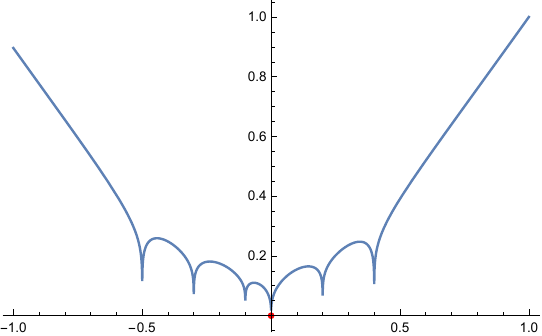}}
										\hspace{20pt}
										\subfloat[]{\includegraphics[width = .25 \textwidth]{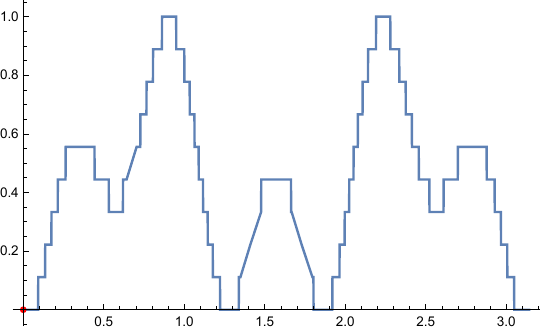}}\\
										\subfloat[]{\includegraphics[width = .25\textwidth]{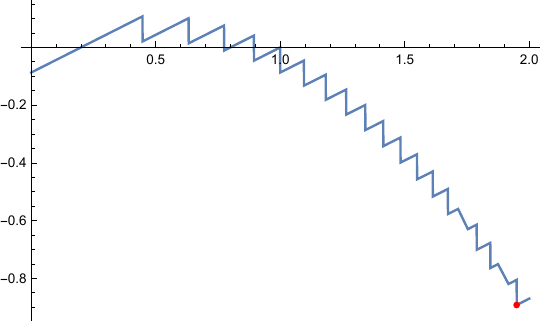}} 
										\hspace{20pt}
										\subfloat[]{\includegraphics[width = .25\textwidth]{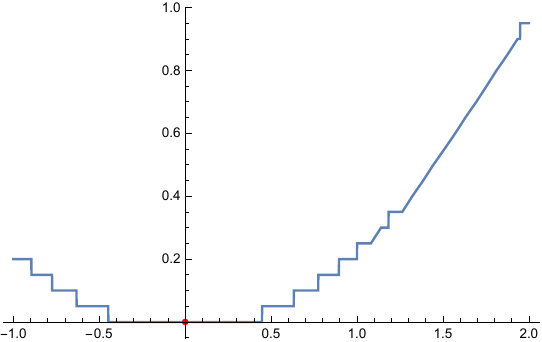}}
										\hspace{20pt}
										\subfloat[]{\includegraphics[width = .25 \textwidth]{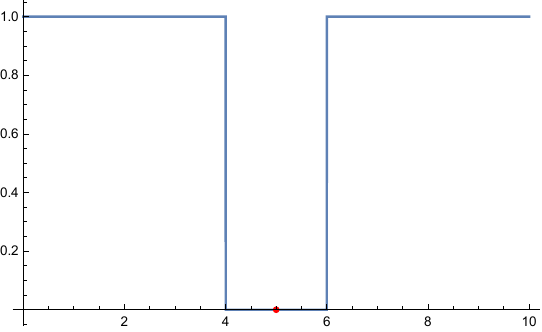}}\\
										\caption{Non-convex non-smooth multimodal.}
										\label{ncnsm}
									\end{figure}
									\begin{TBL}
										{"nsmoothmultimodalresults",normalizeassociation/@combinedreport[{"multimodal", "interior min", "non-convex", "non-smooth"}//selecttestfunctions] // displayreportastable//First//#[[All, {1, 2, 4, 5, 6,7, 8}]]&}
									\end{TBL}
									\begin{center}
										\begin{small}
											\begin{table}
												\begin{tabular}{|c|c|c|c|c|c|c|c|}\hline
Optimizer&\(N_f\)&\(\Pi\)&\(N_s\)&\(\Pi_{100}\)&\(\tau\)&\(\Delta\)&\(\Delta_c\)\\\hline
Nelder Mead &\(7173.3\)&\(0.68\)&\(10549.0\)&\(0.016\)&\(0.46\)&\(0.14\)&\(0.000078\)\\\hline
Random Search&\(13648.3\)&\(0.87\)&\(15687.7\)&\(0.015\)&\(0.71\)&\(0.047\)&\(0.00003\)\\\hline
Differential Evolution&\(10243.8\)&\(0.98\)&\(10452.8\)&\(0.037\)&\(0.64\)&\(0.00012\)&\(0.000048\)\\\hline
Simulated Annealing&\(8748.4\)&\(0.77\)&\(11361.6\)&\(0.017\)&\(0.39\)&\(0.017\)&\(0.000085\)\\\hline
Algorithm&\(386.6\)&\(0.92\)&\(420.2\)&\(0.49\)&\(1.7\)&\(0.0016\)&\(0.000063\)\\\hline
\end{tabular}
												
												\medskip
												\caption{Non-convex non-smooth multimodal, Figure \ref{ncnsm}.}
												\label{tncnsm}				
											\end{table}		
										\end{small}
									\end{center}

											\subsection{Non-convex boundary minima}
											
We consider several examples of non-convex functions with boundary minima, including concave functions (Figures \ref{ncbm}(A)-(C)), discontinuous functions  (Figures \ref{ncbm}(D)-(E)), and a concave piecewise linear function  (Figure \ref{ncbm}(F)). The corresponding numerical results are displayed in Table \ref{tncbm}. While none of the algorithms achieves to find the minimum in all runs (the values $\Pi=1$ for Random Search and Differential Evolution are due to round-off), our algorithm exhibits the highest efficiency index  $\Pi_{100}$ and the smallest value $N_s$. Random Search and Differential Evolution, the best algorithms in terms of $\Pi$, require 47 and 23 times more function evaluations, respectively.

											\begin{figure}
												\subfloat[]{\includegraphics[width = .25\textwidth]{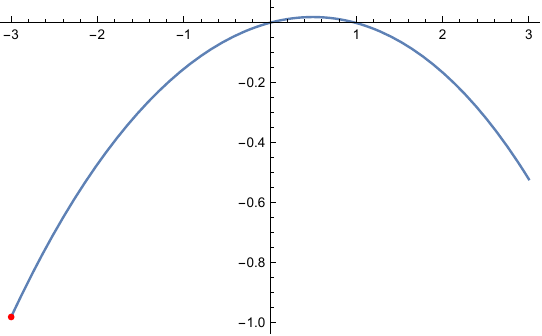}} 
												\hspace{20pt}
												\subfloat[]{\includegraphics[width = .25\textwidth]{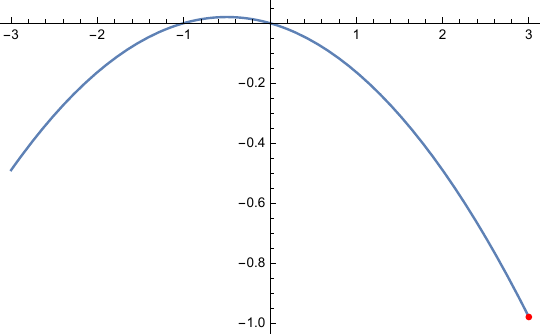}}
												\hspace{20pt}
												\subfloat[]{\includegraphics[width = .25 \textwidth]{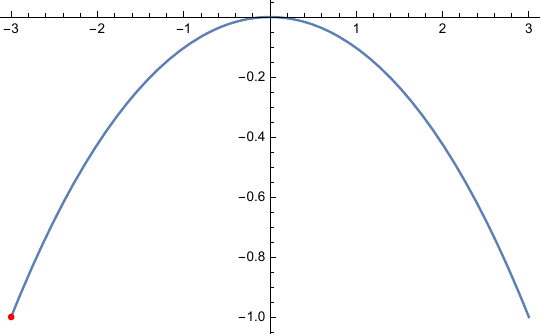}}\\
												\subfloat[]{\includegraphics[width = .25\textwidth]{figures/Quantum.pdf}} 
												\hspace{20pt}
												\subfloat[]{\includegraphics[width = .25\textwidth]{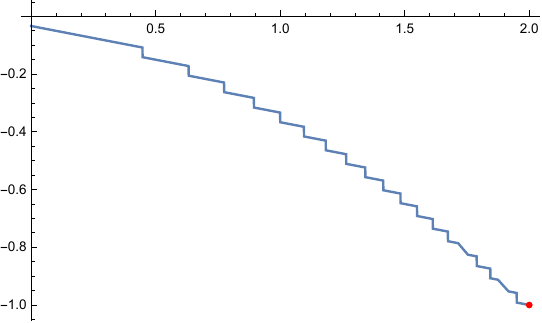}}
												\hspace{20pt}
												\subfloat[]{\includegraphics[width = .25 \textwidth]{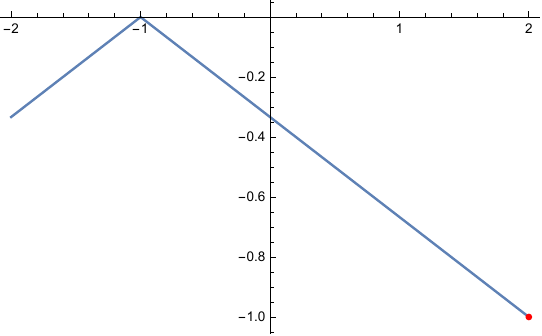}}
												\caption{Non-convex boundary minimum.}
												\label{ncbm}
											\end{figure}
											\begin{TBL}
												{"boundaryresults",normalizeassociation/@combinedreport[{"non-convex", "boundary min"}//selecttestfunctions] // displayreportastable//First//#[[All, {1, 2, 4, 5, 6,7, 8}]]&}
											\end{TBL}
											\begin{center}
												\begin{small}
													\begin{table}
														\begin{tabular}{|c|c|c|c|c|c|c|c|}\hline
Optimizer&\(N_f\)&\(\Pi\)&\(N_s\)&\(\Pi_{100}\)&\(\tau\)&\(\Delta\)&\(\Delta_c\)\\\hline
Nelder Mead&\(2072.1\)&\(0.85\)&\(2437.8\)&\(0.087\)&\(0.045\)&\(0.078\)&\(     
4.6\times 10^{-19}\)\\\hline
Random Search&\(8967.\)&\(1.\)&\(8967.\)&\(0.069\)&\(0.17\)&\(0.00019\)&\(0.\)\\\hline
Differential Evolution&\(4303.7\)&\(1.\)&\(4303.7\)&\(0.14\)&\(0.083\)&\(0.00077\)&\(   
8.5\times 10^{-18}\)\\\hline
Simulated Annealing&\(3268.\)&\(0.95\)&\(3440.\)&\(0.087\)&\(0.064\)&\(0.025\)&\(0.00011\)\\\hline
Algorithm&\(190.7\)&\(0.98\)&\(194.6\)&\(0.86\)&\(0.28\)&\(0.01\)&\(0.\)\\\hline\end{tabular}
														
														\medskip
														\caption{Non-convex boundary minimum, Figure \ref{ncbm}.}
														\label{tncbm}
													\end{table}		
												\end{small}
											\end{center}

													\subsection{Discontinuous functions}
													\label{discontinuousres}
Lastly, we test our algorithm on discontinuous functions (Figure \ref{discontinuous}).
This class of functions includes piecewise constant functions (Figure \ref{discontinuous} (A), (D), and (F)), a function whose
gradient is almost everywhere negative, but the minimum is situated near the right side of the boundary (Figure \ref{discontinuous} (B)), and a function with
a singularity at the boundary (Figure \ref{discontinuous} (E))
The numerical results are presented in Table \ref{tdiscontinuous}. Although Differential Evolution locates all minima, our algorithm is a close second, finding 98\% of the minima while utilizing 24 times fewer function evaluations.

													\begin{figure}
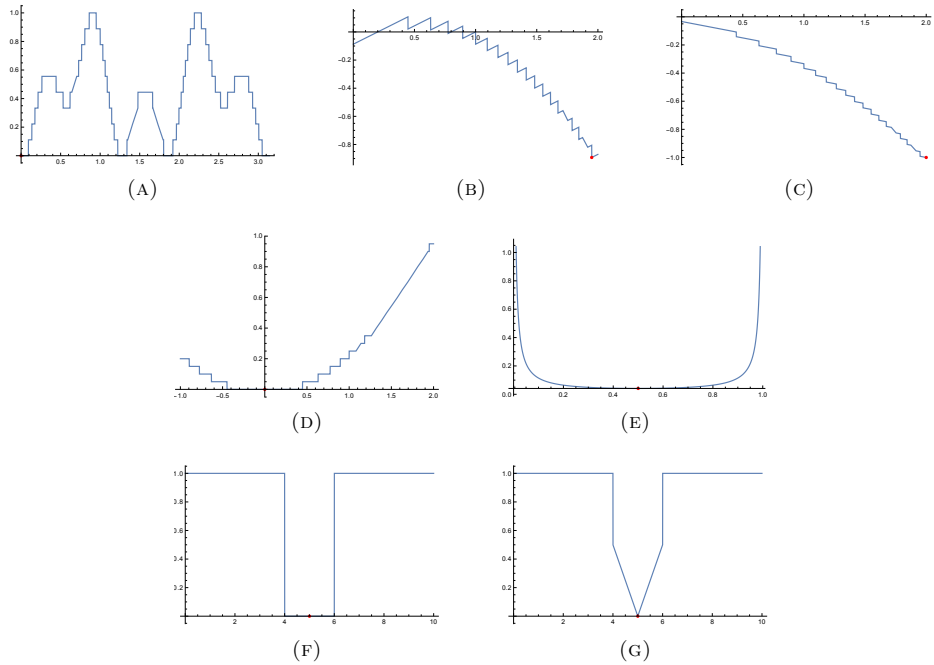

														\subfloat[]{\includegraphics[width = .25\textwidth]{figures/Quantum.pdf}} 
														\hspace{20pt}
														\subfloat[]{\includegraphics[width = .25\textwidth]{figures/QuantumconcaveA.pdf}}
														\hspace{20pt}
														\subfloat[]{\includegraphics[width = .25 \textwidth]{figures/QuantumconcaveB.pdf}}\\
														\subfloat[]{\includegraphics[width = .25\textwidth]{figures/Quantumconvex.pdf}} 
														\hspace{20pt}
														\subfloat[]{\includegraphics[width = .25\textwidth]{figures/Singular.pdf}}\\
														\subfloat[]{\includegraphics[width = .25 \textwidth]{figures/Step.pdf}}
														\hspace{20pt}
														\subfloat[]{\includegraphics[width = .25\textwidth]{figures/VStep.pdf}} 
														\caption{Discontinuous functions.}
														\label{discontinuous}
													\end{figure}
													\begin{TBL}
														{"discontinuousresults", normalizeassociation/@combinedreport["discontinuous"//selecttestfunctions] // displayreportastable//First//#[[All, {1, 2, 4, 5, 6,7, 8}]]&}
													\end{TBL}
													\begin{center}
														\begin{small}
															\begin{table}
																\begin{tabular}{|c|c|c|c|c|c|c|c|}\hline
Optimizer&\(N_f\)&\(\Pi\)&\(N_s\)&\(\Pi_{100}\)&\(\tau\)&\(\Delta\)&\(\Delta_c\)\\\hline
Nelder Mead&\(4779.3\)&\(0.76\)&\(6288.6\)&\(0.03\)&\(0.17\)&\(0.21\)&\(   
3.5 \times 10^{-10}\)\\\hline
Random Search&\(12562.3\)&\(0.86\)&\(14607.3\)&\(0.015\)&\(0.32\)&\(0.073\)&\(    
4. \times 10^{-10}\)\\\hline
Differential Evolution&\(7157.4\)&\(1.\)&\(7157.4\)&\(1.\)&\(0.25\)&\(    
2.8 \times 10^{-10}\)&\(     
2.8\times 10^{-10}\)\\\hline
Simulated Annealing&\(7111.1\)&\(0.95\)&\(7485.4\)&\(0.042\)&\(0.23\)&\(0.013\)&\(0.000095\)\\\hline
Algorithm&\(295.9\)&\(0.98\)&\(301.9\)&\(0.73\)&\(0.54\)&\(0.0013\)&\(     
2.6 \times 10^{-6}\)\\\hline\end{tabular}
																
																\medskip
																\caption{Discontinuous functions, Figure \ref{discontinuous}.}
																\label{tdiscontinuous}
															\end{table}		
														\end{small}
													\end{center}

\bibliographystyle{alpha}
\bibliography{optimization}

\end{document}